\documentclass[preprint]{imsart} 

\RequirePackage[OT1]{fontenc}
\usepackage[table]{xcolor}
\RequirePackage{amsmath,amssymb,amsthm,mathrsfs, enumerate,bm,makeidx,xcolor,multirow,pbox,tikz}
\RequirePackage{showidx}
\RequirePackage[colorlinks,linkcolor=red,citecolor=blue,urlcolor=blue]{hyperref}
\RequirePackage{natbib}
\usepackage{color}
\usepackage{bbm}
\usepackage{mathtools}
\usepackage{enumitem}
\usepackage{versions} %

\RequirePackage{xr}

\numberwithin{equation}{section}
\newcommand{\iid}{\stackrel{\text{iid}}{\sim}}
\newcommand{\mc}[1]{\mathcal{#1}}
\newcommand{\wh}[1]{\widehat{#1}}

\renewcommand{\epsilon}{\varepsilon}

\newcommand{\one}{\mathbbm{1}}

\DeclareMathOperator*{\argmax}{argmax}

\newcommand{\beq}{\begin{equation}}
  \newcommand{\eeq}{\end{equation}}
\newcommand{\beqa}{\begin{equation} \begin{aligned}}
    \newcommand{\eeqa}{\end{aligned} \end{equation}}
\newcommand{\beqas}{\begin{equation*} \begin{aligned}}
    \newcommand{\eeqas}{\end{aligned} \end{equation*}}

\newcommand{\bit}{\begin{itemize}}
  \newcommand{\eit}{\end{itemize}}
\newcommand{\bmat}{\begin{bmatrix}}
  \newcommand{\emat}{\end{bmatrix}}

\makeindex

\newenvironment{longform}{\color{blue}}{}
\newenvironment{notes}{\color{blue} {\bf Note:}}{}

\excludeversion{mylongform} %

\excludeversion{mynotes} %

\theoremstyle{definition}\newtheorem{problem}{Problem}[section]
\theoremstyle{definition}\newtheorem{definition}[problem]{Definition}
\theoremstyle{definition}
\theoremstyle{remark}
\theoremstyle{remark}
\theoremstyle{remark}\newtheorem{remark}[problem]{Remark}
\theoremstyle{definition}
\theoremstyle{plain}\newtheorem{theorem}[problem]{Theorem}
\theoremstyle{plain}\newtheorem{lemma}[problem]{Lemma}
\theoremstyle{plain}\newtheorem{proposition}[problem]{Proposition}
\theoremstyle{plain}\newtheorem{corollary}[problem]{Corollary}
\theoremstyle{plain}

\global\long\def\UUn{\mathbb{U}_n}

\newcommand{\AAA}{{\mathbb A}}

\newcommand{\BB}{{\mathbb B}}
\newcommand{\DD}{{\mathbb D}}
\newcommand{\FF}{ {\mathbb F}}
\newcommand{\GG}{{\mathbb G}}

\newcommand{\PP}{{\mathbb P}}
\newcommand{\RR}{{\mathbb R}}
\newcommand{\UU}{{\mathbb U}}
\newcommand{\YY}{{\mathbb Y}}
\newcommand{\XX}{{\mathbb X}}

\DeclareMathOperator{\interior}{int}
\newcommand{\lv}{\left\vert} %
  \newcommand{\rv}{\right\vert}
\newcommand{\lp}{\left(} %
  \newcommand{\rp}{\right)}
\newcommand{\lb}{\left\{} %
  \newcommand{\rb}{\right\}}

\global\long\def\inv#1{\frac{1}{#1}}
\global\long\def\tauplusa{\tau_{+}^{0}}
\global\long\def\tauminusa{\tau_{-}^{0}}
\global\long\def\FFLct{F_{L}}
\global\long\def\FFRct{F_{R}}
\global\long\def\FFLt{F_{L}} %
\global\long\def\FFRt{F_{R}}

\global\long\def\FFLti#1{F_{L,#1}} %
\global\long\def\FFRti#1{F_{R,#1}}
\global\long\def\XXLct{X_{L}}
\global\long\def\XXRct{X_{R}}
\global\long\def\XXRt{X_{R}}
\global\long\def\XXLt{X_{L}}

\global\long\def\HHLt{H_{L}}
\global\long\def\HHRt{H_{R}}
\global\long\def\HHsp{H^{+}}
\global\long\def\HHst{\HHsp}
\global\long\def\HHLti#1{H_{L,#1}}
\global\long\def\HHRti#1{H_{R,#1}}

\global\long\def\YYLt{Y_{L}}
\global\long\def\YYRt{Y_{R}}
\global\long\def\YYsp{Y^+}
\global\long\def\TauL{\tau_{L}} %
\global\long\def\TauR{\tau_{R}} %
\global\long\def\TausR{\tau_{R}^+} %
\global\long\def\Tausp{\tau_{+}^+} %

\global\long\def\Sa{\widehat{S}^0}

\global\long\def\vva{\widehat{\varphi}^0}
\global\long\def\psia{\widehat{\psi}^0}
\global\long\def\vv{\widehat{\varphi}}
\global\long\def\ffi#1{\varphi_{#1}} %
\global\long\def\FFi#1{F_{#1}} %
\global\long\def\ani#1{a_n^{#1}}
\global\long\def\bni#1{b_n^{#1}}
\global\long\def\ffol{h_0} %

\global\long\def\vp{\varphi}
\global\long\def\vvo{\vp_{0}}
\global\long\def\ffo{f_{0}} %
\global\long\def\FFo{F_{0}}
\global\long\def\vvna{\widehat{\vp}_{\xn}^{0}}
\global\long\def\psina{\widehat{\psi}_{\xn}^{0}}

\global\long\def\xn{n}

\global\long\def\xn{n}
\global\long\def\mm{m}
\global\long\def\An{A_{\xn}}
\global\long\def\Bn{B_{\xn}}

\global\long\def\AnL{A_{\xn,L}}
\global\long\def\AnR{A_{\xn,R}}

\global\long\def\snR{s_{\xn,R}} %
\global\long\def\snL{s_{\xn,L}} %

\global\long\def\RRna{R_{\xn}^0}

\global\long\def\RRnaR{R_{\xn}^0}

\global\long\def\RRnaL{R_{\xn}^0}

\global\long\def\RRno{R_{\xn}}

\global\long\def\HHnpsisR{\widehat{H}_{\xn,R}^{\psi}}
\global\long\def\YYnpsisR{\mathbb{Y}_{\xn,R}^{\psi}}

\global\long\def\YYnffo{\mathbb{Y}_{\xn}^{f}}
\global\long\def\HHnffo{\widehat{H}_{\xn}^{f}}
\global\long\def\YYnvvo{\mathbb{Y}_{\xn}^{\varphi}}
\global\long\def\HHnffaL{\widehat{H}_{\xn,L}^{f}}
\global\long\def\HHnffaR{\widehat{H}_{\xn,R}^{f}}
\global\long\def\HHnvvo{\widehat{H}_{\xn}^{\varphi}}
\global\long\def\HHnvvaL{\widehat{H}_{\xn,L}^{\varphi}}
\global\long\def\HHnvvaR{\widehat{H}_{\xn,R}^{\varphi}}
\global\long\def\YYnffaR{\mathbb{Y}_{\xn,R}^{f}}
\global\long\def\YYnffaL{\mathbb{Y}_{\xn,L}^{f}}
\global\long\def\YYnvvaR{\mathbb{Y}_{\xn,R}^{\varphi}}
\global\long\def\YYnvvaL{\mathbb{Y}_{\xn,L}^{\varphi}}
\global\long\def\XXnffo{\mathbb{X}_{\xn}^{f}}
\global\long\def\XXnvvo{\mathbb{X}_{\xn}^{\varphi}}
\global\long\def\XXnffaL{\mathbb{X}_{\xn,L}^{f}}
\global\long\def\XXnffaR{\mathbb{X}_{\xn,R}^{f}}
\global\long\def\XXnvvaL{\mathbb{X}_{\xn,L}^{\varphi}}
\global\long\def\XXnvvaR{\mathbb{X}_{\xn,R}^{\varphi}}

\global\long\def\IntUCo{\int_{\mm}^{\tnb}} %
\global\long\def\IntUCi{\int_{\mm}^{v}} %
\global\long\def\IntLo{\int^{\snL}_{\tnb}} %
\global\long\def\IntLi{\int^{\snL}_{v}} %
\global\long\def\IntRo{\int_{\snR}^{\tnb}} %
\global\long\def\IntRi{\int_{\snR}^{v}} %
\global\long\def\Fn{\mathbb{F}_{\xn}}
\global\long\def\FnL{\mathbb{F}_{\xn,L}}
\global\long\def\FnR{\mathbb{F}_{\xn,R}}
\global\long\def\ffna{\widehat{f}_{n}^{0}}
\global\long\def\ggns{\widehat{g}_{n}^{+}}
\global\long\def\ggna{\widehat{g}_{n}^{0}}
\global\long\def\psins{\widehat{\psi}_{n}^{+}}
\global\long\def\FFna{\widehat{F}_{n}^{0}}
\global\long\def\GGna{\widehat{G}_{n}^{0}}
\global\long\def\FFnaL{\widehat{F}_{\xn,L}^{0}}
\global\long\def\FFnaR{\widehat{F}_{\xn,R}^{0}}
\global\long\def\GGnsR{\widehat{G}_{\xn,R}^{+}}
\global\long\def\FnsR{\mathbb{F}_{\xn,R}^{+}}
\global\long\def\vvna{\widehat{\vp}_{\xn}^{0}}

\global\long\def\ffn{\widehat{f}_{\xn}}
\global\long\def\FFn{\widehat{F}_{\xn}}
\global\long\def\ffn{\widehat{f}_{\xn}}
\global\long\def\vvn{\widehat{\vp}_{\xn}}
\global\long\def\HHnaR{\widehat{H}_{\xn,R}^{0}}
\global\long\def\HHnaL{\widehat{H}_{\xn,L}^{0}}
\global\long\def\YYnR{\mathbb{Y}_{\xn,R}}
\global\long\def\YYnL{\mathbb{Y}_{\xn,L}}
\global\long\def\HHn{\widehat{H}_{\xn}}
\global\long\def\YYn{\mathbb{Y}_{\xn}}

\newcommand{\SLC}{{\cal SLC}}
\newcommand{\LC}{{\cal LC}}

\usepackage{xparse} %
\usepackage{etoolbox} %
\NewDocumentCommand\tnb{O{}}{
  \ifstrempty{#1}{
    t_{n,b}
  }{
    t_{#1}
  }
}

\begin{document}

\begin{frontmatter}
  \title{Univariate log-concave density estimation with symmetry or modal constraints}
  \runtitle{Mode- or symmetry-constrained LC MLE}

  \begin{aug}
    \author{\fnms{Charles R.}
      \snm{Doss}\thanksref{t1}\ead[label=e1]{cdoss@stat.umn.edu}\ead[label=u2,url]{http://users.stat.umn.edu/\textasciitilde cdoss/}}
    \and
    \author{\fnms{Jon A.} \snm{Wellner}\thanksref{t2}\ead[label=e2]{jaw@stat.washington.edu}}
    \ead[label=u1,url]{http://www.stat.washington.edu/jaw/}
    \thankstext{t1}{Supported in part by NSF Grants DMS-1104832 and a University of Minnesota Grant-In-Aid grant.} %
    \thankstext{t2}{Supported in part by NSF Grants DMS-1104832 and DMS-1566514,  and NI-AID grant 2R01 AI291968-04}
    \runauthor{Doss and Wellner}
    \address{School of Statistics \\University of Minnesota\\Minneapolis,  MN 55455\\
      \printead{e1}\\
      \printead{u2}}
    \address{Department of Statistics, Box 354322\\University of Washington\\Seattle, WA  98195-4322\\
      \printead{e2}\\
      \printead{u1}}
  \end{aug}

\begin{abstract}
  We study nonparametric maximum likelihood estimation of a log-concave density function $f_0$ which is known to satisfy further constraints, where either (a) the mode $m$ of $f_0$ is known, or (b) $f_0$ is known to be symmetric about a fixed point $m$.  We develop asymptotic theory for both constrained log-concave maximum likelihood estimators (MLE's), including consistency, global rates of convergence, and local limit distribution theory.  In both cases, we find the MLE's pointwise limit distribution at $m$ (either the known mode or the known center of symmetry) and at a point $x_0 \ne m$.  Software to compute the constrained estimators is available in the R package \verb+logcondens.mode+.

  The symmetry-constrained MLE is particularly useful in contexts of location estimation.  The mode-constrained MLE is useful for mode-regression.  The mode-constrained MLE can also be used to form a likelihood ratio test for the location of the mode of $f_0$.  These problems are studied in separate papers.  In particular, in a separate paper we show that, under a curvature assumption, the likelihood ratio statistic for the location of the mode can be used for hypothesis tests or confidence intervals that do not depend on either tuning parameters or nuisance parameters.
\end{abstract}



\end{frontmatter}
\newpage

\tableofcontents

\newpage

\section{Introduction and overview}

\label{sec:Intro}
The classes of log-concave densities on $\RR$ (and on $\RR^d$) have great importance in statistics
for a variety of reasons including their many natural closure properties, including closure under convolution,
affine transformations, convergence in distribution, and marginalization.
These classes are also unimodal and serve as important nonparametric generalizations
of the class of Gaussian distributions.

Nonparametric estimation in the {\sl unconstrained classes} of log-concave densities
has developed rapidly in the past 10--15 years.
Existence
of maximum likelihood estimators for log-concave densities on $\RR$ was provided by
\cite{MR1941467},  %
while
\cite{MR2459192} %
established consistency.
\cite{DR2009LC} %
gave rates of convergence in certain uniform metrics,
 and provided efficient algorithms based on
``active set'' methods (see also \cite{Dumbgen:2010ux}).
\cite{BRW2007LCasymp}  %
established pointwise limit distribution theory for the MLE's,  %
while \cite{DossWellner:2016a} established rates of convergence of the MLE in the Hellinger metric.
There has also been rapid progress in estimation of log-concave densities on $\RR^d$;
see e.g.\
\cite{MR2758237}, %
\cite{MR2645484},  
\cite{MR2816336}, %
\cite{MR2766867}, %
and \cite{MR3485962}.

Interesting uses of the unconstrained log-concave MLE's in more complicated
models, mostly in mixture modeling and clustering,  have been considered by
\cite{MR2408591},   %
\cite{MR2370883},   %
\cite{MR2757433},  %
and
\cite{MR2645484}.  

On the other hand, for a number of important statistical problems it is of great
interest to understand
estimation in several important {\sl sub-classes} of the class of all log-concave densities on $\RR$. %
\begin{itemize}
\item
For testing that a log-concave density on $\RR$ %
is symmetric about a known point, for example $0$, %
we need know how to estimate the log-concave density both {\sl with} and {\sl without} the constraint of symmetry.
\item
For the basic problem of estimation of location with a symmetric error density, it is important to
know how to estimate a {\sl symmetric} log-concave density with mode (and median and mean) equal to $0$.
\item
For inference about the mode of a log-concave density it is necessary to understand how to
estimate a log-concave density with a known mode $m$ (but without the constraint of symmetry).
\end{itemize}

Once the properties of nonparametric estimators within these sub-classes is understood,
then the estimators can be used to develop statistical methods with known properties
for other more complex statistical problems.
For example:  the basic procedures we study here can be viewed as building blocks to be used for, among others:
\begin{description}
\item[(a)]  Testing the hypothesis of symmetry of a log-concave density.
\item[(b)]  Estimation of the location of a symmetric log-concave density.
\item[(c)]  Inference about the mode of a log-concave density.
\item[(d)]  Nonparametric modal regression (as in \cite{MR3476607}, but using log-concavity). %
\item[(e)]  Semiparametric estimation in mixture models based on symmetric log-concave distributions; see e.g.
        \cite{BalabdaouiDoss:2017}, \cite{pu2017semiparametric}, and \cite{MR2370883}. %
\item[(f)]  Modal clustering (as in \cite{1609.04721}, but using log-concavity). %
\item[(g)] Estimation of a spherically symmetric multivariate log-concave density, which is pursued in \cite{xu2017high}.
\item[(h)]  Inference about the center of an elliptical multivariate distribution based on the assumption of
a log-concave underlying shape.
\end{description}
\smallskip

Thus our focus here is on estimation of a log-concave density in two important sub-classes:
Let $\mathcal{L C}$ denote the class of all log-concave densities on the real line $\RR$.
The two subclasses we study here are: \\ %
(1)  The class $\mathcal{LC} (m) = \mathcal{LC}_m$  of all log-concave densities with mode a \\
$\phantom{blab}$fixed number $m$.\\
(2)  The class $\mathcal{S L C}(0) = \mathcal{SLC}_0$ of all log-concave densities symmetric at $0$.

We let $\hat{f}_n^0$ denote the maximum likelihood estimator of $f_0 \in {\cal L C}_m$ 
 based on an i.i.d. sample $X_1, \ldots , X_n$ from $f_0$; and we let $\hat{g}_n^0$ denote the maximum likelihood estimator
 of $g_0 \in {S L C}_0$, based on an i.i.d. sample from $g_0$.\\

We rely on the methods and properties developed here for the subclass $\mathcal{L C}(m)$
to derive new inference procedures for the mode in
\cite{Doss-Wellner:2016ModeInference}.
The sub-class $\mathcal{S L C} (0)$ has already been used in \cite{BalabdaouiDoss:2017}  %
to study semiparametric mixture models.
The methods developed here for $\mathcal{ S L C} (0)$  are also being used in an on-going study
by \cite{Laha:2017}  of efficient estimation of a location parameter in the classical semiparametric symmetric location
model with the (very natural) assumption of a symmetric log-concave error distribution.
Methodology based on modes or local maxima of nonparametrically estimated functions has seen a resurgence in recent years;
see, e.g., \cite{MR3476607}, 
\cite{MR3375871}, 
and \cite{MR3485960}. 
A recent survey on estimation and inference for the mode and on mode-based methodology
is given by
\cite{Chacon:2018}.
%

%
%
%
%
%
%
%
%
%
%
%
%
%
%
%
%
%
%
%
%
%
%
%
%

%
%
%
%
%
%
%
%
%
%
%
%
%
%
%
%
%
%
%
%
%

%
%
%
%
%
%
%
%
%
%
%
%
%
%
%

Thus our main goals here are the following:
\begin{enumerate}[label=(\alph*)]
\item To show that the mode-constrained MLE's $\hat{f}_n^0 \in \mathcal{L C}(m)$ and $\hat{g}_n^0 \in \mathcal{S L C}(0)$
exist and to provide useful characterizations thereof.
\item  Establish useful finite-sample properties of $\hat{f}_n^0$  and $\hat{g}_n^0$.
\item Establish consistency of the mode-constrained and symmetric mode-constrained MLE's with respect to %
the Hellinger metric.
\item Establish local rates of convergence of the constrained estimators $\hat{f}_n^0$  and $\hat{g}_n^0$
 and establish the (pointwise) asymptotic distributions of the constrained estimators.
\item Establish global rates of convergence of the constrained estimators.
\end{enumerate}

Here is a brief summary of the paper:
In Section 2 we show that the constrained estimators exist and satisfy useful characterizations.
Section 3 provides plots of the constrained estimators and provides comparisons to each other and to the unconstrained
maximum likelihood estimators $\hat{f}_n \in \mathcal{L C}$.
In Section 4 we summarize results concerning consistency and global rates of convergence,
while Section 5 addresses local rates of convergence and limiting distributions at fixed points.
Section 6 summarizes some problems and difficulties concerning extensions to higher dimensions.
All the proofs are given in the Appendix.  %

Many of our theorems have parts labeled ``A'', ``B'', and ``C.''  In general the ``A parts'' of results here have been proved by other authors (as noted in the theorem statements), the ``B parts'' were proved (for the most part) in the University of Washington Ph.D. dissertation of the first author, Doss (2013b). The ``C parts'' are new findings by the present authors, whose proofs are in some cases (as noted in text near the corresponding results) related to proofs developed by  \cite{BalabdaouiDoss:2017}.
  

%
%
%
%
%
%
%
%
%
%
%

%
%

\section{Maximum likelihood estimator finite sample properties: unconstrained and mode-constrained}
\label{sec:MLEs}

\subsection{Notation and terminology}

Several classes of concave functions will play a central role in this paper.
In particular, we let
\begin{eqnarray}
  {\cal C} := \{ \varphi : \ \RR \rightarrow [-\infty, \infty) \ | \ \varphi \ \ \mbox{is concave, closed, and proper} \}
  \label{UnconstrainedConcaveClass}
\end{eqnarray}
and, for any fixed $ m \in \RR$,
\begin{eqnarray}
  {\cal C}_m := \{ \varphi \in {\cal C} \ | \ \varphi(m) \ge \varphi (x) \ \ \mbox{for all} \ \ x \in \RR \}
  \label{ConstrainedConcaveClass}
\end{eqnarray}
is the class of concave functions on $\RR$ with mode at $m$.
We also let
\begin{eqnarray}
{\cal S C}_0 := \{ \varphi \in {\cal C}_0 \ | \ \varphi (-x) = \varphi (x) \ \ \mbox{for all} \ \ x \in \RR \}.
\label{SymmetricConstrainedConcaveClass}
\end{eqnarray}
Here proper and closed concave functions are as defined in
\cite{MR0274683}, %
pages 24 and 50.  We will follow  the convention that all concave functions $\varphi$
are defined on all of $\RR$ and take the value $-\infty$ off of their effective domains where
$\mbox{dom} (\varphi ) := \{ x \ : \ \varphi (x) > -\infty \}$
(\cite{MR0274683}, %
page 40).
The classes of unconstrained and constrained log-concave densities are then
\begin{eqnarray*}
  && {\cal LC } := \left \{ e^{\varphi}  \ : \ \int e^{\varphi} d \lambda = 1, \ \ \varphi  \in {\cal C} \right \} , \\
  && {\cal LC}_m := \left \{ e^{\varphi} \ : \ \int e^{\varphi} d \lambda =1, \ \ \varphi \in {\cal C}_m \right \}, \ \ \ \mbox{and}\\
  && {\cal SLC}_0 := \left \{ e^{\varphi} \ : \int e^{\varphi} d \lambda =1, \ \ \varphi \in {\cal SC}_0 \right \},
\end{eqnarray*}
where $\lambda$ is Lebesgue measure on $\RR$.
We let $X_1, \ldots , X_n$ be the observations, independent and identically distributed with density
$f_0$ with respect to Lebesgue measure.
Here we assume throughout that $f_0 \in {\cal LC}$ and frequently that $f_0 = e^{\varphi_0} \in {\cal LC}_m$
for some $m\in \RR$ or $f_0 = e^{\varphi_0} \in {\cal S L C }_0$.
We let $X_{(1)} < \cdots < X_{(n)}$ denote the order statistics of the $X_i$'s, and write
$|X|_{(1)} < \cdots < |X|_{(n)}$ for the order statistics of $|X_1|, \ldots , |X_n|$.
We let $\PP_n = n^{-1} \sum_{i=1}^n \delta_{X_i}$ denote the empirical measure, let
$\FF_n (x) = n^{-1} \sum_{i=1}^n 1_{(-\infty,x]} (X_i )$ denote the empirical distribution function,
and let $\GG_n (x) = n^{-1} \sum_{i=1}^n 1_{[0,x]} (|X_i |)$ denote  the empirical distribution function
of $|X_1| , \ldots , |X_n |$.

We define the log-likelihood criterion function $\Psi_n : {\cal C} \rightarrow \RR$ by
\begin{eqnarray}
  \Psi_n (\varphi ) = \frac{1}{n} \sum_{i=1}^n \varphi (X_i)  - \int_{\RR} e^{\varphi (x)} dx
  = \PP_n \varphi  - \int_{\RR} e^{\varphi} d\lambda
  \label{eqn:AdjLogLikCriterion}
\end{eqnarray}
where we have used the standard device of including the Lagrange term $\int_{\RR} e^{\vp(x)} dx$ in $\Psi_n$ to avoid the normalization constraints involved in the classes $\LC_m$ and $\SLC_0$.  This is as in
\cite{MR663433},  %
\cite{DR2009LC}, %
and other current literature.

We will denote the unconstrained MLE's of $\varphi_0$,  $f_0$, and $F_0$ by
$\widehat{\varphi}_n$, $\widehat{f}_n$, and $\widehat{F}_n$ respectively.  The corresponding
constrained estimators with mode $m$ and symmetric estimators with mode at $0$ will be denoted by
$\widehat{\varphi}_n^0$, $\widehat{f}_n^0$, $\widehat{F}_n^0$, and
$\widehat{\psi}_n^0$, $\widehat{g}_n^0$, $\widehat{G}_n^0$ respectively.
Thus
\begin{eqnarray*}
  \widehat{\varphi}_n \equiv \argmax_{\varphi \in {\cal C}} \Psi_n (\varphi),\ \
  \widehat{\varphi}_n^0  \equiv \argmax_{\varphi \in {\cal C}_m } \Psi_n (\varphi) , \ \ \mbox{and} \ \
  \widehat{\psi}_n^0  \equiv \argmax_{\psi \in {\cal SC}_0 } \Psi_n (\psi) .
\end{eqnarray*}
Before proceeding to results concerning existence and uniqueness of the constrained estimators
$\vvna$ and $\psina$, we first explain some undesirable properties of ``naive'' constrained estimators based on the unconstrained MLE's $\ffn$ and $\vvn$.

\subsection{Naive Estimators}
\label{subsec:naive-estimators}

We can easily construct ``naive'' estimators under our two classes of constraints.  For instance, a naive mode-constrained estimator based on the unconstrained log-concave MLE is $\tilde{f}_n^0(x) = \ffn(x - (m-\hat{m}_n))$,
where $\hat{m}_n$ is the mode of $\ffn$.
Then $\tilde{f}_n^0$ indeed  has mode $m$.
Let $\tilde{\vp}_n^0 = \log \tilde{f}_n^0$.
Unfortunately, these estimators have quite undesirable properties.  For example,
when $\varphi_0^{\prime} (x_0) \not= 0$, we can see that
\begin{align*}
  n^{2/5}\left ( \tilde{\vp}_n^0(x_0)  - \vvo(x_0) \right )
  & =  n^{2/5} \left ( \vvn(x_0 + (\hat{m}_n - m))
    - \vvo(x_0) - \varphi_0^{\prime} (x_0) ( \hat{m} - m) \right ) \\
  & \qquad + \ n^{2/5} \varphi_0^{\prime} (x_0)  ( \hat{m} - m) \\
  & =  O_p(1)  + n^{1/5} O_p(1)
\end{align*}
since $\hat{m}_n - m = O_p(n^{-1/5}) $
by Theorem~3.6 of \cite{BRW2007LCasymp}, so the first summand is $O_p(1)$ by Corollary~2.2 of
\cite{BRW2007LCasymp} (and its proof, see  their (4.34)).
Thus, away from the mode, this naive estimator in fact converges at a slower rate than $n^{-2/5}$.

Similarly, a naive $0$-symmetric estimator can be constructed.  Let $\tilde{g}_n^s(y) = ( \ffn(\hat{m}_n + y) \ffn(\hat{m}_n - y))^{1/2}$.  Then $\tilde{g}_n^s$ is symmetric about its mode $0$.  (It is not necessarily a bona fide density that integrates to $1$,  but its integral converges to $1$.)   Again, unfortunately, a similar analysis as above shows that if $x_0 \ne m$, then
\begin{align*}
  n^{2/5} ( \log \tilde{g}_n^s(x_0) - \vvo(x_0))
  = O_p(1) +  n^{2/5} \hat{m}_n \vvo'(x_0)
\end{align*}
since $\vvo'(x_0) = - \vvo'(-x_0)$.  Since $n^{1/5} \hat{m}_n = O_p(1)$ and $\vvo'(x_0) \ne 0$, we again see that the naive estimator converges at a slower rate than $n^{-2/5}$.

In summary, naive plug-in estimation for the mode and symmetry constraints does not work.
The poor performance of these and other ``naive'' or ``plug-in'' estimators motivates  study of the constrained MLE's, which we now pursue.

\subsection{The  unconstrained and the constrained MLE's}
\label{ssec:UnconstrainedAndConstrainedEstimators}

To develop theory for the mode-constrained estimators $\widehat{\varphi}_n^0$,  $\widehat{f}_n^0$, and $\widehat{F}_n^0$ it will
be helpful to consider {\sl mode-augmented data} $Z_1, \ldots , Z_N$ with $N= n$ or $n+1$ as follows:
\begin{description}
\item[(1)]
  If $m = X_{(k)} $ for some $k \in \{1, \ldots , n \}$ then $Z_j \equiv X_{(j)}$ for $j\in \{ 1, \ldots , n\}$ and $N=n$.
\item[(2)]
  If $m \in (X_{(k-1)}, X_k)$ for some $k \in \{ 1, \ldots , n+1 \}$ (where $X_{(0)} \equiv - \infty $ and $X_{(n+1)} \equiv + \infty$),
  then we define $Z_i \equiv X_{(i)}$ for $i \in \{ 1, \ldots , k-1 \}$, $Z_k \equiv m$,  and $Z_i \equiv X_{(i-1)}$ for $i \in \{ k+1, \ldots , n+1\}$.
  In this case \\
  $\underline{Z} = (X_{(1)}, \ldots , X_{(k-1)}, m , X_{(k)} , \dots , X_{(n)}) \in \RR^{n+1} $ and $N = n+1$.
\end{description}

\begin{theorem}
\label{thm:Existence}
The following statements hold almost surely when $X_1, \ldots, X_n$ are i.i.d.\ from a density on $\RR$.
\begin{enumerate}[label=\Alph*.,ref=\Alph*]
\item \label{thm:existence:A-unconstrained} (\cite{MR2459192}, \cite{RufibachThesis}) %
  For $n\ge 2$ the (unconstrained) nonparametric MLE $\widehat{\varphi}_n$ exists and is unique.  It is linear on all intervals $[X_{(j)}, X_{(j+1)}]$, $j=1, \ldots , n$.  Moreover, $\widehat{\varphi}_n = - \infty$ and $\widehat{f}_n = 0$ on $\RR\setminus [X_{(1)}, X_{(n)} ]$.

\item \label{thm:existence:B-mode-constrained} (\cite{Doss:2013}) For $N\ge 2$ the mode-constrained MLE $\widehat{\varphi}_n^0$ exists and is unique.  It is piecewise linear with knots at the $Z_i$'s and domain $[Z_1, Z_N]$.  If $m$ is not a data point, then at least one of $(\widehat{\varphi}_n^0 )^{\prime} (m+)$ or $(\widehat{\varphi}_n^0 )^{\prime} (m-)$ is $0$.

\item \label{thm:existence:C-symmetric} The constrained MLE $\widehat{\psi}_n^0 \in {\cal SC}_0$ exists for $n\ge1$ and is unique.  It is piecewise linear with knots contained in the set
  of  $2n+1$ points \\
  $-|X|_{(n)}, \ldots , - | X|_{(1)} , 0, |X|_{(1)}, \ldots , | X |_{(n)}$, and is $-\infty$ for $x \notin [-|X|_{(n)}, |X|_{(n)} ]$.
%
Furthermore, $(\widehat{\psi}^0_{n})^{\prime} (0\pm)= 0$.
\end{enumerate}
\end{theorem}

\medskip

\par\noindent
The previous result shows that the MLE's exist.  Unfortunately, there is no closed form expression for the MLE's.  However, since they are solutions to optimization problems, they satisfy certain optimality conditions.  Thus, the next two theorems we present provide   systems of inequalities and equalities that characterize the MLE's.
\medskip

\begin{theorem}
  \label{thm:CharThmOne}
  $\phantom{blab}$ %
  \begin{enumerate}[label=\Alph*.,ref=\Alph*,leftmargin=*]
  \item  \label{thm:CharThmOne:A-UC}
  (\cite{RufibachThesis}, \cite{DR2009LC} )  %
  Let $\widehat{\varphi}_n $ be a concave function such that $ \{ x : \ \widehat{\varphi}_n (x) > -\infty \} = [X_{(1)}, X_{(n)}]$.
  Then $\widehat{f}_n = e^{\widehat{\varphi}_n} \in {\cal LC} $ is the unconstrained MLE if and only if
  \begin{eqnarray*}
    \int \Delta (x) d \FF_n (x) \le \int \Delta (x)  \exp ( \widehat{\varphi}_n (x)) dx = \int \Delta (x) d \widehat{F}_n (x)
  \end{eqnarray*}
  for any function $\Delta: \RR \rightarrow \RR$ such that $\widehat{\varphi}_n + \lambda \Delta$ is concave for some $\lambda >0$.

\item  \label{thm:CharThmOne:B-MC}
  (\cite{Doss:2013}) %
  Suppose that $\widehat{f}_n^0 = e^{\widehat{\varphi}_n^0 }\in {\cal LC}_m$.
  Then $\widehat{f}_n^0 $
  is the MLE over ${\cal LC}_m$ if and only if
  \begin{eqnarray}
    \int \Delta d \FF_n \le \int \Delta d \widehat{F}_n^0
    \label{eq:char0}
  \end{eqnarray}
  for all $\Delta $ such that   $\widehat{\varphi}_n^0 + t \Delta \in {\cal C}_m$ for some $t>0$.

\item \label{thm:CharThmOne:C-symm}
  Suppose that
  $\widehat{g}_n^0 = e^{\widehat{\psi}_n^0 }\in {\cal SLC}_0$ and
  $\widehat{G}_n^0 (x) \equiv \int_{-\infty}^x \widehat{g}_n (y)dy$.  Then
  $\widehat{g}_n^0 $
  is the MLE over ${\cal SLC}_0$ if and only if
  \begin{eqnarray}
    \int \Delta d \FF_n \le \int \Delta d \widehat{G}_n^0
    \label{eq:char0s}
  \end{eqnarray}
  for all $\Delta $ such that $\widehat{\psi}_n^0 + t \Delta \in {\cal SC}_0$
  for some $t>0$.
\end{enumerate}
\end{theorem}

\medskip

\par\noindent
To state the second characterization theorem for the MLE's, we first introduce some further notation and definitions.
For a continuous and piecewise linear function $h: [A,B]  \rightarrow \RR$ we define its {\sl knots} to be
\begin{eqnarray*}
  {\cal S}_n (h) := \{ t \in (A , B ) : \ h' (t-) \not= h' (t+) \} \cup \{ A ,  B \} .
\end{eqnarray*}
Note that $\widehat{\varphi}_n$, $\widehat{\varphi}_n^0$, and $\widehat{\psi}_n^0$  are
all continuous and piecewise linear functions
(with $A = X_{(1)}$, $B = X_{(n)}$ in the case of $\widehat{\varphi}_n$ and $\widehat{\varphi}_n^0$,
and with $A = -|X|_{(n)}$, $B=|X|_{(n)}$ in the case of $\widehat{\psi}_n$), and we have
\begin{eqnarray*}
  && {\cal S}_n ( \widehat{\varphi}_n ) \subset \{ X_{(1)}, X_{(2)}, \ldots , X_{(n)} \}, \\
  && {\cal S}_n ( \widehat{\varphi}_n^0) \subset \{ X_{(1)}, X_{(2)}, \ldots , X_{(n)} \}, \\
  && {\cal S}_n ( \widehat{\psi}_n^0 ) \subset \{ - | X|_{(n)}, \ldots , | X |_{(n)} \}.
\end{eqnarray*}
Now suppose that $\widehat{\varphi}_n^0$ is piecewise linear with knots at the
(mode-augmented) data, let $m\in \RR$, and assume that %
$\widehat{f}_n^0 \equiv \exp ( \widehat{\varphi}_n^0) \in \mathcal{L C}_m$.
For $t \in \RR$ define
\begin{eqnarray}
  \label{eq:def:finite-sample-L-R-processes}
  \begin{array}{l l}
    \FF_{n,L} (t) \equiv \int_{(-\infty, t]} d \FF_n (y), & \FF_{n,R} (t) \equiv \int_{[t, \infty)} d \FF_n (y) , \\
    \YY_{n,L} (t) \equiv \int_{X_{(1)}}^t \FF_{n,L} (x) dx , & \YY_{n,R} (t) \equiv \int_{t}^{X_{(n)}} \FF_{n,R} (x) dx , \\
    \widehat{F}_{n,L}^0 (t) \equiv \int_{-\infty}^t \widehat{f}_n^0 (y) dy, & \widehat{F}^0_{n,R} (t) \equiv \int_t^{X_{(n)}} \widehat{f}_n^0 (y) dy, \\
    \widehat{H}_{n,L}^0  (t) \equiv \int_{X_{(1)}}^t \widehat{F}_{n,L}^0 (x) dx , &  \widehat{H}_{n,R}^0 (t) \equiv \int_t^{X_{(n)}} \widehat{F}_{n,R}^0 (x) dx .
  \end{array}
\end{eqnarray}
\begin{definition}
  With $m$ considered as a possible knot of $\widehat{\varphi}_n^0$ we say that $m$ is a {\sl left knot} (or LK) if
  $(\widehat{\varphi}_n^0 )^{\prime} (m-) > 0$ and that $m$ is a {\sl right knot} (or RK) if
  $(\widehat{\varphi}_n^0 )^{\prime} (m+) < 0$.  We say that $m$ is {\sl not a knot} (or NK) if $(\widehat{\varphi}_n^0)^{\prime} (m) = 0$.
  All other knots are considered to be left knots (LKs) or right knots (RKs) depending on whether they are strictly smaller or strictly
  larger than $m$.
\end{definition}

%

\begin{theorem}
  \label{thm:CharThmTwo}
  $\phantom{blab}$ %
  \begin{enumerate}[label=\Alph*.,ref=\Alph*,leftmargin=*]
  \item \label{thm:CharThmTwo:A-UC}
    (\cite{RufibachThesis}, %
    \cite{DR2009LC})
    Let $\widehat{F}_n (x) \equiv \int_{-\infty} ^ x e^{ \widehat{\varphi}_n (y) } dy$, and assume further that
    $\widehat{f}_n = e^{\widehat{\varphi}_n} \in {\cal LC}$.
    Then $\widehat{f}_n$ is the MLE in ${\cal LC}$ if and only if
    \begin{eqnarray*}
      \widehat{H}_n (t) \equiv \int_{X_{(1)}}^t \widehat{F}_n (y)dy \le \int_{X_{(1)}}^t \FF_n (y) dy \equiv \YY_n (t) \ \ \
      \mbox{for all} \ \ t \in \RR
    \end{eqnarray*}
    with equality if $t \in {\cal S}_n ( \widehat{\varphi}_n )$.

  \item \label{thm:CharThmTwo:B-MC}
    \citep{Doss:2013}
    With the notation in \eqref{eq:def:finite-sample-L-R-processes},  $\widehat{f}_n^0 = e^{\widehat{\varphi}_n^0}$ is the MLE of $f_0 \in {\cal LC}_m$ if and only if
    \begin{eqnarray}
      \widehat{H}_{n,L}^0 (t) \le \YY_{n,L} (t)\ \ \ \mbox{for} \ \ X_{(1)} \le t \le m
      \label{eq:LeftFenchelInequalitiesConstrained}
    \end{eqnarray}
    \begin{eqnarray}
      \widehat{H}_{n,R}^0 (t) \le \YY_{n,R} (t)\ \ \
      \mbox{for} \ \ m \le t \le X_{(n)}
      \label{eq:RightFenchelInequalitiesConstrained}
    \end{eqnarray}
    with equality in (\ref{eq:LeftFenchelInequalitiesConstrained}) if $t$ is a {\sl left knot} of $\widehat{\varphi}_n^0$
    and equality in (\ref{eq:RightFenchelInequalitiesConstrained}) if $t $ is a {\sl right knot} of $\widehat{\varphi}_n^0$.

  \item \label{thm:CharThmTwo:C-symm}
    $\widehat{g}_n^0 = e^{\widehat{\psi}_n^0} \in {\cal SLC}_0$ is the MLE  if and only if
    $\widehat{g}_n^+ \equiv 2 \widehat{g}_n^0$ satisfies,
    with $\widehat{G}_{n,R}^+(x) \equiv \int_x^{|X|_{(n)}} \widehat{g}_n^+ (y) dy$
    and $\FF_{n,R}^+ (x)\equiv n^{-1} \sum_{i=1}^n 1 \{ |X|_{(i)} \ge x\}$,
    \begin{eqnarray*}
      \int_{t}^{|X|_{(n)}}
      \widehat{G}_{n,R}^+ (x)dx
      \left \{
      \begin{array}{ l l}
        \le \int_{t}^{|X|_{(n)}}
        \FF_{n,R}^+ (x)dx,  & \mbox{if} \ \ t \in [0, |X|_{(n)}], \\
        =  \int_{t}^{|X|_{(n)}}
        \FF_{n,R}^+ (x)dx,
                            & \mbox{if} \ \ t \in
                              {\cal S}_n (\widehat{\psi}_n^0) \cap [0, |X|_{(n)} ].
      \end{array}
                              \right .
    \end{eqnarray*}
  \end{enumerate}
\end{theorem}

\medskip

\begin{remark}
  \label{rm:GlobalConstraintViaDFequalsOne}
  The conditions (\ref{eq:LeftFenchelInequalitiesConstrained}) and
  (\ref{eq:RightFenchelInequalitiesConstrained}) only involve data from the left and right sides of $m$, and hence are separate
  characterizations in a sense.  But they are coupled by way of the (global) constraint $\widehat{F}_n^0 (X_{(n)} ) = 1$
  (or, equivalently, $\widehat{\varphi}_n^0 \in {\cal C}_m$) which involves the data on {\sl both} sides of $m$.
\end{remark}

\begin{remark}
\label{rm:ConnectionWithBalabdaouiDoss17}
The ``C parts'' of Theorems~\ref{thm:Existence}, \ref{thm:CharThmOne}, and \ref{thm:CharThmTwo}
will be proved here in detail via methods similar to those introduced briefly in 
\cite{BalabdaouiDoss:2017} in the course of a study of two-component mixture models based on symmetric log-concave components.
\end{remark}

These characterization theorems have two important corollaries.  (Recall that $\GG_n$ denotes the empirical distribution function
of the $|X_i|$'s.)
\smallskip

\begin{corollary}
  \label{cor:EstimatedFNearlyTouchesEDFatKnotsUnconstrained}
  (MLE's related to $\FF_n$ at knot points)
  Each of the following holds almost surely.
  \begin{enumerate}[label=\Alph*.,ref=\Alph*,leftmargin=*]
  \item   $\FF_n - n^{-1} \le \widehat{F}_n \le \FF_n $  on  ${\cal S}_n ( \widehat{\varphi}_n )$.

  \item  $\FF_n - n^{-1} \le \widehat{F}_n^0 \le \FF_n $ on ${\cal S}_n ( \widehat{\varphi}_n^0 ) \setminus \{ m \}$.

  \item  $\GG_n - n^{-1} \le \widehat{G}_n^+ \le \GG_n$ on ${\cal S}_n ( \widehat{\psi}_n^0 ) \cap [0, |X|_{(n)}]$.

  \end{enumerate}
\end{corollary}

\smallskip

Now for any distribution function $F$ on $\RR$ let $\mu(F) \equiv \int x dF(x)$
and $\mbox{Var}(F) = \int (x - \mu(F))^2 dF(x)$.

\begin{corollary}
  \label{cor:MeanVarianceInequalities}
  (Mean and variance inequalities)\\
A.  $\mu (\widehat{F}_n ) = \mu ( \FF_n )$   and $\mbox{Var}( \widehat{F}_n ) \le \mbox{Var} ( \FF_n ) $.\\
B. $\mu(\GGna) = 0$ and $\mbox{Var}(\GGna) \le \int x^2 d\FF_n(x)$.
\end{corollary}
\par\noindent

Because $\Delta_{\pm}(x) = \pm x$ does not have mode $m$, and because $-(x-\mu)^2$ only has mode $m$ if $\mu = m$, we cannot make comparisons between the mean and variances of $\FF_n$ and $\FFna$.

\section{Hellinger consistency and rates}
\label{sec:Consistency}

\cite{MR2459192}  %
showed that the unconstrained MLE's $\{ \widehat{f}_n \}$ are a.s. consistent in the
Hellinger metric $H$ where $H^2 (p,q) \equiv (1/2) \int \{ \sqrt{p} - \sqrt{q} \}^2 d \lambda$,
and their methods also yield consistency for the MLE's over any sub-class ${\cal S} \subset {\cal LC}$ for which
the MLE's $\{ \widehat{g}_n \}$ exist and satisfy
\begin{eqnarray*}
  \sup_n \sup_x \log \widehat{g}_n (x) < \infty \ \ \ \mbox{a.s.} .
\end{eqnarray*}
This nicely includes the subclass ${\cal S} = {\cal LC}_m$ when $f_0 \in {\cal LC}_m$; i.e. the mode $m$ has been
correctly specified.
Further consistency results are due to
\cite{MR2645484},  %
\cite{RufibachThesis}, %
and \cite{DR2009LC}.
To the best of our knowledge, this is the first treatment
of the consistency and global rate properties
of the constrained estimators.

\begin{theorem}
\label{thm:HellingerConsistencyRates}
(Hellinger consistency and rates of convergence)\\
A. \  (\cite{DossWellner:2016a}) If $f_0 \in {\cal LC }$,
then $H ( \widehat{f}_n , f_0 ) = O_p (n^{-2/5})$.    \\
B. \ (\cite{Doss:2013}) If $f_0 \in {\cal LC}_m$,
then $H ( \widehat{f}_n^0 , f_0 ) = O_p (n^{-2/5})$. \\
C. \ (\cite{BalabdaouiDoss:2017})  If $f_0 \in {\cal SLC}_0$,
       then $H ( \widehat{g}_n^0 , f_0) = O_p (n^{-2/5})$.
\end{theorem}

\begin{remark}
  \cite{MR3576560} %
  extend Part A of Theorem~\ref{thm:HellingerConsistencyRates} by upper
  bounding the maximal risk of $\ffn$: their Theorem 5 implies that $\sup_{f \in \LC} E_f H^2( \ffn, f)$ is $O(n^{-4/5})$ (considering squared Hellinger rather than Hellinger distance).
  They also provide a matching lower bound: their Theorem~1 implies
 $$\inf_{\tilde{f}_n} \sup_{f \in \LC} E_f H^2(\tilde{f}_n , f ) \ge c n^{-4/5},$$
 for some $c>0$, where the infimum is over all (measurable) estimators $\tilde{f}_n$ of $f$.
  Neither upper nor lower bounds for the (Hellinger) minimax risk  are known for either of the  constrained density classes we consider in the present paper, although we conjecture that $n^{-4/5}$ is the minimax rate of convergence in both cases.
\end{remark}

\begin{remark}
When $f_0 \in {\cal LC} \setminus {\cal LC}_m$, then we can show that
$H^2 ( \widehat{f}_n^0 , f_0^*) \rightarrow_{a.s.} 0$ where $f_0^*$ satisfies
$$
K(f_0 , f_0^*) = \inf_{g \in {\cal LC}_m} K( f_0 , g).
$$
%
Similarly, when $f_0 \in {\cal LC} \setminus {\cal SLC}_0$, then we can show that
$H^2 ( \widehat{g}_n^0 , g_0^*) \rightarrow_{a.s.} 0$ where $g_0^*$ satisfies
$$
K(f_0 , g_0^*) = \inf_{g \in {\cal SLC}_0} K( f_0 , g) ;
$$
but we will not pursue this here since our  goal
in this paper is to understand the null hypothesis (or correctly specified) behavior of
the constrained estimators $\widehat{\varphi}_n^0$, $\widehat{f}_n^0$, and $\widehat{F}_n^0$.
See
\cite{Doss-Wellner:2016ModeInference} %
for some initial steps concerning the power of a likelihood ratio test based on
$2 \log \lambda_n$ when $\ffo \in \mc{LC} \setminus \mc{LC}_m$.
\end{remark}

In addition to considering Hellinger distance, one can consider the sup norm (on compact sets) as a metric for global convergence.  It turns out that the proofs in \cite{Doss-Wellner:2016ModeInference} rely crucially on knowing the rate of sup-norm convergence for $\vvna$
(as well as for $\vvn$).  Thus we study the sup-norm rate of convergence for $\vvna$ in that paper.
In Theorem~4.1 of that paper we find, when the true log-density satisfies a H\"older condition of order $2$ and $\vvo^{(2)}(m) < 0$, that
the rate of convergence is $(\log n / n)^{2/5}$ on compact sets interior to the support of $f_0$.

\section{Local limit processes and limiting distributions at fixed points}
\label{sec:LimitDistributions}

Our goal in this section is to describe the limiting distributions of our estimators, both
unconstrained and constrained, at fixed points $x_0$ (and $m$ and $0$) at which the
true density $f_0$ satisfies a curvature condition.  We also want to compare and contrast
the behavior of the three different estimators.

\subsection{The limit processes, unconstrained and constrained}
\label{subsec:limit-processes}
We first need to introduce the local limit proceses which are needed to treat the local
(at a single point or in a neighorhood of a point) limiting distributions of the estimators,
unconstrained and constrained.
For all of our estimators (including the unconstrained and the two different mode-constrained estimators),
the limit distributions are not Gaussian.  Rather, they are defined in terms of
 so-called {\em invelope} processes of integrated Brownian motion.
We first recall  the invelope process related to the limit
distribution for the unconstrained estimators; this process was first
presented and studied in \cite{MR1891741} (and shown to yield the limit
distribution in several  convex function estimation problems in
\cite{MR1891742}).  %
Let $W$ be a two-sided standard Brownian motion starting at $0$ and for any
$t \in \RR$ let
\begin{equation}
  \label{eq:defn_Xs}
    X(t)  =  W(t)  - 4t^{3} ,  \ \ \ \
    Y(t)   =  \int_0^t X(s) ds = \int_0^t W(s)\, ds -t^4.
\end{equation}

\begin{theorem}[\cite{MR1891741}] %
  \label{thm:charzn_uniqueness_full_UC}
  Let $W$, $X$, and $Y$ be as in \eqref{eq:defn_Xs}.
  Then there
  exists an almost surely uniquely defined random continuous function $H$
  satisfying the following conditions:\\
  (i) \ \  The function $H$ is everywhere below $Y$:
    \begin{equation*}
      H(t) \le Y(t) \;\;\; \mbox{ for all } t \in \RR.
    \end{equation*}
 (ii) \ $H$ has a concave second derivative.\\
(iii)  $H$ satisfies
\begin{equation*}
      \int_{-\infty}^{\infty} (H(t)-Y(t)) d(H^{(3)})(t)=0.
    \end{equation*}
\end{theorem}
\medskip

The random variables $H^{(2)}(0)$ and $H^{(3)}(0)$ give the universal
component of the limit distribution of $\widehat{f}_n(x_0)$ and
$(\widehat{f}_n)'(x_0)$; see Theorem~\ref{thm:BRW2009-UMLE-limit}, below.

Theorem~\ref{thm:charzn_uniqueness_full_UC} concerns a process $H$, related
to the unconstrained concave estimation problem. In the mode constrained
estimation problem, $f_0 \in {\cal LC}_m$, instead of having one process we have two, one for the
left-hand side of $0$ (negative axis) and one for the right-hand side of
$0$ (positive axis).
(Here, $0$ corresponds to the mode $m$, by a translation.)
The definitions of the left- and right-hand processes depend on a
random starting point for the corresponding integrals involved, which we
will eventually denote $\TauL$ and $\TauR$
(this is made clear in \eqref{eq:defn:tau0-}--\eqref{eq:defn:tauL-tauR}, below).
To define $\TauL$ and $\TauR$,
we must define rigorously the possible `bend points' of $\vva$.
To describe the situation exactly, we also
will define `bend points' $\tau_+^0$ and $\tau_-^0$, satisfying
$\tau_+^0 \le \TauR$ and $\tau_-^0 \ge \TauL$,
where the inequality may or may not be
strict; these bend points arise in \eqref{eq:CharznEquality2_full} below.
For a concave function $g$, we let $g'(\cdot -)$ and $g'(\cdot +)$ be the left and right derivatives, respectively (which are always well defined).

\begin{theorem}
\label{thm:MC-process-uniqueness-theorem}
Assume that $\{ H_L (t) : \ t \le 0 \} $ and $\{ H_R (t) : \ t \ge 0 \}$ 
are random processes with concave second derivatives so that 
$\widehat{\varphi}^0 (t) \equiv H_L^{(2)} (t) 1_{(-\infty, 0)} (t) + H_R^{(2)} (t) 1_{[0,\infty)} (t)$ satisfies
$\widehat{\varphi}^0 \in {\cal C}_0$.
 Define the  `bend points' $\Sa$ by
 \begin{equation}
    \label{eq:defnSa}
    (\Sa(\vva))^c \equiv (\Sa)^c := \lb t \in \RR :  (\vva)^{(2)}(t \pm )  = 0 \rb.
\end{equation}
Next, define
  \begin{align}
    \tau_-^0(\vva) \equiv
    \tau_-^0 = \sup \lb t \in \Sa : (\vva)'((t-\epsilon)-) > 0 \mbox{ for all }
    \epsilon > 0\rb,     \label{eq:defn:tau0-} \\
    \tau_+^0(\vva) \equiv
    \tau_+^0 = \inf \lb t \in \Sa : (\vva)'((t+ \epsilon)+) < 0 \mbox{ for all
    } \epsilon > 0 \rb,  \label{eq:defn:tau0+} \\
    \TauL = \sup  \lp \Sa \cap (-\infty, 0) \rp
    \mbox{ and }
    \TauR = \inf \lp \Sa \cap (0, \infty) \rp. \label{eq:defn:tauL-tauR}
  \end{align}
  Let $W$ be a standard two-sided Brownian motion with $W(0)=0$, and for $t \in \RR$ let
\begin{equation*}
X(t) = W(t)  - 4t^3,
\end{equation*}
  \begin{equation}
    \label{eq:defYLR}
    \YYLt(t)  =  \int_t^{\TauL} \int_u^{\TauL} dX(v) du
    \quad \mbox{ and } \quad
    \YYRt(t) =  \int_{\TauR}^t \int^u_{\TauR} dX(v) du.
  \end{equation}
With these definitions,  we assume that:\\
(i) $-\infty < \TauL \le 0$ and $0 \le \TauR < \infty$ and
    \begin{equation}
      \label{eq:CharznFequalsX_M_full}
      \int_{\TauL}^{\TauR} (\vva(v)dv - dX(v)) = 0.
    \end{equation}
(ii)
    \begin{align}
      & \HHLt(t)-\YYLt(t) \le 0 \mbox{ for } t \le 0, \label{eq:CharznInequalityL2_full} \\
      & \HHRt(t) - \YYRt(t) \le 0 \mbox{ for } t \ge 0,  \label{eq:CharznInequalityR2_full}
    \end{align}
(iii)
    \begin{equation}
      \label{eq:CharznEquality2_full}
      \int_{(-\infty, \tau_-^0]} (\HHLt(u) - \YYLt(u)) d(\vva)'(u)
      = 0
      = \int_{[\tau_+^0, \infty)} (\HHRt(u)-\YYRt(u)) d(\vva)'(u).
    \end{equation}
Then, $\HHLt$ and $\HHRt$ are unique, as are $\TauL$ and $\TauR$.
\end{theorem}

Theorem~\ref{thm:MC-process-uniqueness-theorem} shows that processes with the
given properties are unique; that they exist follows from
the proofs of
Theorem~\ref{thm:MC-MLE-limit-A-process-version}
and
Theorem~\ref{thm:process-asymptotics-symmetry},
which show that $\HHLt$ and $\HHRt$ exist since they are limit versions of certain finite sample processes ($\HHnvvaL, \HHnvvaR$).



If $\vva$ is, in fact, piecewise linear, then $\tauminusa$ is just the last
knot point $\tau$ of $\vva$ with $(\vva)'(\tau-) > 0$.  By Theorem 23.1 of
\cite{MR0274683}, a finite, concave function on $\RR$ such as $\vva$
has well-defined right- and left-derivatives at all of $\RR$; the
specification of left- and right- derivatives in the definitions of
$\tauminusa$ and $\tauplusa$ are for concreteness but not necessary since
we consider all $\epsilon > 0$.

The distinction between $\TauR$ and $\tau_+^0$ depends only on the behavior of $\vva$ at $0$, and can be understood by considering the case where the infima
in \eqref{eq:defn:tau0+} and in  \eqref{eq:defn:tauL-tauR} are actually minima (the infima are attained).   In that case,
we see that $\tau_+^0$ can be thought of as the smallest
``right-knot'' in the sense that $\vva$ has a strictly negative slope to the right of $\tau_+^0$.
And $\TauR$ can be thought of as the smallest positive knot.  Note that (by concavity) all positive knots are right-knots, so that $\TauR \ge \tau_+^0$.
Note that
 the infimum
defining $\TauR$ in  \eqref{eq:defn:tauL-tauR} is taken over knots that are strictly larger than $0$, so that (when the infima are attained)  we have $\TauR > 0$.
On the other hand,
if $0$ is a right-knot then $\tau_+^0 = 0$, so that then $\TauR$ and $\tau_+^0$ are distinct.  If $0$ is not a right-knot, then we will have $\tau_+^0 = \TauR$.
These statements are slightly complicated by the fact that $\tau_+^0$ and $\TauR$ are defined as infima rather than minima, but the intuitive differences are captured by the previous description.
Corresponding statements hold for $\TauL$ and $\tau_-^0$.

The distinction between
the two sets of knots pairs is important
because
many of our arguments depend on constructing ``perturbations'' of $\vva$, and we can use different types of perturbations at each pair.  This means that the different knot pairs have different properties:
if we replace $\TauL,
\TauR$ by $\tau_-^0, \tau_+^0$ in \eqref{eq:CharznFequalsX_M_full}, then
that display may not hold, and similarly, if we replace $\tau_-^0, \tau_+^0$
by $\TauL, \TauR$ in \eqref{eq:CharznEquality2_full}, then that display
may not hold.
The following lemma holds for $\tau_-^0, \tau_+^0$ but not
necessarily for $\TauL, \TauR$.

\begin{lemma}
  \label{lem:process-uniqueness-theorem-constant-interval}
  With the definitions and assumptions as in
  Theorem~\ref{thm:MC-process-uniqueness-theorem},
  \begin{equation}
    \label{eq:vva-modal-interval}
    (\vva)'(t) = 0 \mbox{ for } t \in  (\tau^0_-, \tau^0_+).
  \end{equation}
\end{lemma}
\medskip

Now we introduce the appropriate limit processes for the symmetric about $0$
mode-constrained estimators.  The characterization is similar to that for the mode-constrained (but not symmetric) processes, but since it is defined only on $[0,\infty)$ the processes are not the same.

\begin{mynotes}
  \begin{notes}
    We decide to always let $0 \in \widehat{S}^+$.  Explicitly excluding $0$ in the definition of $(\widehat{S}^+)^c$ (including $0$ in the set of knots) is somewhat redundant since by definition $(\psia)'(0-)$ is not defined.  But we make it explicit.

    This definition seems like the appropriate analog for the two-sided case (in the two-sided mode-constrained case we may have $0$ as a ``knot'' even though it is not a right-knot.  This entails only that $(F_L-X_L)(0)=0$, which is effectively tautologically true in this one-sided setting).  Thus, just as in the two-sided case, we do not generally have that $\tau_+^+ = \inf \widehat{S}^+ \cap [0,\infty)$.
  \end{notes}
\end{mynotes}
\begin{theorem}
  \label{thm:symm-process-uniqueness-theorem}
  Assume $\HHst$ is a random process
  on $[0,\infty)$,
  and assume that $\psia\equiv
  (\HHst)^{(2)}  \in {\cal C}_0$. Define
  \begin{equation*}
    (\widehat{S}^+)^c \equiv S(\psia)^c
    = \lb t \ge 0 :  (\psia)^{(2)}(t\pm) = 0\rb \setminus \lb 0 \rb,
  \end{equation*}
  \begin{equation*}
    \Tausp = \inf \lb t \in \widehat{S}^+ : (\psia)^\prime (t+\epsilon+) < 0 \text{ for all } \epsilon > 0 \rb, \qquad \text{ and }
  \end{equation*}
  \begin{equation*}
    \TausR = \inf \widehat{S}^+ \cap (0, \infty).
  \end{equation*}
For $t \ge 0$, let $W(t)$ be a one-sided standard Brownian motion with $W(0)=0$,
let $X(t) = W(t) - 4t^3$, and let
\begin{equation*}
  \YYsp(t) = \int_{\TausR}^t \int_{\TausR}^u dX(v) du.
\end{equation*}
Suppose that $\TausR < \infty$ and \\
(i)
\begin{equation}
  \label{eq:12}
  \int_0^{\TausR} (\psia(u)du - dX(u)) = 0,
\end{equation}
(ii)
\begin{equation}
  \label{eq:15}
  \HHsp(t) - \YYsp(t) \le 0 \quad \text{ for all } t \ge 0,
\end{equation}
(iii)
\begin{equation}
  \label{eq:16}
          \int_{[\Tausp, \infty)} ( \HHsp - \YYsp ) d(\psia)' = 0.
\end{equation}
Then  $\HHsp$ is unique.
\end{theorem}
\begin{mynotes}
  \begin{notes}
    I am suspecting that $\Tausp$ and $\TausR$ are equal in the symmetric case.  In the non-symmetric case, the discrepancy is caused by having the mode be a knot which is a left-knot but not a right-knot, or vice versa.  This cannot happen in the symmetric case, because of symmetry.  However, proving this requires proving that $\psia$ almost surely has a flat modal region of positive length.  So it is simpler to just continue using $\Tausp$ and $\TausR$.
  \end{notes}
\end{mynotes}

\subsection{Unconstrained and constrained pointwise limit theory at $x_0 \ne m$}
\label{subsec:limit-theory}

The two main limit theorems below will concern the limiting distributions of our estimators and their derivatives.
Recall, we assume that $X_1, \ldots, X_n \iid f_0 = e^{\vp_0}$, where $f_0$ is a non-degenerate density on $\RR$.
The three sets of estimators of $f_0$, $\vp_0$, $f_0'$, and $\vp_0'$ to be considered are:\\
A  $\widehat{f}_n$, $\widehat \varphi_n$, $(\widehat{f}_n)'$, and $(\widehat{\varphi}_n)'$.\\
B  $\widehat{f}_n^0$, $\widehat \varphi_n^0$, $(\widehat{f}_n^0)'$, and $(\widehat{\varphi}_n^0)'$.\\
C  $\widehat{g}_n^0$, $\widehat \psi_n^0$, $(\widehat{g}_n^0)'$, and $(\widehat{\psi}_n^0)'$.\\
Then the  corresponding curvature assumptions are: \\
Curvature Assumption 1.  $\vp_0''(x_0) < 0$, where $x_0 \in \interior \lb x: f_0(x) > 0 \rb$. \\  %
Curvature Assumption 2a. $\vp_0''(x_0) < 0$ with $x_0 \ne m$ and
$x_0 \in \interior \lb x: f_0(x) > 0 \rb$. \\  %
Curvature Assumption 2b: $\vp_0''(m) < 0$. \\
Note that
\cite{MR771431} %
shows that Curvature Assumption 2b holds for the class of
symmetric $\alpha$-stable densities on $\RR$ for all $0 < \alpha < 2$.
Assumption 1 will be used for the estimators in A, whereas for the estimators in B and C we will use Assumptions 2a and 2b.
In all three cases we assume $X_1, \ldots, X_n \iid f_0$.

To state our theorem we first define some constants as follows:
\begin{align*}
    c(x_0,\varphi_0) = \left(
      \frac{f_0(x_0)^3 |\varphi_0^{(2)}(x_0)|}{4!}
    \right)^{1/5},
    \qquad
    d(x_0,\varphi_0) =  \left(
      \frac{f_0(x_0)^4 |\varphi_0^{(2)}(x_0)|^3}{(4!)^3}
    \right)^{1/5}
  \end{align*}
  \begin{align*}
    C(x_0,\varphi_0) = \left(
      \frac{ |\varphi_0^{(2)}(x_0)|}{f_0(x_0)^2 4!}
    \right)^{1/5},
    \qquad
    D(x_0,\varphi_0) = \left(
      \frac{ |\varphi_0^{(2)}(x_0)|^3}{f_0(x_0) (4!)^3}
    \right)^{1/5}.
\end{align*}

\begin{theorem} (Limiting distributions at a fixed point $x_0 \ne m$)
  \label{thm:BRW2009-UMLE-limit} \\
  A.  \cite{BRW2007LCasymp}.  Suppose that $f_0 \in \mathcal{LC}$ and that
  the Curvature Assumption 1 holds at $x_0$.  Then with $H$ as in
  Theorem~\ref{thm:charzn_uniqueness_full_UC} and $\vv \equiv H''$,
  \begin{equation}
    \label{eq:19}
    \begin{split}
      \left(\begin{array}{c}
              n^{2/5} (\widehat{f}_n(x_0) - f_0(x_0)) \\
              n^{1/5} (\widehat{f}_n^{\prime}(x_0) - f_0^{\prime}(x_0)) \\
              n^{2/5} (\widehat \varphi_n(x_0) - \varphi_0(x_0)) \\
              n^{1/5} (\widehat \varphi_n'(x_0) - \varphi_0'(x_0)) \\
            \end{array}\right)
          \overset{d}{\rightarrow}
          \left(\begin{array}{c}
                  c(x_0,\varphi_0) \vv(0) \\
                  d(x_0,\varphi_0) \vv'(0) \\
                  C(x_0,\varphi_0) \vv(0) \\
                  D(x_0,\varphi_0) \vv'(0) \\
                \end{array}\right),
    \end{split}
  \end{equation}
  B.  Suppose that $f_0 \in \mathcal{LC}_m$ and that the Curvature Assumption 2a holds at \\
  $\phantom{blab}$$x_0$ with $x_0 \not= m$.  Then \eqref{eq:19} continues to hold with\\
  $\phantom{blab}$$\widehat{f}_n$, $\widehat \varphi_n$, $(\widehat{f}_n)'$, and $(\widehat{\varphi}_n)'$ replaced by
  $\widehat{f}_n^0$, $\widehat \varphi_n^0$, $(\widehat{f}_n^0)'$, and $(\widehat{\varphi}_n^0)'$. \\
  C.  Suppose that $f_0 \in \mathcal{SLC}_0$ and that the Curvature Assumption 2a  holds at \\
  $\phantom{blab}$$x_0$ with $x_0 \not= 0$.
  Then \eqref{eq:19} continues to hold with  \\
  $\phantom{blab}$$\widehat{f}_n$, $\widehat \varphi_n$, $(\widehat{f}_n)'$, and $(\widehat{\varphi}_n)'$ replaced by
  $\widehat{g}_n^0$, $\widehat{\psi}_n^0$, $(\widehat{g}_n^0)'$, and $(\widehat{\psi}_n^0)'$; \\
  $\phantom{blab}$ and $c(x_0, \varphi_0)$, $d(x_0, \varphi_0)$, $C(x_0, \varphi_0)$, and $D(x_0, \varphi_0)$ replaced by \\
  $\phantom{blab}$$2^{-2/5} c(x_0, \vp_0)$, $2^{-1/5} d(x_0, \vp_0)$, $2^{-2/5} C(x_0, \vp_0)$, and $2^{-1/5} D(x_0 , \vp_0 )$.
\end{theorem}

\begin{remark}
\label{ImprovementAtInteriorPoint}
(i) Comparing the MLE's for ${\cal LC}$ and ${\cal
  LC}_m$: Note that the limiting distributions in A and B at a point $x_0
\not= m$ are the same.  At a fixed point $x_0 \not=
m$, the constraint that the mode is known does not help in estimating
the function at $x_0$.  As we will see below, this picture changes when $x_0 = m$.  \\
(ii) Note that the rate of convergence of $\ffna$ and $\ggna$ as $x_0 \ne m$ (or $x_0 \ne 0$ in the case of $\ggna$) is $n^{-2/5}$  in contrast to the $n^{-1/5}$ rate achieved by the naive estimators discussed in Subsection~\ref{subsec:naive-estimators}.\\
(iii) Comparing the MLE's for ${\cal LC}$ and ${\cal
  LC}_m$ with the MLE for ${\cal
  SLC}_m$: The limiting distributions for the symmetric log-concave class
${\cal
  SLC}$ in C are smaller than the limiting distributions of the MLE's for the
possibly asymmetric log-concave classes ${\cal LC}$ and ${\cal
  LC}_m$ by a factor of $2^{-2/5} \approx
.757858\ldots$ for the functions themselves and by a factor of $2^{-1/5}
\approx .870551\ldots
$ for the derivatives of the functions.  Thus the symmetry constraint substantially reduces
the variability of the estimators (see also Figure~\ref{fig:plots:empirical-processes}).
\end{remark}

\subsection{Mode-constrained and symmetry-constrained pointwise limit theory at $x_0 = m$}
\label{subsec:limit-theory-MC}

The limit distribution of the mode-constrained estimators at a point $x_0$
depends on whether $x_0 = m$ or $x_0 \ne m$.  In the latter case the
asymptotics are the same as the unconstrained estimator, but in the former
case they depend on the mode-constrained limit process.
\begin{theorem}(Limiting distributions at $m$)  \label{thm:MC-MLE-limit}
  \begin{enumerate}[label=\Alph*.,ref=\Alph*,leftmargin=*]
    \setcounter{enumi}{1}
  \item \label{thm:MC-MLE-limit:item-A}
    Let Curvature Assumption 2b
    hold.
    Let constants $c(x_0, \varphi_0), d(x_0, \varphi_0), C(x_0,\varphi_0),$ and
    $ D(x_0, \varphi_0)$ be as given in Theorem~\ref{thm:BRW2009-UMLE-limit}.
    Let $\vva$ be as in  Theorem~\ref{thm:MC-process-uniqueness-theorem}.  Then
    \begin{equation}
      \label{thm:mode-ff-asymptotics}
    \begin{split}
      \left(
        \begin{array}{c}
          n^{2/5} (\widehat{f}^0_n(m) - f_0(m)) \\
          n^{1/5} ((\widehat{f}^0_n)^{\prime}(m) - f_0^{\prime}(m)) \\
          n^{2/5} (\widehat \varphi_n^0(m) - \varphi_0(m)) \\
          n^{1/5} ((\widehat \varphi_n^0)'(m) - \varphi_0'(m))
        \end{array}
      \right)
      \overset{d}{\rightarrow}
      \left(
        \begin{array}{c}
          c(m,\varphi_0) \vva(0) \\
          d(m,\varphi_0) (\vva)'(0) \\
          C(m,\varphi_0) \vva(0) \\
          D(m,\varphi_0) (\vva)'(0)
        \end{array}
      \right).
    \end{split}
  \end{equation}

\item \label{thm:MC-MLE-limit:item-B} Suppose that $f_0 \in \SLC_0$ and that
  the Curvature Assumption 2b holds at $m=0$. Then
  \begin{equation}
    \label{thm:mode-ff-asymptotics}
    \begin{split}
      \left(
        \begin{array}{c}
          n^{2/5} (\widehat{g}^0_n(0) - f_0(0)) \\
          n^{2/5} (\widehat \psi_n^0(0) - \psi_0(0)) \\
        \end{array}
      \right)
      \overset{d}{\rightarrow}
      \left(
        \begin{array}{c}
          2^{-2/5} c(0,\varphi_0) \psia(0) \\
          2^{-2/5} C(0,\varphi_0) \psia(0) \\
        \end{array}
      \right)
    \end{split}
  \end{equation}  %
  where $\psia$ is given by Theorem~\ref{thm:symm-process-uniqueness-theorem}.
  \end{enumerate}
\end{theorem}

\medskip

We label the parts of Theorem~\ref{thm:MC-MLE-limit} as ``B'' and ``C''
 to be in parallel with the labeling in Theorem~\ref{thm:BRW2009-UMLE-limit}.
Notice that for the symmetric MLE, $((\widehat{g}^0_n)^{\prime}(0)$
and $((\widehat \psi_n^0)'(0)$ are always both equal to $0$, the value of $f_0'(0)$ and of 
$\vp_0'(0)$, so we do not state a limit theorem for these estimators.
Theorems~\ref{thm:BRW2009-UMLE-limit} and \ref{thm:MC-MLE-limit} follow from more 
general theorems about the estimators not just at $x_0$ but in local $n^{-1/5}$ neighborhoods of $x_0$,
stated below as
Theorems~\ref{thm:MC-MLE-limit-A-process-version}
and \ref{thm:process-asymptotics-symmetry}.
The ``local neighborhood'' 
Theorem~\ref{thm:MC-MLE-limit-A-process-version}
is the version from which we can derive the limit distribution of the mode likelihood 
ratio statistic $2 \log \lambda_n$ studied in \cite{Doss-Wellner:2016ModeInference}.
Monte Carlo estimates of the distribution functions of $\vv'(0)$ and of $(\vva)'(0)$ are 
presented in  Figure~\ref{fig:limitdistns} (left plot).  %
Note that $(\widehat{\varphi}^0)^{\prime} (0)$ is stochastically
smaller than $\widehat{\varphi}^{\prime} (0)$.

\subsection{Local process limit  theory, mode- and symmetry-constrained}
\label{subsec:process-limit-theory-MC}

Here we state the local process limit theorems behind Theorem~\ref{thm:MC-MLE-limit} \ref{thm:MC-MLE-limit:item-A},
where $x_0 = m$.  In this subsection we will formulate a more general version of %
that theorem
which applies to our estimators in $n^{-1/5}$ neighborhoods of $m$.

Recall the definition of $Y$ in \eqref{eq:defn_Xs}. Now, for positive numbers $a$ and $\sigma$, Let
\begin{eqnarray}
&& Y_{a,\sigma}(t) \equiv \sigma \int_0^t W(s) ds - a t^4 \stackrel{d}{=}  \sigma (\sigma/a)^{3/5} Y( (a/\sigma)^{2/5} t) , \label{eq:defn:Y-a-sigma} \\
&& Y_{a,\sigma}^{(1)} (t) = \sigma W(t) - 4a t^3 \stackrel{d}{=}  \sigma (a/\sigma)^{1/5} Y^{(1)} ( (a/\sigma)^{2/5} t).\label{eq:defn:X-a-sigma}
\end{eqnarray}
Let $H_{a,\sigma}, H_{L, a, \sigma}$, and $H_{R, a, \sigma}$ denote the unconstrained and mode-constrained left- and right-processes for $Y_{a ,\sigma}$.  Then
\begin{eqnarray*}
&& H_{a,\sigma}(t)   \stackrel{d}{=}  \sigma (\sigma/a)^{3/5} H( (a/\sigma)^{2/5} t) , \\
&& H_{a,\sigma}^{(1)} (t)  \stackrel{d}{=}  \sigma (\sigma / a)^{1/5} H^{(1)} ( (a/\sigma)^{2/5} t),
\end{eqnarray*}
and
\begin{eqnarray}
  \widehat{\varphi}_{a,\sigma}
  & = & H_{a,\sigma}^{(2)}  \stackrel{d}{=}  \sigma^{4/5} a^{1/5} H^{(2)} ((a/\sigma)^{2/5} \cdot )
  \label{eq:H2-ScalingRelationUnConstrained-1}
\end{eqnarray}
Identical scaling relationships hold for
$H_{L,a, \sigma}, H_{R, a, \sigma}$, and the corresponding derivatives, including $\vva_{a, \sigma} \equiv H^{(2)}_{R,a , \sigma}$:
\begin{eqnarray}
  \widehat{\varphi}^0_{a,\sigma}
  & = & H_{R, a,\sigma}^{(2)}  \stackrel{d}{=}
  \sigma^{4/5} a^{1/5} H_R^{(2)} ((a/\sigma)^{2/5} \cdot ) .
  \label{eq:H2-ScalingRelationMC-1}
\end{eqnarray}

\begin{theorem}
  \label{thm:MC-MLE-limit-A-process-version}
  Let Curvature Assumption~2b hold.
  Let $H$ be as in Theorem~\ref{thm:charzn_uniqueness_full_UC} and let $\vv \equiv H''$ 
  and let $\vva$ be as in Theorem~\ref{thm:MC-process-uniqueness-theorem}.  Let
$\sigma \equiv 1/\sqrt{f_0 (m)}$ and  $a = | \varphi_0^{(2)} (m) | /4!$. Then
  \begin{equation*}
    \left(\begin{array}{c}
        n^{2/5} (\widehat \varphi_n(m + n^{-1/5} t) - \varphi_0(m)) \\
        n^{1/5} (\widehat \varphi_n'(m + n^{-1/5} t) - \varphi_0'(m)) \\
        n^{2/5} (\widehat \varphi_n^0(m + n^{-1/5} t) - \varphi_0(m)) \\
        n^{1/5} ((\widehat \varphi_n^0)'(m + n^{-1/5} t) - \varphi_0'(m)) \\
      \end{array}\right)
    \overset{d}{\rightarrow}
    \left(\begin{array}{c}
        \vv_{a,\sigma}(t) \\
        \vv_{a,\sigma}'(t) \\
        \vva_{a,\sigma}(t) \\
        (\vva_{a,\sigma})'(t) \\
      \end{array}\right)
  \end{equation*}
  as processes in $(\mc{C}_\infty \times \mc{D}_\infty)^2$, where
  $\mc{C}_\infty$ is the set of continuous functions on $(-\infty,\infty)$
  with the topology of uniform convergence on compact sets and
  $\mc{D}_\infty$ is the set of right-continuous with limits from the left
  (``cadlag'') functions on $(-\infty,\infty)$ with the topology of $M_1$
  convergence on compacta.  (The $M_1$ topology is discussed in detail in
  Subsection~\ref{subsec:proofs-pointwise-limit-theory}).
\end{theorem}

A similar useful result for the symmetry constrained problem which generalizes or extends
Theorem~\ref{thm:MC-MLE-limit} part~\ref{thm:MC-MLE-limit:item-B} is as follows.

\begin{theorem}
  \label{thm:process-asymptotics-symmetry}
  Suppose that  the Curvature Assumption~2b holds and  $f_0 \in \SLC_0$.
  Let $\sigma \equiv 1 / \sqrt{2 f_0(0)}$, and $a \equiv | \vvo^{(2)}(0)| / 4!$.  
  We let $Y_{a,\sigma}$ and $Y_{a,\sigma}^{(1)}$ be as in \eqref{eq:defn:Y-a-sigma} 
  and \eqref{eq:defn:X-a-sigma}, and $H^+_{a,\sigma}(|t|)$ be the corresponding symmetry-constrained 
  process for $t \in \RR$ as in Theorem~\ref{thm:symm-process-uniqueness-theorem}, and let 
  $\psia_{a,\sigma}(|t|) \equiv (H^+_{a,\sigma})''(|t|)$.
 Then
 \begin{equation}
   \label{eq:17}
   \begin{split}
        \begin{pmatrix}
     n^{2/5} ( \psina(n^{-1/5} t  ) - \vvo(0)) \\
     n^{1/5} (  (\psina)'(n^{-1/5} t) - \vvo'(0) )
   \end{pmatrix}
   \to_d
   \begin{pmatrix}
     \psia_{a,\sigma}(|t|) \\
     (\psia_{a,\sigma})^\prime(|t|)
   \end{pmatrix}
   \end{split}
 \end{equation}
 as processes in $ \mc{C}_\infty \times \mc{D}_\infty $ with the topology of uniform convergence on
 compacts on $\mc{C}_\infty$ and the $M_1$ topology on $\mc{D}_\infty$.
\end{theorem}

\begin{remark}
  To this point we have focused on the case in which the point of symmetry 
  $m$ is known (and equal to $0$).  If $m$ is unknown (and possibly different 
  from $0$) and $f_0 \in {\cal SLC}_m$, then it is well-known that $m$ is 
  also the mean and median of $f_0$ and hence it can be estimated in several 
  different ways by estimators $\hat{m}$ satisfying $\sqrt{n} (\hat{m} - m) = O_p (1)$.  
  For example, we could take $\hat{m} = \overline{X}_n$ or $\hat{m} = \FF_n^{-1} (1/2)$, 
  the sample median.  
  Then we can proceed by assuming that $f_0 \in {\cal SLC}_{\hat{m}}$ and carrying 
  out the estimation as described above with the $X_i$'s shifted by $\hat{m}$.  
  Denote the resulting estimators of $g$ and $\psi$ by $\hat{\hat{g}}_n^0$ 
  and $\hat{\hat{\psi}}_n^0$.   Then since  $n^{-1/2} = o(n^{-2/5})$
  it is easily seen that that Theorems
  \ref{thm:BRW2009-UMLE-limit} C,
  \ref{thm:MC-MLE-limit} \ref{thm:MC-MLE-limit:item-B} and
  \ref{thm:process-asymptotics-symmetry} continue to hold with $\hat{g}_n^0$ 
  replaced by $\hat{\hat{g}}_n^0$ and $\hat{\psi}_n^0$ replaced by $\hat{\hat{\psi}}_n^0$.
\end{remark}

\subsubsection{Asymptotics for the maximum}

We now consider the asymptotic distribution of estimators of
the maximum functional $N(f) \equiv \sup_{x \in \RR} f(x)$.   The maximum functional is 
of interest, for instance, in \cite{MR1345204} (see page 872).  An estimate of $N(f_0)$ is 
needed for estimation of functionals of the form $\int f_0^{k}(x) dx$.
In the mode constrained case where $f_0 \in \LC_0$, $n^{2/5} (N(\ffna) - N(f_0)) =
n^{2/5}( \ffna(m) - f_0(m))$.  In the symmetry-constrained case with $f_0 \in \SLC_0$,
$n^{2/5} ( N(\wh{g}_n^0) - N(f_0)) = n^{2/5} ( \wh{g}_n^0(0) - f_0(0))$. Thus the asymptotic
distribution in those two cases is given by Theorem~\ref{thm:MC-MLE-limit}.  
We present the asymptotic distribution in the unconstrained case here.
In fact, in the unconstrained case, we can present a somewhat stronger result, 
where we allow the possibility of increasingly flat modal regions.
\nocite{romano1987bootstrapping}

\begin{theorem}
  \label{thm:max-functional}
  Let $W$ denote two-sided Brownian motion starting at $0$, and define 
  $Y_k(t) = \int_0^t W(s) ds - t^{k+2}$ for $k \ge 2$ an even integer.  Let $H_k$ be the lower invelope of 
  $Y_k$, as defined in Theorem~2.1 of \cite{BRW2007LCasymp}, and let $\vv_k \equiv H_k^{(2)}$.  
  Suppose $f_0 \in \LC$ and that $\varphi_0^{(j)}(m)=0$ for $j=2, \ldots, k-1$, $\vp_0^{(k)}(m) < 0$ and 
  $\vp_0^{(k)}$ is continuous in a neighborhood of $m$.  Then
  \begin{equation*}
    \left (
      \begin{array}{c}
        n^{1/(2k+1)} (\widehat{M}_n - M(f_0)) \\
        n^{k/(2k+1)} (\widehat{N}_n - N(f_0))
      \end{array} \right )
    \rightarrow_d
    \left (
      \begin{array}{c} e_k(\vp_0)  M( \vv_k ) \\
        c_k (\varphi_0 ) N(\vv_k)
      \end{array}
    \right )
  \end{equation*}
  where
  \begin{align*}
    e_k (\varphi_0)^{2k+1}
    = \left ( \frac{(k+2)!^2}{f_0 (m) | \varphi_0^{(k)} (m ) |^2 } \right ),
    \quad
    c_k (\varphi_0 )^{2k+1}
    = \left (\frac{f_0 (m)^{k+1} | \varphi_0^{(k)} (m ) | }{(k+2)!}  \right ).
  \end{align*}
\end{theorem}

\medskip

\noindent
We compared, by Monte Carlo, the densities of  $N(\vv)$  (note $\vv \equiv \vv_2$) and of $N(\vva)$.  
See estimates in Figure~\ref{fig:limitdistns} (right plot); those estimates are log-concave 
MLE's (based on Monte Carlo simulations, as described in the caption).  
Additionally, simulations not presented here indicate that $P( |N(\vva)| \le t) \ge P( |N(\vv)| \le t)$ 
for all $ t \ge 0$ (i.e., $|N(\vva)|$ is stochastically smaller than $|N(\vv)|$).
The estimated density of $N(\vv)$ in
Figure~\ref{fig:limitdistns}
should be compared with the estimated density of $\vv(0)$ given in Figure~1 of
\cite{MR3178372},
noting that their $\mathbb{C}(0)$ is our $\vv(0)$.

\begin{remark}
  As is well known, one can see from Theorem~\ref{thm:max-functional} that the rate of convergence 
  of the mode decreases as $k$ increases, while the rate of convergence of the maximum increases 
  (and gets closer to $n^{1/2}$).  One can also check that $e_k(\log f) \searrow 1$ and 
  $c_k(\log f) \nearrow 1$ when $f(x) = C_k \exp (- |x|^{k}/k)$ where $k$ is an even integer.
\end{remark}
\begin{figure}
  \caption{\protect\includegraphics[scale=.55]{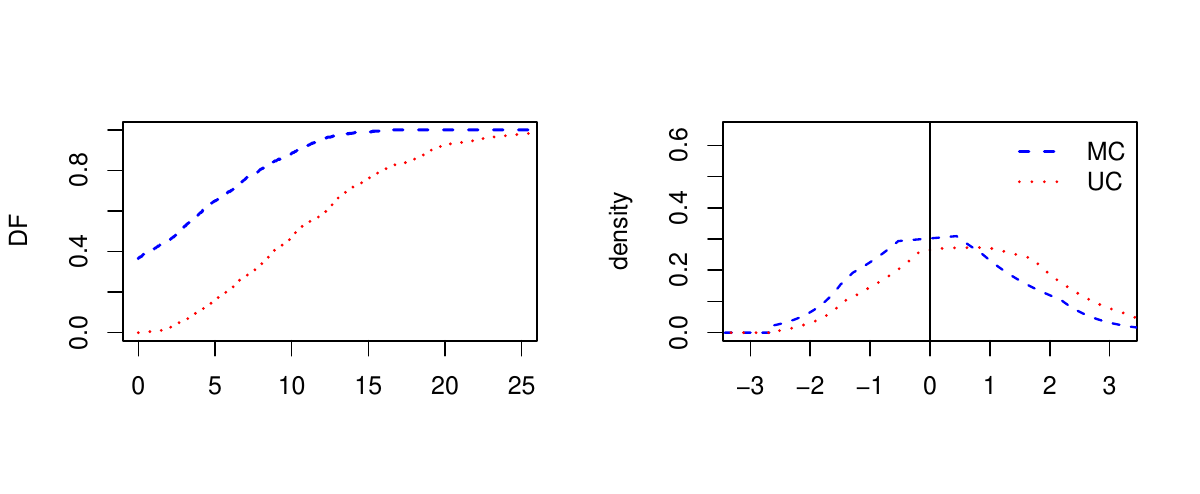}
    Monte Carlo estimates of limit distributions for estimation of the derivative and the maximum.  
    The left plot gives Monte Carlo estimates of distribution functions of $|\vv'(0)|$  and $|(\vva)'(0)|$.   
    The right plot gives density estimates of $N(\vv)$ and $N(\vva)$.
    The number of Monte Carlos used is $10^3$, each with sample size $10^4$.  
    Here, we sampled from $N(0,1)$ (although this is inconsequential).
 \label{fig:limitdistns}
}
\end{figure}

\section{Simulation results}
\label{sec:FiniteSample}

Software to compute the mode-constrained estimator, and also to implement the
likelihood ratio test and corresponding confidence intervals studied in
\cite{Doss-Wellner:2016ModeInference}, %
is available in the package \verb+logcondens.mode+
\citep{doss:logcondens.mode} in R \citep{R-core}.  Here we illustrate the
existence and characterization results on simulated data in
Figure~\ref{fig:plots:gamma-norm}.  There are two columns of four plots.  The
left column includes the mode-constrained log-concave MLE.  The right column
includes the $0$-symmetric log-concave MLE.  The data points are represented
by vertical hash lines along the bottom of each plot.  The density, log
density, and distribution function are plotted in the top three rows, with
the unconstrained log-concave MLE in red
and the true (unknown) function in black.
On the left the mode-constrained MLE is in blue, and on the right the $0$-symmetric MLE is in blue.
The empirical df $\Fn$ is plotted in green in the third row.
In the last row,
we plot $\YYnL-\HHnaL $ (blue) and $\YYnR-\HHnaR$ (purple) to illustrate Theorem~\ref{thm:CharThmTwo}~\ref{thm:CharThmTwo:B-MC} (left plot), the corresponding symmetry-constrained process (blue) to illustrate 
Theorem~\ref{thm:CharThmTwo}~\ref{thm:CharThmTwo:C-symm} (right plot),
and $\YY_n- \widehat{H}_n$ in red (both plots)
to illustrate Theorem~\ref{thm:CharThmTwo}~\ref{thm:CharThmTwo:A-UC}.
In all the plots, dashed vertical red lines give $S_n(\vvn)$
and dashed vertical blue lines give %
knots of the constrained estimator %
(which frequently overlap).
The solid blue line is the specified mode value for the mode-constrained MLE.

Figure~\ref{fig:plots:empirical-processes} gives plots of $\sqrt{n}( \GGna - F_0)$  (``SC''), $\sqrt{n} (\FFna - F_0)$ (``MC''), $\sqrt{n} (\FFn - F_0)$ (``UC''), and $\sqrt{n}(\Fn - F_0)$ (``E'').  The left and right plot are each one simulation with sample size
$n=200$ and $n=2000$, respectively, from a $N(0,1)$ distribution.
The plots show improvements by $\GGna$ and $\FFna$ over $\FFn$.
The MC and UC lines are indistinguishable when $n=2000$ since one needs to
plot locally to the mode $0$ to see differences between $\FFna$ and $\FFn$ when $n$ is large.

\begin{figure}
  \caption{\protect\includegraphics[scale=.58]{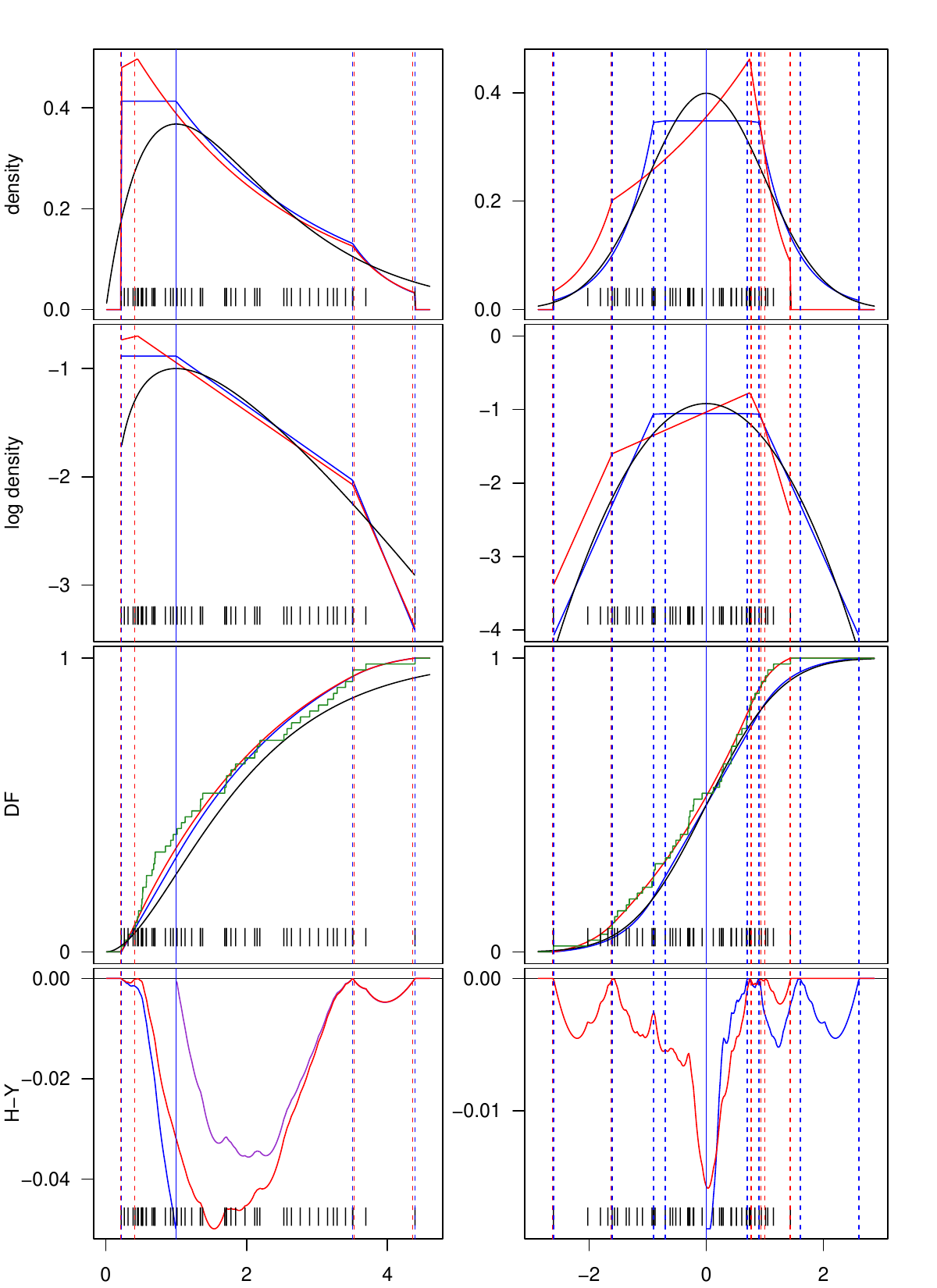} Left: Log-concave MLE and mode-constrained MLE of Gamma with density $xe^{-x}$, $x \ge 0$, with
    $n=50$ and $m=1$ well-specified.  Right: Log-concave MLE and $0$-symmetric MLE of $N(0,1)$ with $n=50$. \label{fig:plots:gamma-norm} }
\end{figure}

\begin{figure}
  \caption{\protect\includegraphics[width=.58\textheight,height=.4\textheight]{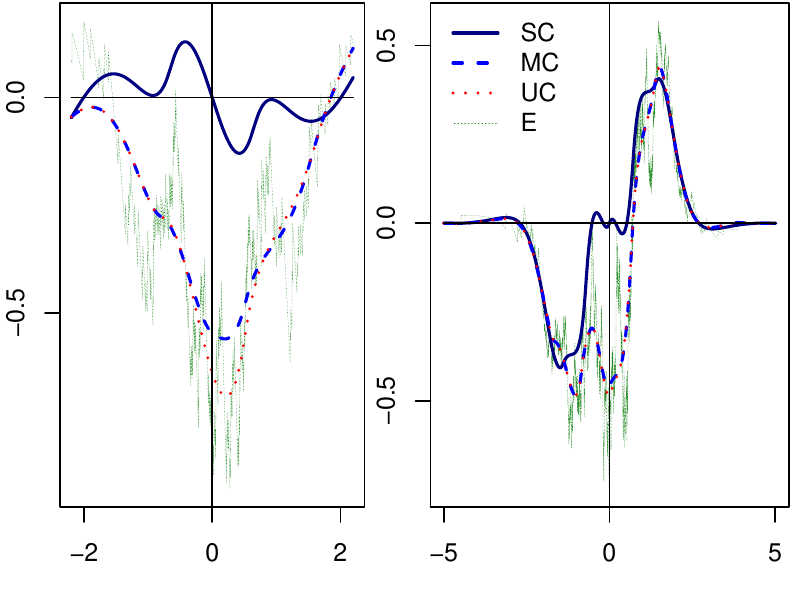}
    Empirical processes for the symmetry constrained (SC) log-concave MLE, mode constrained (MC) log-concave MLE, unconstrained (UC) log-concave MLE, and empirical (E) distribution.  We used a sample of $n=200$ (left) and $n=2000$ (right) from a $N(0,1)$ distribution. \label{fig:plots:empirical-processes} }
\end{figure}

\section{Outlook and further problems}
\label{sec:Outlook}

Motivated by likelihood ratio test considerations as well as potential uses in several semiparametric settings, we have introduced estimators for a log-concave density known to satisfy a further constraint of either having a known mode or of being symmetric (about $0$).  Our estimators are based on the maximum likelihood principle.  The constrained MLE's that we develop are more challenging to compute and to study theoretically than certain na\"ive estimators discussed in Subsection~\ref{subsec:naive-estimators}, but have much better theoretical behavior.  We developed a fast algorithm for computation of the estimators which is made available in the \verb+logcondens.mode+ package for the R programming language.  We found that the constrained MLE's are consistent and indeed we presented the $n^{-2/5}$ rate of convergence, globally and locally (with some proofs given in \cite{Doss:2013}).  We found the pointwise asymptotic distribution of the MLE's, in fact; this necessitated studying and characterizing certain limit processes that govern the limit distributions of the MLE's.  Studying the limit processes in the constrained cases seems to be somewhat more challenging than in the unconstrained case (i.e., than in the case given in Theorem~\ref{thm:charzn_uniqueness_full_UC}), because the definitions and characterizing conditions for the constrained limit processes depend on certain knots (of the limit process) in complicated ways.  Nonetheless, our proofs of Theorem~\ref{thm:MC-process-uniqueness-theorem} and Theorem~\ref{thm:symm-process-uniqueness-theorem} are different and shorter than the proof of Theorem~\ref{thm:charzn_uniqueness_full_UC}; for the latter, \cite{MR1891741} initially characterized the process on an interval $[-c,c]$ and then through further tightness-type arguments they showed that one can let $c \to \infty$.  In the proofs of Theorem~\ref{thm:MC-process-uniqueness-theorem} and Theorem~\ref{thm:symm-process-uniqueness-theorem}, we argue directly about a process on $(-\infty,\infty)$, skipping the step of considering the process on $[-c,c]$ and allowing for more direct proofs of the results.

The following are interesting questions beyond those already posed in the introduction that are motivated by the present work.
\medskip

\par\noindent
(a)  One motivation for the study of $\widehat{f}_n^0$ given here has been the likelihood ratio
tests and confidence intervals for the mode introduced in
\cite{Doss-Wellner:2016ModeInference}. %
But the constrained estimators may be of interest for the study of semiparametric two- and $k-$sample problems
with (constrained) log-concave errors.
For example suppose that $X_1, \ldots , X_m$ are i.i.d. with $X_i \stackrel{d}{=} \mu + \epsilon_i$, $i=1,\ldots , m$, while
$Y_1, \ldots , Y_n$ are i.i.d. with  $Y_j \stackrel{d}{=} \nu + \delta_j$ where $\mu, \nu \in \RR$ and
$\epsilon_i , \ \delta_j$ are i.i.d. with log-concave density $f$ with mode at $0$.  Other variants of this problem might involve
constraining $f$ to be log-concave with mean or median at $0$ rather than mode at $0$.
Constraining $f$ to be symmetric about its mode of $0$ and log-concave, as in
\cite{BalabdaouiDoss:2017}, is also of interest.
\medskip

\par\noindent
(b)
In \cite{BalabdaouiDoss:2017}, a mixture density $g_0(x) = \sum_{i=1}^k \pi_i f_0(x-\mu_i)$ is estimated where $f_0$ is constrained to be in $\SLC_0$, $k=2$, $\mu_1<\mu_2$, and $\pi_1 = 1- \pi_2 \notin \lb 0, 1/2, 1 \rb$.
Restriction to the case $k=2$ is made for identifiability reasons.
Can the asymptotic distribution theory we developed in the present paper be extended to the MLE of $f_0$ (and of $g_0$) in the semiparametric mixture setting?   Extensions to the case $k > 2$ could also  be possible and would certainly be interesting.

\section{Proof sketches and outlines}
\label{sec:Proofs-outlines}

In this section we give some outlines of the proofs of the results in
Section~\ref{sec:LimitDistributions}.  Full proofs are given in
Section~\ref{sec:proofs-new}. Here we outline the material in each subsection of
Section~\ref{sec:proofs-new}.

\bigskip

\par\noindent {\em Subsection~\ref{ssec:LocalRatesProofs}:}

The main goal
of Subsection~\ref{ssec:LocalRatesProofs} is to show the following
proposition.
\begin{proposition}
  \label{cor:local:vp_vpprime_tight}
  Suppose either Assumption~2a holds (at $x_0 \ne m$) or Assumption~2b holds (at $x_0 = m$). Let $C > 0$ and let $I_n = [x_0 - C n^{-1/5} , x_0 + C n^{-1/5} ]$.
  Then
  \begin{equation}
    \label{eq:vpprime_tight}
    \sup_{t \in I_n}
    \left| (\widehat{\varphi}_n^0)'(t) - \varphi_0 '(x_0) \right| = O_p(n^{-1/5}),
  \end{equation}
  where $(\widehat{\varphi}_n^0)^\prime $ denotes either the right or left derivative,
  and
  \begin{equation}
    \label{eq:vp_tight}
    \sup_{t \in I_n}
    \left| \widehat{\varphi}_n^0 (t) - \varphi_0 (x_0) - (t-x_0) \varphi_0'(x_0) \right| = O_p(n^{-2/5}).
  \end{equation}
  We may replace the interval $I_n$ by
  $[\xi_n - C n^{-1/5} , \xi_n + C n^{-1/5} ]$ for any $\xi_n \to m$.  Then
  the random variables implied by the $O_p$ upper bounds depend on $C$ but
  not on $\xi_n$.
\end{proposition}

\par\noindent
This proposition is of crucial importance for showing the results in Subsection~\ref{subsec:limit-theory-MC} (Theorems~\ref{thm:MC-MLE-limit},
\ref{thm:MC-MLE-limit-A-process-version},
and \ref{thm:process-asymptotics-symmetry}).
The proof of Proposition~\ref{cor:local:vp_vpprime_tight} depends on
the following two propositions.
\begin{proposition}
  \label{cor:UniformConsistencyConstrained}
  Let $K = [b,c]$ be a closed interval contained in the interior of the support of $f_0 \in {\cal LC}_m$.
  Then
  \begin{align*}
    \sup_{t \in K} | \widehat{\varphi}_n^0 (t) - \varphi_0 (t) | \rightarrow_{a.s.} 0
    \; \text{ and }  \;
    \sup_{t \in K} | \widehat{f}_n^0 (t) - f_0 (t) | \rightarrow_{a.s.} 0 \ \ \
    \text{ as } \ \ n \rightarrow \infty.
  \end{align*}
\end{proposition}
\medskip

\begin{proposition}
  \label{prop:uniform-local-tightnessKnots}
  Suppose that Curvature Assumption~2b
  holds so $\vvo''(m)<0$.\\
  Let $\tau_+^0(\xi_n)$  denote the smallest
  knot of $\widehat{\varphi}_n^0$ strictly greater than $\xi_n$,
  and let $\tau_{-}^0(\xi_n)$  denote the largest
  knot of $\widehat{\varphi}_n^0$ strictly smaller than $\xi_n$.
  Then for all $\epsilon>0$ there exists $C^0> 0$ such that for any random variables $\xi_n \rightarrow_p m$
  \begin{equation*}
    P( n^{1/5} (\tau_+^0 (\xi_n) -  \tau_-^0 (\xi_n)) \ge C^0 ) \le \epsilon
  \end{equation*}
  for $n \ge N^0$ for some $N^0$.
  The integer  $N^0$ may depend on $\xi_m$ but  $C^0$
  does not depend on $\xi_m$.
\end{proposition}
\medskip

\par\noindent
Proposition~\ref{cor:UniformConsistencyConstrained}
is proved in \cite{Doss:2013}.  It is needed to show
Proposition~\ref{prop:uniform-local-tightnessKnots}, which
is needed to prove
Proposition~\ref{cor:local:vp_vpprime_tight}.
The proof of Proposition~\ref{cor:local:vp_vpprime_tight} depends on
Proposition~\ref{prop:uniform-local-tightnessKnots}, finding points of closeness of $\vvn$ and $\vvna$, and properties of convex functions.
The full proof of
Proposition~\ref{cor:local:vp_vpprime_tight}
is given in
\cite{Doss:2013} (see Corollary~4.2.7 there).
Thus the main goal of Subsection~\ref{ssec:LocalRatesProofs} is to prove
Proposition~\ref{prop:uniform-local-tightnessKnots} about the ``gap problem''
(a term coined in \cite{BW2007kmono}) for the constrained MLE near $m$.
The proof depends on constructing certain (somewhat complicated) classes of perturbation functions which can be related to  $\tau_+^0 (\xi_n) -  \tau_-^0 (\xi_n)$, and
then applying an argument pioneered by
\cite{MR1041391} %
to these perturbations (see Lemma~\ref{lem:gapEPT_NK_LK}).

\bigskip

\par\noindent {\em Subsections~\ref{ssec:LocalRatesProofs} and \ref{ssec:proofs_symm-lim-process}:}

Subsection~\ref{ssec:LocalRatesProofs} is devoted
to the proof of
Theorem~\ref{thm:MC-process-uniqueness-theorem}.  The proof proceeds by
defining an ``objective function''
\begin{equation}
  \phi_{a,b}(g)= \inv{2} \int_{a}^{b} g^2(t)dt - \int_{a}^{b} g(t)dX(t)
\end{equation}
for $a < 0 < b$.
Now, $\vva$ is not a minimizer of this objective function (which has finite bounds of integration) but we can show that $\vva$ behaves in a sense as  if it were a minimizer of $\phi_{a,b}$ over a the space of concave functions with mode at $0$. We show
that, for certain values $a<0<b$, directional derivatives   of $\phi_{a,b}$,
$    \int_{a}^{b} \Delta(t) (\vva(t)dt - dX(t)),$
in the direction of a function $\Delta$ (assumed to be concave with mode at $0$),
are either nonnegative or in some cases are zero (see
Propositions~\ref{pro:charzn0_full}
and \ref{pro:charznEqualityRewrite_full} for exact statements).  This allows us to argue as follows.  We want to show the uniqueness of any $\vp$ satisfying the characterizing conditions
of Theorem~\ref{thm:MC-process-uniqueness-theorem}.  We assume
there exist  $\vp_1$ and $\vp_2$ both satisfying the characterizing conditions.
We examine
$\phi_{a_2,b_2}(\vp_1) - \phi_{a_2,b_2}(\vp_2)$
and
$\phi_{a_1,b_1}(\vp_2) - \phi_{a_1,b_1}(\vp_1)$
where $a_i,b_i$ are certain knot points for $\vp_i$.
By Propositions~\ref{pro:charzn0_full} and
\ref{pro:charznEqualityRewrite_full}, we are able to show that both of these differences are
no smaller than $\int_{-n}^n ( \vp_1(t)-\vp_2(t))^2 dt > 0$ for $n>0$ related to the knot points $a_i,b_i$.  On the other hand, after deriving results about the knots $a_i,b_i$ and relating the processes $\vp_i$ to the ``observed'' process $Y$ (see
Lemma~\ref{lem:unq:Hrepresentation}), we are able to show that
$\phi_{a_2,b_2}(\vp_1) - \phi_{a_2,b_2}(\vp_2)$
and
$\phi_{a_1,b_1}(\vp_2) - \phi_{a_1,b_1}(\vp_1)$ are also nonpositive, by using properties of Brownian motion.  Thus the difference $\int_{-n}^n  ( \vp_1(t)-\vp_2(t))^2 dt = 0$. By letting $n \to \infty$, we see $\vp_1 = \vp_2$ so the proof is complete.
The proof is somewhat complicated by the fact that the ``knots'' of the concave function $\vva$ are not separated but rather could be a complicated ``Cantor-type'' set,
as described in \cite{Sinaui:1992wl}, and so their behavior requires careful analysis.
The proof of Theorem~\ref{thm:symm-process-uniqueness-theorem} in
Subsection~\ref{ssec:proofs_symm-lim-process} follows a similar argument.

\bigskip

\par\noindent
{\em Subsection~\ref{subsec:proofs-pointwise-limit-theory}:}
\label{subsec:proof-sketch-outl}

Theorem~\ref{thm:BRW2009-UMLE-limit} %
A %
about the unconstrained estimator is Theorem~2.1 of \cite{BRW2007LCasymp}.
Part B of the theorem
is then proved in an identical fashion as that theorem, because for $n$ large enough, in an $n^{-1/5}$ neighborhood of $x_0 \ne m$, the constrained and unconstrained estimators satisfy the same characterization.
Part C then follows from Part B. %
The main focus of
Subsection~\ref{subsec:proofs-pointwise-limit-theory} is
to show
Theorem~\ref{thm:MC-MLE-limit-A-process-version}.
Theorem~\ref{thm:process-asymptotics-symmetry} follows in a similar fashion.
Theorem~\ref{thm:MC-MLE-limit-A-process-version} proves
Theorem~\ref{thm:MC-MLE-limit} \ref{thm:MC-MLE-limit:item-A},
which we show now.
A similar argument shows that
Theorem~\ref{thm:process-asymptotics-symmetry} proves
Theorem~\ref{thm:MC-MLE-limit}~\ref{thm:MC-MLE-limit:item-B}.

\begin{proof}[Proof of Theorem~\ref{thm:MC-MLE-limit}~\ref{thm:MC-MLE-limit:item-A}] %
  Note that, now with $\sigma = 1/\sqrt{f_0 (m)}$ and $a = | \varphi_0^{(2)} (m) | /4!$, %
  \eqref{eq:H2-ScalingRelationMC-1} equals
  \begin{eqnarray}
    \frac{1}{\gamma_1 \gamma_2^2} H_R^{(2)} ( \cdot / \gamma_2)
    \equiv
    \frac{1}{\gamma_1 \gamma_2^2} \widehat{\varphi}^0 (\cdot / \gamma_2),  %
    \label{ScalingRelationMC}
  \end{eqnarray}
  where
  \begin{eqnarray}
    && \gamma_1 = \left ( \frac{f_0(m)^4 | \varphi_0^{(2)} (m) |^3}{(4!)^3}  \right )^{1/5} = \frac{1}{\sigma} \left ( \frac{a}{\sigma} \right )^{3/5} ,
        \label{GammaDefnsLogConcaveAtMode-1}\\
    && \gamma_2 = \left ( \frac{(4!)^2}{f_0 (m) | \varphi_0^{(2)} (m) |^2 } \right )^{1/5}  = \left ( \frac{\sigma}{a} \right )^{2/5},
    \label{GammaDefnsLogConcaveAtMode}
  \end{eqnarray}
  and we note that
  \begin{eqnarray}
    && \gamma_1 \gamma_2^{3/2} = \sigma^{-1} = \sqrt{f_0 (m)},  \ \ \ \gamma_1 \gamma_2^4 = a^{-1} = \frac{4!}{| \varphi_0^{(2)} (m)|} ,
    \label{GammaRelationsPart1} \\
    && \inv{\gamma_1 \gamma_2^2} = C(m, \varphi_0) \equiv \left ( \frac{4! f_0 (m)^2}{| \varphi_0^{(2)} (m)|} \right )^{-1/5} .
    \label{GammaRelationsPart2}
  \end{eqnarray}
  This gives the constant in the limit distribution for $\vvna(m)$.  For $(\vvna)'(m)$, we see
from \eqref{ScalingRelationMC} that
  \begin{eqnarray}
    (\vva_{a,\sigma})' =_d \frac{1}{\gamma_1 \gamma_2^3} (\vva)' ( \cdot / \gamma_2),
    \label{ScalingRelation-MC-2}
  \end{eqnarray}
  and $1 / \gamma_1 \gamma_2^3 = D(m, \vvo)$.
  Thus
  Theorem~\ref{thm:MC-MLE-limit} \ref{thm:MC-MLE-limit:item-A}
  follows
  from  Theorem~\ref{thm:MC-MLE-limit-A-process-version}.
\end{proof}

\label{proof-4.8-outline} Next we briefly outline the proof of Theorem~\ref{thm:MC-MLE-limit-A-process-version}.
We define two sets of localized processes.
We need to  define left-side and right-side processes; for ease of exposition,  here we only discuss right-side processes.
We let $\tnb \equiv m + b n^{-1/5}$.
We
let $\snR$ be any knot (sequence) of $\vvna$ strictly larger than $m$ satisfying $n^{1/5}(\snR-m)=O_p(1)$.
The first set of localized processes is the ``$f$-processes'' (written in terms of the densities and the empirical distribution):
\begin{align*}
  \YYnffaR(b) & \equiv \xn^{4/5} \IntRo \left(\IntRi (d\Fn -  \ffo(m) d\lambda) \right) dv, \\
  \HHnffaR(b) &\equiv \xn^{4/5} \IntRo \left( \IntRi (\ffna -
    \ffo(m))d\lambda \right) dv
  + \AnR n^{1/5}(\tnb - \snR),
\end{align*}
where
$  \AnR  = \xn^{3/5}\left(\FnR(\snR)-\FFnaR(\snR)\right) $, and $\AnR n^{1/5}(\tnb-\snR)$ is shown to be asymptotically negligible.
These processes are needed because we can show that
\begin{equation}
  \label{eq:20}
  \YYnffaR(\cdot) \Rightarrow
  \sqrt{f_0(m)} \int_{\nu_{n,R}}^{\cdot} \int_{\nu_{n,R}}^v dW dv
  - f_0^{(2)}(m) \int^\cdot_{\nu_{n,R}} \int_{\nu_{n,R}}^ v \frac{u^2}{2} du dv
\end{equation}
where
$\nu_{n,R} = n^{1/5} (s_{n,R} - m)$ and $W$ is a standard Brownian motion with $W(0)=0$ (see Lemma~\ref{lem:MC-dataprocs-ff-limit}).
Furthermore, the
characterization from Theorem~\ref{thm:CharThmTwo}~\ref{thm:CharThmTwo:B-MC} applies
to $\YYnffaR$ and $\HHnffaR$, so $\YYnffaR(b) -\HHnffaR(b) \ge 0$ for $b>0$, with equality at certain points
(see Lemma~\ref{lem:local-f-characterization}).

On the other hand, the tightness proposition from above,
Proposition~\ref{cor:local:vp_vpprime_tight},  applies to the log-densities.  Thus we define the second set of processes, the ``$\vp$-processes'', written in terms of log-densities:
\begin{align*}
    \YYnvvaR(b) & = \frac{\YYnffaR(b)}{\ffo(\mm)}
  - \xn^{4/5}\IntRo \IntRi \RRnaR(u)dudv,  \\
    \HHnvvaR(b) &=  \xn^{4/5} \IntRo \IntRi \left( \vvna(u)-\vvo(m) \right)du dv
    + \frac{\AnR n^{1/5} (t_{n,b} - \snR) }{\ffo(\mm)}
    \label{eq:local:HHnvvaR2HHnffaR}
\end{align*}
where $R_n^0$ is a remainder term.
These are related to the $f$-processes by a Taylor expansion (the delta method).
One can translate the characterizing inequalities from the $f$-processes to the $\vp$-processes, to see that $\YYnvvaR(b) - \HHnvvaR(b) \ge 0$ for $b \ge 0$ with equality at certain points. By \eqref{eq:20},
$\YYnvvaR$ can be shown to converge to a limiting Gaussian process.
Furthermore, we can apply
Proposition~\ref{cor:local:vp_vpprime_tight} and the Arzela-Ascoli theorem (after analyzing various remainder terms) to see that $\HHnvvaR$ is tight (Lemma~\ref{lem:local:Htight}).

Finally, we make a subsequence argument using tightness.  By tightness,
Prohorov's theorem, and the Skorokhod construction, for any subsequence we
can find a further subsubsequence that converges almost surely to a limit
process.  By using the characterization (Lemma~\ref{lem:Y-bigger-H-vv}, and
by analysis of various remainder terms), we show that the limit process must
satisfy the unique characterizing conditions given by
Theorem~\ref{thm:MC-process-uniqueness-theorem}.  This shows the limit is the
same along all subsequences and so the unique process $\vva$ given in
Theorem~\ref{thm:MC-process-uniqueness-theorem} is the limit, which completes the proof.
The argument showing that the characterizing conditions pass from the finite sample processes to the limit process is somewhat complicated by the fact that the integrands in question are  defined to begin at random knot points.

Another issue of note is that $(\vvna)'$ is a discontinuous function.  We must choose or find an appropriate metric space in which to study its convergence; the metric we choose is the so-called $M_1$ Skorokhod metric, which differs from the (perhaps more standard) $J_1$ Skorokhod metric
(referred to as ``the'' Skorokhod metric in chapter 12 of \cite{Billingsley1999CPM}).
The $J_1$ metric unfortunately does not allow multiple jumps to approximate a single jump, whereas the $M_1$ metric does.  Since we do not have a proof that multiple jumps of $(\vvna)'$ do not approximate a single jump in the limit, we must use the $M_1$ metric.
See Lemma~\ref{prop:FccM-precompact} and the preceding text for further discussion.

\section{Proofs}
\label{sec:proofs-new}

\subsection{Proofs for Section~\ref{sec:MLEs}}
Proofs of
Theorems~\ref{thm:Existence}~\ref{thm:existence:A-unconstrained}, %
\ref{thm:CharThmOne}~\ref{thm:CharThmOne:A-UC}, %
\ref{thm:CharThmTwo}~\ref{thm:CharThmTwo:A-UC} %
and
Corollaries~\ref{cor:EstimatedFNearlyTouchesEDFatKnotsUnconstrained} A %
and
\ref{cor:MeanVarianceInequalities} A %
may be found in \cite{MR2459192},  %
\cite{RufibachThesis},  %
and \cite{DR2009LC}.   %
Proofs of parts B and C of
Corollary~\ref{cor:EstimatedFNearlyTouchesEDFatKnotsUnconstrained} %
and part B of  Corollary~\ref{cor:MeanVarianceInequalities}  %
follow from the corresponding parts of Theorems
\ref{thm:CharThmOne}
and \ref{thm:CharThmTwo},
much as in the unconstrained case, A.
It remains to prove parts B and C of Theorems
\ref{thm:Existence},
\ref{thm:CharThmOne},
and \ref{thm:CharThmTwo}.

In the following proofs, we let ${\cal C}_{n,m}$ denote the (random) class of concave functions with knots at the $Z_i$'s and support
on $[X_{(1)}, X_{(n)}]$, and let ${\cal K}_{n,m}$ denote the class of concave functions $\varphi$ with knots at
the $Z_i$'s and where $e^{\varphi}$ is a density with support $[X_{(1)}, X_{(n)}]$.

\par\noindent
\begin{proof}[Proof of Theorem~\ref{thm:Existence}]

  Proof of Theorem~\ref{thm:Existence}~B:
  As
  in the unconstrained case (Theorem 2.1 of \cite{DR2009LC}), it is easy to argue that the solution is
  piecewise linear with knots at the $Z_i$'s, and that
  $\widehat{\varphi}_{n}^0$ is flat either directly to the left of the mode
  or directly to the right of the mode as long as the mode is not a data
  point.  If the mode is equal to one of the $X_i$'s then the proof
  given by \cite{RufibachThesis} for the unconstrained MLE existence applies directly.
  Thus assume the mode is not one of the $X_i$.
  Consider  a sequence $\lb v_j \rb$ which has limit coordinates $\gamma = (\gamma_1, \ldots, \gamma_N)$ which may be $\pm \infty$.  Let $\vp_\gamma$ be the piecewise linear function given by linearly interpolating $\gamma$.    Then $\vp_\gamma$ has a flat modal region on $[Z_i, Z_j]$ for some $i < j$.  Since $\int e^{\vp_\gamma(x)}dx = 1$, we have $\vp_\gamma(m) \le \log(1/ (Z_j-Z_i))$.  Since $m$ is the mode of $e^{\vp_\gamma}$ no coordinate of $\gamma$ is positive infinity, if one of the coordinates is $-\infty$ then $\Psi_n(\vp_\gamma) = -\infty$.  This shows we can consider the continuous function $\Psi_n$ on a compact set so it achieves a maximum.  The proof that if $\widehat{\psi}_n$ maximizes $\Psi_n$ over ${\cal C}_m$ then $\int e^{\widehat{\psi}_n(x)} dx = 1$ as well as the proof that $\widehat{\psi}_n$ is unique are as in the unconstrained case (\cite{RufibachThesis}).

  For Theorem~\ref{thm:Existence}~C, note that
  \begin{equation*}
    \argmax_{g \in \SLC_0} \inv{n} \sum_{i=1}^n \log g(X_i)
    =     \argmax_{g \in \SLC_0} \inv{n} \sum_{i=1}^n \log g(|X_i|)
    =     \argmax \inv{n} \sum_{i=1}^n \log g(|X_i|)
  \end{equation*}
where the last argmax is over log-concave functions with mode at $0$ and which integrate to $1/2$ on $[0,\infty)$.
\end{proof}

\par\noindent
The proof of Theorem~\ref{thm:CharThmOne}~\ref{thm:CharThmOne:B-MC} is standard.
See \cite{Doss:2013}.
For the  proof of Theorem~\ref{thm:CharThmTwo}~B we need to introduce a certain
 \emph{cone} $C \subseteq \RR^d$ (defined, e.g., on page 13 of
\cite{MR0274683}). We say the cone $C$ is \emph{ (finitely) generated} by a set $\{ b_i
\in C$ : $i=1,\ldots,k < \infty \}$ if for all $c \in C$ we can write $c =
\sum_{i=1}^{k} \alpha_i b_i $ for some nonnegative numbers $\alpha_i \ge 0$.
Let  $(x)_{-}=\min(x,0)$ and
$(x)_{+}=\max(x,0)$. %
\begin{proposition}
  \label{prop:charzn:CnaCone}
  ${\cal C}_{n,m}$ is a convex cone with finite generating set given by
  \begin{equation*}
    \{(x-Z_{i})_{-}\}_{2\le    i\le k}\,\bigcup\,\{(Z_{i}-x)_{-}\}_{k\le i\le N-1}\,\bigcup\{\pm1\}.
  \end{equation*}
\end{proposition}
\begin{proof}
  It is clear that ${\cal C}_{n,m}$ is a cone because concavity and the mode are
  preserved under positive scaling. For $\varphi\in {\cal C}_{n,m}$, with
  $\varphi'(Z_{i}-)=\sum_{j=i}^{k}a_{j}$ for $i\le k$, with,
  $\varphi'(Z_{i}+)=\sum_{j=k}^{i}b_{j}$ for $i\ge k$, and with $\varphi(m)=C$, we
  can write
  $\varphi(x)=C+\sum_{i=2}^{k}a_{i}(x-Z_{i})_{-}+\sum_{i=k}^{n-1}b_{i}(Z_{i}-x)_{-}$.
\end{proof}

\begin{proof}[Proof of Theorem~\ref{thm:CharThmTwo}]
Proof of Part B: First we assume $\widehat{\varphi}_n^0$ is the MLE and use
\eqref{eq:char0} to show that \eqref{eq:LeftFenchelInequalitiesConstrained}
and \eqref{eq:RightFenchelInequalitiesConstrained} hold.  Using the
generating functions described in Proposition~\ref{prop:charzn:CnaCone} as
our $\Delta$ yields equations \eqref{eq:LeftFenchelInequalitiesConstrained}
and \eqref{eq:RightFenchelInequalitiesConstrained} via integration by
parts. That is, for $t \le m$, we choose $\Delta(x)=\left(x-t\right)_{-}$
(which is concave with $m$ as a mode).  Then integration by parts yields
$-\int_{X_{(1)}}^{t}F(x)dx=\int_{X_{(1)}}^{X_{(n)}}(x-t)_{-}dF(x)$, by
Lemma~\ref{lem:integration-by-parts}~\ref{lem:integ-by-parts-closed-interval},
for $F$ equal to either $\FF_n$ or $\widehat{F}_n^0$ since
$\Fn(X_{(1)}-) = \FFna(X_{(1)}-) = 0$.  Thus, by our initial characterization
\eqref{eq:char0}, we get \eqref{eq:LeftFenchelInequalitiesConstrained}.
Similarly, for $t\ge m$, let $\Delta(x)=\left(t-x\right)_{-}$; this yields
\begin{align*}
  \MoveEqLeft{\left(t-X_{(n)}\right)\FF_n(X_{(n)}) -\int_{t}^{X_{(n)}}\FF_n(x)d(-x)}\\
& \le  \left(t-X_{(n)}\right)\widehat{F}_n^0 (X_{(n)})-\int_{t}^{X_{(n)}}\widehat{F}_n^0(x)d(-x),
\end{align*}
and, recalling that we have already shown $\widehat{F}_n^0(X_{(n)})=1$, this is equivalent to
$
\int_{t}^{X_{(n)}}\FF_n(x)dx\le \int_{t}^{X_{(n)}}\widehat{F}_n^0 (x)dx,
$
so we have
  \eqref{eq:RightFenchelInequalitiesConstrained}.
  We get equality at some knot points also: set
  $\Delta(x)=\left(x-b\right)_{+}$ where $b\ge m$ is any RK. Then, by the definition of a
  RK, $\Delta$ is an allowable perturbation because
  $\widehat{\varphi}_n^0 (b+)-\widehat{\varphi}_n^0(b-)>0$ so for some
  $\delta$ small enough, $\widehat{\varphi}_n^0+\delta\Delta$ is still concave with mode at $m$.  Using
  this $\Delta$ we get
  \begin{align*}
    \MoveEqLeft \FF_n (X_{(n)})\left(X_{(n)}-b\right)-\int_{b}^{X_{(n)}}\FF_n(x)dx  \\
    &= \int_{X_{(1)}}^{X_{(n)}}(x-b)_{+}d \FF_n(x) \le \int_{X_{(1)}}^{X_{(n)}}(x-b)_{+}d\widehat{F}_n^0 (x) \\
    &=  \widehat{F}_n^0 (X_{(n)}) \left(X_{(n)}-b\right)-\int_{b}^{X_{(n)}}\widehat{F}_n^0 (x) dx, %
  \end{align*}
  so that $\int_{b}^{X_{(n)}}\FF_n(x)dx \ge\int_{b}^{X_{(n)}}\widehat{F}_n^0 (x)dx$, and thus for any
  $b\ge m$ that is a RK we have the inequality both ways,
  $\int_{b}^{X_{(n)}}\FF_n(x)dx=\int_{b}^{X_{(n)}}\widehat{F}_n^0 (x)dx$.
  We have thus shown that
  \eqref{eq:LeftFenchelInequalitiesConstrained} and
  \eqref{eq:RightFenchelInequalitiesConstrained} hold
  with the appropriate equalities.

  Now we will show the reverse implication. We assume
  \eqref{eq:LeftFenchelInequalitiesConstrained}
  and \eqref{eq:RightFenchelInequalitiesConstrained} hold and consider $\Delta$
  with support $[X_{(1)},X_{(n)}]$ and piecewise linear with knots at
  the $Z_{i}$.  These are all the $\Delta$'s we need to consider,  %
  since the rest were ruled out previously by elementary considerations.  We
  also need $\widehat{\varphi}_n^0 +\epsilon\Delta$ to be concave with mode $m$.
  Now, we do not know if $m$  will be a NK or a LK or
  a RK, so we argue by cases. If $m$ is a knot for $\widehat{\varphi}_n^0 $ in one
  direction, without loss of generality, we can say that $m$ is a RK and
  we have $\int_{m}^{c}\FF_n(x)dx=\int_{m}^{c}\widehat{F}_n^0 (x)dx$ for any $c>m$ that
  is also a knot. Recall that we  have defined the indices
  $1=j_{1},\ldots,j_{l^0}=N$ so that $Z_{j_{i}}$ are the knots. We  write
  \begin{equation}
    \Delta'(r-)=\sum_{i=2}^{l}-C_{i}1_{[Z_{j_{i}-1}<r\le
      Z_{j_{i}}]}+\sum_{j=j_{i-1}+1}^{j_{i}}\beta_{j}1_{[Z_{j_{i}-1}<r\le
      Z_{j}]}\label{eq:deltaDecomposition}
  \end{equation}
  with $\beta_{j}\ge0$ and all $C_{i}\ge0$.   %
  Since $m$ is a RK,
  $m$ is not also a LK (otherwise $m$ is simply a knot and
  $\widehat{\varphi}_n^0$ coincides with the unconstrained MLE and the characterization
  of the unconstrained MLE in \citep{DR2009LC} implies we are done).
  This forces $C_{p}=0$ (which refers to the interval $(Z_{j_{p}-1}, m=Z_{j_{p}}]$).
  We thus have
  \begin{align*}
    \int\Delta d\FF_n
    & =  \Delta(X_{(n)})-\left[\sum_{i=2}^{l}
      -C_{i}\int_{Z_{j_{i}-1}}^{Z_{j_{i}}}\FF_n\left(x\right)dx +
      \sum_{j=j_{i-1}+1}^{j_{i}}\beta_{j}\int_{Z_{j_{i}-1}}^{Z_{j}}\FF_n\left(x\right)dx\right]\\
    & \le \Delta(X_{(n)})-\left[\sum_{i=2}^{l}-C_{i}\int_{Z_{j_{i}-1}}^{Z_{j_{i}}} \widehat{F}_n^0 \left(x\right)dx
      +\sum_{j=j_{i-1}+1}^{j_{i}}\beta_{j}\int_{Z_{j_{i}-1}}^{Z_{j}}\widehat{F}_n^0 \left(x\right)dx\right]\\
    & =  \int\Delta d\widehat{F}_n^0
  \end{align*}
  as desired, where the inequality follows from noting that
  $  -\beta_{j}\int_{Z_{j_{i}-1}}^{Z_{j}}\FF_n(x) dx \le -\beta_{j}\int_{Z_{j_{i}-1}}^{Z_{j}}\widehat{F}_n^0 (x)dx
  $
  by assumption and because $\beta_{j}\ge0$. We also have
  $$
  C_{i}\int_{Z_{j_{i}-1}}^{Z_{j_{i}}}\FF_n\left(x\right)dx  =  C_{i}\int_{Z_{j_{i}-1}}^{Z_{j_{i}}}\widehat{F}_n^0 \left(x\right)dx
  $$
  for all $i$ except for $i=p$, by the equality-at-knots assumption.
  However, for $i=p,$ we have $C_{i}=0$.
  An analogous argument holds for the case where $m$ is an LK
  and for the case where $m$ is neither an LK nor an RK.

  Part C follows as in the proof of Theorem~\ref{thm:Existence}~C;
  $\widehat{g}_n^+$ is the MLE over $\LC_0$ of $|X_1|, \ldots, |X_n|$.  Thus
  we apply the result of Part B.  Note that
  $0 \in S_n(\widehat{\psi}_n^0) \cap [0, |X|_{(n)}]$ only if
  $(\widehat{\psi}_n^0)'(0+)<0$ (so that $0$ is a right knot).  This is only
  possible if $0$ is an observed data point.
\end{proof}

Proof of Corollary~\ref{cor:EstimatedFNearlyTouchesEDFatKnotsUnconstrained}~B and C follow as in the unconstrained case (Corollary~2.5, \cite{DR2009LC}).

\subsection{Proofs for local rates of convergence, limit processes, and limit theory}

\subsubsection{Proofs for local rates of convergence} %
\label{ssec:LocalRatesProofs}

This section is devoted to showing,
Proposition~\ref{cor:local:vp_vpprime_tight},
stated above, which is needed for
proving Theorems~\ref{thm:MC-MLE-limit-A-process-version},
and \ref{thm:process-asymptotics-symmetry}.
See \ref{sec:Proofs-outlines} for a discussion of the proof of
Proposition~\ref{cor:local:vp_vpprime_tight}, the full details of which are
given in \cite{Doss:2013} (Corollary~4.2.7).
Our main goal here is to
prove Proposition~\ref{prop:uniform-local-tightnessKnots}.

\begin{proof}[Proof of Proposition~\ref{prop:uniform-local-tightnessKnots}] %
  We first consider the case  $\xi_{n} = m$. %
  For ease of notation and without loss of generality, we assume $m=0$
  and abbreviate $\tau_{n,+}^0 (0)$  by  $\tau_{n,+}^0$  and $\tau_{n,-}^0 $ by
  $\tau_{n,-}^0$.
  We will argue via a family of perturbations which can
  be separated into subfamilies, depending on whether $0$ is a left-knot
  (LK), $0$ is a right-knot (RK), or $0$ is not a knot (NK).
  If $0$ is a one-sided knot (LK or RK), we have different perturbation subfamilies
  depending on whether $\tau_{n,+}^0 >- \ \tau_{n,-}^0$  or not. We will start
  with the case in which $0$ is a LK and we construct $\Delta$ that
  has the two properties
  \begin{equation}
    \int_{\tau_{n,-}^0}^{\tau_{n,+}^0}\Delta(t)dt=0,\label{eq:GP_LKfilter1}
  \end{equation}
  \begin{equation}
    \int_{\tau_{n,-}^0}^{0} t \Delta(t)dt=0.\label{eq:GP_LKfilter2}
  \end{equation}
  Whereas in the unconstrained case construction of such an acceptable
  perturbation function was straightforward \citep{BRW2007LCasymp}, in the
  constrained case, construction of such a $\Delta$ that is an acceptable
  perturbation (i.e.\ keeps the mode fixed) is much less straightforward.
  We consider several cases separately.\\
  {\bf Case 1.}  $\tau_{n,+}^0 > - \tau_{n,-}^0$.
  In this case we define
  \begin{eqnarray}
    \Delta_{LK,0}(t)=\begin{cases}
      \frac{\tau_{n,+}^0 +m_{2}\cdot(-\frac{\tau_{n,-}^0}{4})}{\frac{\tau_{n,-}^0}{4}-\tau_{n,-}^0}\cdot(t-\tau_{n,-}^0) ,
      & \,\,\tau_{n,-}^0 \le t \le \frac{\tau_{n,-}^0}{4} \\
      \tau_{n,+}^0 + m_{2}\cdot(-t),  & \,\,\frac{\tau_{n,-}^0}{4}\le t\le0 \\
      (\tau_{n,+}^0 -t),  & \,\,0\le t\le\tau_{n,+}^0 \\
      0,  & \mbox{otherwise,}  \end{cases}  \label{LKDeltaZeroPart1}
  \end{eqnarray}
  where
  $$
  m_{2}:=m_{2}(\tau_{n,-}^0,\tau_{n,+}^0)
  =\left( \frac{-9-3\frac{\tau_{n,+}^0}{-\tau_{n,-}^0}}{1 - 5\frac{\tau_{n,+}^0}{-\tau_{n,-}^0}} \right)
  \left(\frac{\tau_{n,+}^0}{-\tau_{n,-}^0}\right) .
  $$
  This function has integral
  $$
  \frac{(\tau_{n,+}^0)(5\tau_{n,+}^0 -\tau_{n,-}^0)(\tau_{n,+}^0 -\tau_{n,-}^0)}{2(5\tau_{n,+}^0 + \tau_{n,-}^0)}
  \equiv M_{LK,case1} (\tau_{n,+}^0 -\tau_{n,-}^0) .
  $$
  {\bf Case 2:}  $\tau_{n,+}^0 < - \tau_{n,-}^0$.  In this case we define
  \begin{eqnarray}
    \Delta_{LK,0}(t)=\begin{cases}
      (t-\tau_{n,-}^0),  & \,\,\tau_{n,-}^0 \le t \le  \frac{\tau_{n,-}^0}{2} \\
      \frac{\tau_{n,-}^0}{2}-\tau_{n,-}^0 + m_{2} \cdot (\frac{\tau_{n,-}^0}{2} - t),  & \,\,\frac{\tau_{n,-}^0}{2}\le t\le0 \\\
      \left(\frac{-\frac{\tau_{n,-}^0}{2} + m_{2}\frac{\tau_{n,-}^0}{2}}{\tau_{n,+}^0 } \right) (\tau_{n,+}^0 - t),  & \,\,0\le t\le\tau_{n,+}^0 \\
      0,  & \mbox{otherwise,}\end{cases} \label{LKDelta0Part2}
  \end{eqnarray}
  where
  $$
  m_{2}:=m_{2}(\tau_{n,-}^0,\tau_{n,+}^0)
  =\frac{2\tau_{n,-}^0 + \tau_{n,+}^0}{2\tau_{n,-}^0 - 5\tau_{n,+}^0 }
  $$
  is defined so that \eqref{eq:GP_LKfilter2} holds.
  This function has integral
  $$
  \frac{-\tau_{n,-}^0 (3\tau_{n,+}^0 - \tau_{n,-}^0)(\tau_{n,+}^0 - \tau_{n,-}^0)}{10\tau_{n,+}^0 - 4\tau_{n,-}^0}
  \equiv M_{LK,case2} (\tau_{n,+}^0 - \tau_{n,-}^0) .
  $$
  Then we define
  \begin{equation}
    \Delta_{LK,1}(t) = \Delta_{LK,0}(t)-M_{LK} 1_{[\tau_{n,-}^0,\tau_{n,+}^0 ]}(t),
    \label{eq:Delta1_LK}
  \end{equation}
  where
  $$
  M_{LK}=\frac{(\tau_{n,+}^0)(5\tau_{n,+}^0 -\tau_{n,-}^0)}
  {2(5\tau_{n,+}^0 + \tau_{n,-}^0 )}1_{[\tau_{n,+}^0 >\tau_{n,-}^0]}
  +  \frac{-\tau_{n,-}^0(3\tau_{n,+}^0 - \tau_{n,-}^0)}
  {10\tau_{n,+}^0 - 4\tau_{n,-}^0}1_{[\tau_{n,+}^0 \le\tau_{n,-}^0]},
  $$
  is $o_{p}(1)$ by uniform consistency of $\widehat{\varphi}_n^0$ and is the appropriate
  shift so that \eqref{eq:GP_LKfilter1} holds.
  Then $\Delta_{LK,0}$
  is an acceptable perturbation for $\widehat{\varphi}_n^0$, since we can have
  $m_{2}>1$ when $0$ is a LK, and $\Delta_{LK,1}$ has the properties
  \eqref{eq:GP_LKfilter1} and \eqref{eq:GP_LKfilter2}.  We also define
  $\Delta_{RK,1}$ analogously as $\Delta_{LK,1}$, with analogous constant  $M_{RK}$.

  Now consider the case in which $0$ is not a knot (NK). In this case,
  because $(\widehat{\varphi}_n^0)^{\prime} (t) = 0$  for all
  $t\in (\tau_{n,-}^0,\tau_{n,+}^0)$, we only need $\Delta$ to have the property
  \begin{equation}
    \int_{\tau_{n,-}^0}^{\tau_{n,+}^0}\Delta(t)dt=0.\label{eq:GP_NKfilter}
  \end{equation}
  So, if $\tau_{n,+}^0>-\tau_{n,-}^0$ define
  \begin{eqnarray}
    \Delta_{NK,0}(t):=
    \begin{cases}
      \frac{\tau_{n,+}^0}{-\tau_{n,-}^0}(t-\tau_{n,-}^0), & \mbox{ for }t\in[\tau_{n,-}^0,0]\\
      (\tau_{n,+}^0-t), & \mbox{ for }t\in[0,\tau_{n,+}^0] \\
      0, & \mbox{ otherwise, }
    \end{cases} \label{NKDelta0Part1}
  \end{eqnarray}
  and
  if $\tau_{n,+}^0 < -\tau_{n,-}^0$,  define
  \begin{eqnarray}
    \Delta_{NK,0}(t):=
    \begin{cases}
      (t-\tau_{n,-}^0), & \mbox{ for }t\in[\tau_{n,-}^0, 0]\\
      \frac{-\tau_{n,-}^0}{\tau_{n,+}^0}(\tau_{n,+}^0-t), & \mbox{ for }t\in[0,\tau_{n,+}^0]\\
      0, & \mbox{ otherwise.}\end{cases}  \label{NKDelta0Part2}
  \end{eqnarray}
  Denote $h_{l}=\max(\tau_{n,+}^0,-\tau_{n,-}^0)$ and $h_{s}=\min(\tau_{n,+}^0,-\tau_{n,-}^0)$.
  We then compute
  $\int\Delta_{NK,0}(x)dx= h_{l}(\tau_{n,+}^0-\tau_{n,-}^0)/2$,
  so set
  \begin{equation}
    \Delta_{NK,1}(x):=\Delta_{NK,0}(x)-\frac{h_{l}}{2}1_{[\tau_{n,-}^0,\tau_{n,+}^0]}(x),\label{eq:delta1_NK}
  \end{equation}
  so that \eqref{eq:GP_NKfilter} is satisfied. Let
  \begin{align}
    & \Delta_{0}(x)  =\Delta_{RK,0}(x)1_{[m \text{ is RK}]}+\Delta_{LK,0}(x)1_{[m  \text{ is LK}]}+\Delta_{NK,0}(x)1_{[m \text{ is NK}]},     \nonumber \\
    & M_{n}  =M_{RK}1_{[m \text{ is RK}]}+M_{LK}1_{[m \text{ is LK}]}+h_{l}1_{[m \text{ is NK}]},
      \nonumber
  \end{align}
  and let
  \begin{align}
    \begin{split}
      \Delta_{1}(x)
      &=  \Delta_{RK,1}(x)1_{[m \text{ is RK}]}
      +\Delta_{LK,1}(x)1_{[m \text{ is LK}]}
      +  \Delta_{NK,1}(x)1_{[m \text{ is NK}]}\\
      & =  \Delta_{0}(x)-M_{n}.
    \end{split}
        \label{eq:GP_defDelta1}
  \end{align}
  Note that $M_{n}$ is $o_{p}(1)$ by uniform consistency in a neighborhood of
  $m$.  By the characterization
  Theorem~\ref{thm:CharThmOne}\ref{thm:CharThmOne:B-MC} and
  Corollary~\ref{cor:EstimatedFNearlyTouchesEDFatKnotsUnconstrained}, we see
  that
  \begin{eqnarray}
    \int\Delta_{1}d(\FF_n - F_0)
    & =  & \int\Delta_{1}d(\FF_n-\widehat{F}_n^0) + \int\Delta_{1}d(\widehat{F}_n^0 - F_0) \nonumber \\
    & =  & \int \Delta_{0}d(\FF_n- \widehat{F}_n^0) - M_{n}\int_{\tau_{n,-}^0}^{\tau_{n,+}^0}d(\FF_n-\widehat{F}_n^0)
    +\int\Delta_{1}d(\widehat{F}_n^0 - F_0)\nonumber \\
    & \le & M_{n}\bigg |\int_{\tau_{n,-}^0}^{\tau_{n,+}^0}d(\FF_n-\widehat{F}_n^0 ) \bigg | +  \int\Delta_{1}(x)(\widehat{f}_n^0- f_0)(x)dx
    \nonumber \\
    & \le & \frac{2M_{n}}{n} +\int\Delta_{1}(x)(\widehat{f}_n^0 -f_0)(x)dx.
    \label{eq:BasicInequalityKnotTightness}
  \end{eqnarray}
  Then Lemma \ref{lem:gapEPT_NK_LK} yields
  \begin{equation}
    \left| \int\Delta_{1}d(\FF_n- F_0) \right | \le  O_{p}(n^{-4/5}) + \frac{K}{2}h_{l}^{4},
    \label{eq:kimPollard_LK}
  \end{equation}
  \[
  \int\Delta_{1}(x)(\widehat{f}_n^0 - f_0)(x)dx  \le-Kh_{l}^{4}+o_{p}(h_{l}^{4}),
  \]
  for some $K>0$ and where we picked $\epsilon$ from
  Lemma \ref{lem:gapEPT_NK_LK} to be $K/2$.
  So, by rearranging (\ref{eq:BasicInequalityKnotTightness}),  we have
  \begin{eqnarray*}
    K(1+o_{p}(1))h_{l}^{4}
    & \le &  -\int\Delta_{1}(x)(\widehat{f}_n- f_0)(x)dx\\
    & \le & \frac{2M_{n}}{n} - \int\Delta_{1}d(\FF_n - F_0)
    \le  \frac{2M_{n}}{n}+ \bigg |\int\Delta_{1}d(\FF_n- F_0) \bigg |\\
    & \le & O_{p}(n^{-4/5})+\frac{K}{2}h_{l}^{4},
  \end{eqnarray*}
  and hence
  \[
  (K/2 + o_p (1) ) h_{l}^{4}) \le  O_{p}(n^{-4/5})
  \]
  which yields  $0 \le h_{l}=O_{p}(n^{-1/5})$.
  Since $\tau_{n,+}^0$ and $-\tau_{n,-}^0$ are both  bounded by $h_{l}$,
  we are done for the case $\xi_n  = m$.

  Extending to the case where $\xi_n \to m$ rather than being fixed and equal to
  $m$  follows from considering the event
  $m \in [\tau_{n,-}^0(\xi_n), \tau_{n,+}^0(\xi_n)]$
  and its complement separately, and then by showing on the event
  $m \notin [\tau_{n,-}^0(\xi_n),\tau_{n,+}^0(\xi_n)]$ that
  $\tau_{n,+}^0(\xi_n)-\tau_{n,-}^0(\xi_n) = O_p(n^{-1/5})$.
  This is straightforward.
\end{proof}

The following lemmas were used.

\begin{lemma}
  \label{lem:gapEPT_NK_LK}
  We continue with the setup of Proposition~\ref{prop:uniform-local-tightnessKnots}.  %
  That is, we define
  $h_{l} =\max(\tau_{n,+}^0 -m, m-\tau_{n,-}^0 )$,
  $h_{s} =\min(\tau_{n,+}^0 -m, m-\tau_{n,-}^0)$
  and
  $\Delta_{1}$ as in \eqref{eq:GP_defDelta1}.
  Then %
  for all $\epsilon>0$,
  \begin{equation}
    \left |\int \Delta_{1}d(\FF_n- F_0) \right |
    \le\epsilon h_{l}^{4} + O_{p}(n^{-4/5})   \label{eq:R2n}
  \end{equation}
  and
  \begin{equation}
    \int \Delta_{1}(x)(\widehat{f}_n^0 - f_0 ) (x) dx
    \le - \frac{f_0(0)|\varphi_0''(0)|}{2}Kh_{l}^{4} + o_{p}(h_{l}^{4}),\label{eq:R1n}
  \end{equation}
  where $K>0$ is from Lemma \ref{lem:deltaMonomialIntegral}, and does
  not depend on $f_0$.
\end{lemma}

\begin{proof}
  We examine $\int \Delta_{1}(x)(\widehat{f}_n^0 - f_0)(x)dx$ by repeated
  Taylor expansions, where we let $\Delta_{1}$ be $\Delta_{LK,1}$ or $\Delta_{NK,1}$,
  and we will expand at $m$, which we again take to be $0$, without
  loss of generality.
  Write $(\widehat{f}_n^0 -f_{0})=f_{0}((\widehat{f}_n^0 / f_0 ) -1 ) =f_{0}(\exp\{\widehat{\varphi}_n^0 -\varphi_{0} \} -1)$.
  Then write $d_{n}=\widehat{\varphi}_n^0 -\varphi_{0}$ so we can expand
  \begin{eqnarray*}
    \exp(d_{n}(t))-1
    & = & \sum_{i=1}^{1}\frac{d_{n}(t)^{i}}{i!} + e^{\xi_{1,n,t}}\frac{d_{n}(t)^{2}}{2!}, \\
    f_0 (t) & = & \sum_{i=0}^{1}\frac{f_0^{(i)}(0)t^{i}}{i!} +\frac{f_0^{(2)}(\xi_{2, n,t})t^{2}}{2!}
  \end{eqnarray*}
  for $t\in[\tau_{n,-}^0 ,\tau_{n,+}^0]$, where $\xi_{1,n,t}$ is between
  $0$ and $d_{n}(t)$ and $\xi_{2,n,t}$ is between $0$ and $t$.
  So, writing $\|\cdot \|_{\infty}$ as the uniform norm over $[\tau_{n,-}^0,\tau_{n,,+}^0 ]$,
  we can see that
  $f_0(t)(e^{d_{n}(t)}-1)$ equals
  \begin{align*}
    \MoveEqLeft
    \left(\sum_{i=0}^{1}\frac{(f_0)^{(i)}(0)t^{i}}{i!} + \frac{(f_0)^{(2)}(\xi_{2,n,t})t^{2}}{2!}\right)
    \left(\sum_{i=1}^{1}\frac{d_{n}(t)^{i}}{i!}+e^{\xi_{1,n,t}}\frac{d_{n}(t)^{2}}{2!}\right)\\
    & =   f_0 (0) d_{n}(t) + o_{p}(\| d_{n}(t)\|_{\infty})
  \end{align*}
  since $f_0(0)>0,$ $f_0^{(i)}$ is continuous and thus bounded on
  a neighborhood of $0$ for $i \in \{ 0, 1, 2\}$, and, by uniform consistency,
  $d_{n}(t)^{i}$ and $t^{i}$ both go to $0$ uniformly in a neighborhood of $0$.
  Then we examine
  \begin{equation*}
    \int_{\tau_{n,-}^0}^{\tau_{n,+}^0}\Delta_{1}(t) d_{n}(t) dt
    =  \sum_{i=0}^{1}\frac{d_{n}^{(i)}(0)}{i!}\int_{\tau_{n,-}^0}^{\tau_{n,+}^0}t^{i}\Delta_{1}(t)dt
    +\int_{\tau_{n,-}^0}^{\tau_{n,+}^0}\frac{d_{n}^{(2)}(\xi_{3,n,t})}{2!}t^{2}\Delta_{1}(t)dt
    \label{eq:GP_splitLKNK}
  \end{equation*}
  where $\xi_{3,n,t} \in[\tau_{n,-}^0,\tau_{n,+}^0]$.
  Note that $d_{n}^{(2)}(t)=-\varphi_{0}^{(2)}(t)$,
  and that since $d_{n}^{(2)}$ is continuous \textit{at} $0$ we
  can write $d_{n}^{(2)}(\xi_{3,n,t}) =d_{n}^{(2)}(0)+\epsilon_{n}(t)$
  where $\|\epsilon_{n}(t)\|_{\infty} \rightarrow_{p}0$ since
  $\tau_{n,+}^0-\tau_{n,-}^0\rightarrow_{p}0$.
  Now we consider the different possible forms $\Delta_{1}$ may take.

  If $0$ is NK, then $(\widehat{\varphi}_n^0 )^{(2)}(0)=0 =\varphi_0^\prime(0)$, so
  \begin{eqnarray}
    \lefteqn{\int_{\tau_{n,-}^0}^{\tau_{n,+}^0}\Delta_{NK,1}(t)d_n(t)dt } \nonumber \\
    & = & \sum_{i=0}^{2}\frac{d_{n}^{(i)}(0)}{i!}\int_{\tau_{n,-}^0}^{\tau_{n,+}^0}t^{i}
    \Delta_{NK,1}(t)dt
    +\int_{\tau_{n,-}^0}^{\tau_{n,+}^0}\frac{\epsilon_{n}(t)}{2!}t^{2}\Delta_{NK,1}(t)dt \nonumber   \\
    & = &  -\frac{\varphi_0^{(2)}(0)}{2!}\int_{\tau_{n,-}^0}^{\tau_{n,+}^0}t^{2}\Delta_{NK,1}(t)dt
    + \  \int_{\tau_{n,-}^0}^{\tau_{n,+}^0}\frac{\epsilon_{n}(t)}{2!}t^{2} \Delta_{NK,1}(t)dt.  \label{eq:GP_expand_dnt_NK}
  \end{eqnarray}
  Now we show that we get the same expansion if $0$ is a LK.
  Note that for
  $t\in[0,\tau_{n,+}^0],$ $\widehat{\varphi}_n^0 (t) = \widehat{\varphi}_n^0 (0) $ and for $t\in[\tau_{n,-}^0,0]$,
  $\widehat{\varphi}_n^0 (t) = \widehat{\varphi}_n^0 (0)  + (\widehat{\varphi}_n^0)^{\prime} (0-)t$.
  Thus
  \begin{align*}
    \int_{\tau_{n,-}^0}^{\tau_{n,+}^0}\Delta_{LK,1}(t) \widehat{\varphi}_n^0(t)dt
    & =  \int_{\tau_{n,-}^0}^{0}\Delta_{LK,1}(t)\widehat{\varphi}_n^0 (t)dt +  \int_{0}^{\tau_{n,+}^0}\Delta_{LK,1}(t)\widehat{\varphi}_n^0 (t)dt\\
    & =  \widehat{\varphi}_n^0 (0) \int_{\tau_{n,-}^0}^{0}\Delta_{LK,1}(t)dt
    +(\widehat{\varphi}_n^0 )^{\prime} (0-) \int_{\tau_{n,-}^0}^{0}t\Delta_{LK,1}(t)dt\\
    & \ \ + \ \widehat{\varphi}_n^0 (0) \int_{0}^{\tau_{n,+}^0}\Delta_{LK,1}(t)dt\\
    & =  0.
  \end{align*}
  Hence,
  \begin{eqnarray}
    \lefteqn{\int_{\tau_{n,-}^0}^{\tau_{n,+}^0}\Delta_{LK,1}(t)d_{n}(t)dt } \nonumber \\
    & = & \int_{\tau_{n,-}^0}^{\tau_{n,+}^0}\Delta_{LK,1}(t)\widehat{\varphi}_n^0 (t)dt
    -\int_{\tau_{n,-}^0}^{\tau_{n,+}^0}\Delta_{LK,1}(t)\varphi_{0}(t)dt   \nonumber \\
    & = & -\int_{\tau_{n,-}^0}^{\tau_{n,+}^0}\Delta_{LK,1}(t)\varphi_{0}(t)dt    \nonumber \\
    & = & -\sum_{i=0}^{2}\frac{\varphi_0^{(i)}(0)}{i!}\int_{\tau_{n,-}^0}^{\tau_{n,+}^0}t^{i}\Delta_{LK,1}(t)dt
    + \int_{\tau_{n,-}^0}^{\tau_{n,+}^0}\frac{\epsilon_{n}(t)}{2!}t^{2}\Delta_{LK,1}(t)dt    \nonumber \\
    & = & - \frac{\varphi_0^{(2)}(0)}{2!} \int_{\tau_{n,-}^0}^{\tau_{n,+}^0}t^{2}\Delta_{LK,1}(t)dt
    + \int_{\tau_{n,-}^0}^{\tau_{n,+}^0}\frac{\epsilon_{n}(t)}{2!}t^{2} \Delta_{LK,1}(t)dt. \label{eq:GP_expandDnt_LK}
  \end{eqnarray}
  Since an analogous statement holds for $\Delta_{RK,1}$, we have shown
  \begin{eqnarray*}
    \int_{\tau_{n,-}^0}^{\tau_{n,+}^0}\Delta_{1}(t)d_{n}(t)dt
    & = & -\frac{\varphi_0^{(2)}(0)}{2!}\int_{\tau_{n,-}^0}^{\tau_{n,+}^0}t^{2}\Delta_{1}(t)dt
    +  \int_{\tau_{n,-}^0}^{\tau_{n,+}^0}\frac{\epsilon_{n}(t)}{2!}t^{2}\Delta_{1}(t)dt
  \end{eqnarray*}
  where $\|\epsilon_{n} \|_{\infty}\rightarrow_{p}0$. Thus
  \begin{eqnarray*}
    \int\Delta_{1}(x)(\widehat{f}_n^0 - f_0 )(x) dx
    & = & f_0 (0)(1+o_{p}(1)) \int \Delta_{1}(t)d_{n}(t)\\
    & = & f_0 (0)(1+o_{p}(1))^{2}\frac{-\varphi_0{(2)}(0)}{2!} \int_{\tau_{n,-}^0}^{\tau_{n,+}^0}t^{2}\Delta_{1}(t)dt .
  \end{eqnarray*}
  Lemma~\ref{lem:deltaMonomialIntegral} shows that
  \begin{equation}
    \int_{\tau_{n,-}^0}^{\tau_{n,+}^0}t^{2}\Delta_{1}(t)dt  \le - Kh_{l}^{4}\
    \label{eq:DeltaMonomialIntegral_NK}
  \end{equation}
  which yields \eqref{eq:R1n}, our desired conclusion:
  \begin{eqnarray*}
    \int_{\tau_{n,-}^0}^{\tau_{n,+}^0}\Delta_{1}(t)  \left(\widehat{f}_n^0 - f_0 \right)(t)dt
    & \le & -\frac{f_0(0) |\varphi_0^{(2)}(0)|}{2}  Kh_{l}^{4}+o(h_{l}^{4}).
  \end{eqnarray*}

  Now we show \eqref{eq:R2n}.
  First for a fixed $\delta>0$, we know that
  $\tau_{n,+}^0$ and $\tau_{n,-}^0$ will be in $[-\delta,\delta]$ eventually, with
  high probability. Now we consider three families of functions, analogous to
  $\Delta_{LK,1}$, $\Delta_{RK,1}$, and $\Delta_{NK,1}$, respectively.
  Define
  $\Delta_{LK,1,b,c}$ by replacing $\tau_{n,-}^0$ with $b$ and $\tau_{n,+}^0$ with
  $c$ in \eqref{eq:Delta1_LK}.
  Define $\Delta_{NK,1,b,c}$ by replacing
  $\tau_{n,-}^0$ with $b$ and $\tau_{n,+}^0$ with $c$ in \eqref{eq:delta1_NK}.
  Define  $\mathcal{F}_{LK,b,R}:=\{\Delta_{LK,b,y}|b<a<y,\mbox{ }0\le y-b\le R\}$
  and
  $\mathcal{F}_{NK,b,R}:=\{\Delta_{NK,b,y}|b<a<y,\mbox{ }0\le y-b\le R\}$
  for  $R>-b$.
  Define $\mathcal{F}_{RK,b,R}$ analogously to $\mathcal{F}_{LK,b,R}$.
  Let
  $\mathcal{F}=\mathcal{F}_{LK,b,R} \cup\mathcal{F}_{RK,b,R}\cup\mathcal{F}_{NK,b,R}$,
  and note that $\mathcal{F}$ is VC-class with VC-index of $4$.
  Thus Theorem~2.6.7 on page 141 of \cite{VW1996WCEP}
  shows that the entropy bound condition in Lemma A.1 on page 2560 of \cite{BW2007kmono}
  holds for $\mathcal{F}$.
  Then the function $F_{b,R}(x)$, defined to be
  constant equal to $(7/4)R$ on $[b,b+R]$ and $0$ otherwise, is an
  envelope for $\mathcal{F}$. That $F_{b,R}$ is an envelope is immediate for
  $\Delta\in\mathcal{F}_{NK,b,R}$ and for the setting where
  $\tau_{n,+}^0<-\tau_{n,-}^0$
  and $\Delta\in\mathcal{F}_{LK,b,R}$
  (and analogously when $\tau_{n,+}^0>-\tau_{n,-}^0$ and $\Delta\in\mathcal{F}_{RK,b,R}$).
  (For $\Delta\in\mathcal{F}_{NK,b,R}$, the
  longer interval has slope $\pm1$ and the other interval has opposite sign
  slope. For the case $\tau_{n,+}^0<-\tau_{n,-}^0$ and
  $\Delta\in\mathcal{F}_{LK,b,R}$, the interval $[\tau_{n,-}^0,\tau_{n,-}^0/2]$
  has slope $1$ and the slope on the rest is opposite sign (and analogously for
  $\mathcal{F}_{RK,b,R}$).)
  For the case $\tau_{n,+}^0\ge-\tau_{n,-}^0$ and
  $\Delta\in\mathcal{F}_{LK,b,R}$, we need only note
  $$
  0\le m_{2}=\frac{\tau_{n,+}^0}{-\tau_{n,-}^0}\left(\frac{-9-3\frac{\tau_{n,+}^0}{-\tau_{n,-}^0}}{1-5\frac{\tau_{n,+}^0}{-\tau_{n,-}^0}}\right)\le3,
  $$
  so that $(-\tau_{n,-}^0/4) m_{2}\le (3/4) \tau_{n,+}^0$.
  Next, we compute the integral of the envelope squared
  \begin{eqnarray*}
    EF_{b,R}^{2}(X)
    & = & \int_{b}^{b+R}R^{2} f_0(x)dx  \le  \| f_0 \|_{\infty} R^{3},
  \end{eqnarray*}
  where $\| f_0 \|_{\infty}$ is the supremum
  over $\RR$ of the (log-concave) density $f_{0}$, and is thus universal across
  $b$ and $R$. Thus, we can conclude from
  Lemma A.1 on page 2560 of \cite{BW2007kmono}
  that for
  $\epsilon>0$ and with $s=2$ and $d=2$,
  \begin{eqnarray*}
    \left | \int\Delta_{1}d(\FF_{n}- F_0) \right |
    & \le & \epsilon(\tau_{n,+}^0-\tau_{n,-}^0)^{4} + O_{p}(n^{-4/5}),\\
    & \le & \epsilon2^{4}h_{l}^{4} + O_{p}(n^{-4/5})
  \end{eqnarray*}
  as desired.
\end{proof}

\begin{lemma}
  \label{lem:deltaMonomialIntegral}
  For $h_{l}=\max(\tau_{n,+}^0- m, m-\tau_{n,-}^0)$,
  $h_{s}=\min(\tau_{n,+}^0- m , m-\tau_{n,-}^0)$ and $\Delta_{1}$ defined by
  \eqref{eq:GP_defDelta1}, we have
  \begin{eqnarray*}
    \int\Delta_{1}(t)(t-  m )^{2}dt & \le & - Kh_{l}^{4},
  \end{eqnarray*}
  for some $K>0$.
\end{lemma}

\begin{proof}
  First, we consider $\Delta_{LK,1}$, assume without loss of generality
  that $m =0$, and assume $\tau_{n,+}^0>-\tau_{n,-}^0$.
  Direct computation shows
  \begin{align*}
    \int_{\tau_{n,-}^0}^{\tau_{n,+}^0}t\Delta_{LK,1}(t)dt
    & =  -\frac{5}{12} \left(\frac{(\tau_{n,+}^0-\tau_{n,-}^0) (\tau_{n,+}^0)^{3}}{5\tau_{n,+}^0+\tau_{n,-}^0}\right)
    \le  -\frac{5}{12}\left(\frac{(\tau_{n,+}^0-\tau_{n,-}^0)(\tau_{n,+}^0)^{3}}{5(\tau_{n,+}^0)}\right)
  \end{align*}
  and this is bounded above by
  \begin{equation*}
    -\frac{5}{12}\frac{(\tau_{n,+}^0)^{4}}{5\tau_{n,+}^0}
    =  -\frac{(\tau_{n,+}^0)^{3}}{12}  =  -\frac{h_{l}^{3}}{12},
  \end{equation*}
  since $\tau_{n,+}^0>-\tau_{n,-}^0$.

  We break
  $\int_{\tau_{n,-}^0}^{\tau_{n,+}^0}t^{2}\Delta_{LK,1}(t)dt$ into
  two pieces. First we see
  \begin{eqnarray*}
    \int_{\tau_{n,-}^0}^{0}t^{2}\Delta_{LK,1}(t)dt
    & = & \frac{-3(\tau_{n,-}^0)^{4}\tau_{n,+}^0-19(-(\tau_{n,-}^0))^{3}(\tau_{n,+}^0)^{2}}{96(5\tau_{n,+}^0+\tau_{n,-}^0)}
    <  0.
  \end{eqnarray*}
  Next we see that
  \begin{eqnarray*}
    \int_{0}^{\tau_{n,+}^0}t^{2}\Delta_{LK,1}(t)dt
    & = & \frac{-3(-\tau_{n,-}^0)(\tau_{n,+}^0)^{4}-5(\tau_{n,+}^0)^{5}}{12(5\tau_{n,+}^0+\tau_{n,-}^0)}\\
    & \le & -\frac{3(-\tau_{n,-}^0)(\tau_{n,+}^0)^{4}+5(\tau_{n,+}^0)^{5}}{12\cdot5(\tau_{n,+}^0)}
    \le  -\frac{(\tau_{n,+}^0)^{4}}{12}.
  \end{eqnarray*}
  Thus,
  \[
  \int_{\tau_{n,-}^0}^{\tau_{n,+}^0}t^{2}\Delta_{LK,1}(t)dt\le-\frac{(\tau_{n,+}^0)^{4}}{12}=-\frac{h_{l}^{4}}{12},
  \]
  as desired.

  Now we consider $\Delta_{LK,1}$ when $\tau_{n,+}^0 <-\tau_{n,-}^0$, and
  again split the computation into two pieces. First we see
  \begin{eqnarray*}
    \int_{0}^{\tau_{n,+}^0}t^{2}\Delta_{LK,1}(t)dt
    & = & -\frac{2(\tau_{n,-}^0)^{2}(\tau_{n,+}^0)^{3}
      - 3(\tau_{n,+}^0)^{5}}{12(5\tau_{n,+}^0-2\tau_{n,-}^0)} .
  \end{eqnarray*}
  Next we find that
  \begin{eqnarray*}
    \int_{\tau_{n,-}^0}^{a}t^{2}\Delta_{LK,1}(t)dt
    & = & -\frac{(-\tau_{n,-}^0)^{5}+5(\tau_{n,-}^0)^{4}\tau_{n,+}^0}{48(5\tau_{n,+}^0-2\tau_{n,-}^0)}\\
  \end{eqnarray*}
  so that
  \begin{eqnarray}
    \int_{\tau_{n,-}^0}^{\tau_{n,+}^0}t^{2}\Delta_{LK,1}(t)dt
    & =  &\frac{(\tau_{n,-}^0 )^5 - 5 (\tau_{n,-}^0 )^4 \tau_{n,+}^0 - 8 (\tau_{n,-}^0 )^3 (\tau_{n,+})^3 + 12 (\tau_{n,+}^0 )^5}
    {48 ( 5 \tau_{n,+}^0 - 2 \tau_{n,-}^0 )}  \nonumber \\
    & \le & - \frac{(-\tau_{n,-}^0)^4}{48\cdot 7}
    \label{DeltaLKtimestSquaredCase2}
  \end{eqnarray}
  and it remains only to prove the last inequality.
  To do this, let $a \equiv \tau_{n,-}^0 < 0 < \tau_{n,+}^0 \equiv b$ where $-a\ge b$.
  Thus we want to show that
  \begin{eqnarray*}
    K(a,b) = - 5 a^5 + 30 a^4 b + 56 a^2 b^3 - 84 b^5 \ge 0
  \end{eqnarray*}
  for all $a < 0 < b$ with $-a \ge b$.  But this holds if and only if $J(v,c) \equiv K(-v,cv) \ge 0$
  holds for all $v\ge 0 $ and $0 \le c \le 1$.  But
  \begin{align*}
    J(v,c) = K(-v,cv)
    \ge  v^5 (5 c^5 + 30 c^5 + 56 c^5 - 84 c^5 )
    & =  v^5 c^5 \cdot 7 \ge 0 .
  \end{align*}
  Thus \eqref{DeltaLKtimestSquaredCase2} holds.

  Identical calculations hold for $\Delta_{RK,1}$.
  Now we examine $\Delta_{NK,1}$.
  Direct computation shows
  \begin{equation*}
    \int_{\tau_{n,-}^0}^{\tau_{n,+}^0}t^{2}\Delta_{NK,1}(t)dt=- \frac{1}{12}(h_{s}^{3}h_{l}+h_{l}^{4})\le - \frac{1}{12}h_{l}^{4}.
  \end{equation*}
\end{proof}

\subsubsection{Proofs for mode-constrained limit process}
\label{subsec:proofs-MC-process}

This entire section is devoted to the proof of
Theorem~\ref{thm:MC-process-uniqueness-theorem}, and thus throughout this section we take
the definitions and assumptions as given in that theorem.

\begin{proof}[Proof of Lemma~\ref{lem:process-uniqueness-theorem-constant-interval}]
  If $0 \in \Sa$ then one of $\tauplusa$ or
  $\tauminusa$ is $0$, because then $(\vva)^{(2)}$ is not
  constant-equal-to-$0$ on any open neighborhood of $0$; if it is not constant-equal-to-$0$
  on $[0,\delta)$ for any $0 < \delta $ then
  $(\vva)'$ is also not constant on $[0, \delta)$,
  so since $0 \ge (\vva)'(0+)$, we have
  $(\vva)'(\delta+) < 0$ for all $\delta > 0$, i.e.\ $\tauplusa = 0$ (of
  course $\tauplusa \ge 0 $ since $0$ is the mode of $\vva$).  Similarly, if
  $(\vva)^{(2)}$ is not $0$ on any neighborhood below $0$, then $\tauminusa =
  0$.

  Now, if $0 \notin \Sa$ then $\tauminusa = \sup \Sa \cap (-\infty, 0]$ and
  $\tauplusa = \inf \Sa \cap [0, \infty)$.  If $0 \in \Sa$ then we have shown
  that one of $\tau^0_{\pm}$ equals $0$.  Thus, in either case, it is clear
  that $(\vva)^{(2)}(t \pm) = 0$ for $ 0 \le t < \tauplusa$ and for $
  \tauminusa < t \le 0$.  Thus, regardless of whether $0 \in \Sa$,
  $(\vva)^{(2)}$ and so $(\vva)'$ are constant and equal to $0$ on
  $(\tauminusa, \tauplusa)$, i.e.\ \eqref{eq:vva-modal-interval} holds (and
  so $(\tauminusa, \tauplusa)$ is the modal interval of $\vva$).
\end{proof}

The next lemma gives the sense in which $\vva$ is piecewise affine.
\begin{lemma}
  \label{cor:S0-knot-properties}
  Assume that $H_L, H_R,$ and $\vva$ are as in
  Theorem~\ref{thm:MC-process-uniqueness-theorem}.  Define
  \begin{align}
    S_L& = \{ t \le 0 : H_L(t) = Y_L(t), H_L'(t)=Y_L'(t) \}, \nonumber \\
    S_R &= \{ t \ge 0 : H_R(t) = Y_R(t), H_R'(t)=Y_R'(t) \},  \nonumber  \\
    \mbox{ and }
    S^0 & = S_L \cup S_R  \cup \{ 0 \}. \label{eq:def:S}
  \end{align}
  Then $(\vva)'$ is a monotonically nonincreasing function and the `bend
  points' of $\vva$, $\Sa$, defined in \eqref{eq:defnSa}, satisfy $\Sa \subset S^0$.
  Additionally, with probability $1$ the following statements hold.  The sets
  $S_L, S_R,$ and $S^0$ are all
  closed and have Lebesgue measure $0$.  For any fixed $t \ne 0$, $t \notin
  S^0$ and so $(H^0)^{(3)}(t)$ is well-defined.
\end{lemma}
The lemma says that for any knot $\tau \le \tau_+^0$, if $\tau < 0$ then
$\tau \in S_L$.  Similarly if $\tau \ge \tau_-^0$ and $\tau > 0$ then $\tau \in
S_R$.  It is possible but not guaranteed that $0$ is a knot and lies in
either $S_L$ or $S_R$.
\begin{proof}
  By Theorem~\ref{thm:MC-process-uniqueness-theorem}, displays
  \eqref{eq:CharznInequalityL2_full} and \eqref{eq:CharznInequalityR2_full}, $H^0 - Y^0
  \le 0$,
  which allows us to conclude
  \begin{equation}
    \label{eq:equality-S}
    \begin{split}
      \lb t < 0  : H_L(t)=Y_L(t) \rb
      &  = \lb t < 0 :
      H_L(t)=Y_L(t), H_L'(t)=Y_L'(t) \rb, \\
      \lb t > 0  : H_R(t)=Y_R(t) \rb
      & = \lb t > 0 :
      H_R(t)=Y_R(t), H_R'(t)=Y_R'(t) \rb;
    \end{split}
  \end{equation}
  the first line follows since $H_L-Y_L$ is differentiable on $ (-\infty,0)
  $, and a differentiable function has derivative $0$ at a local maximum
  (see, e.g., \cite{MR0349288}, %
  page 153, Problem 3, part (a)).  The same argument applies to the second
  line of \eqref{eq:equality-S}.

  Now, the following argument holds with probability $1$ and for any fixed $c
  > 0$.  On $[0,c]$, $H_R$ has a bounded second derivative, so that there
  exists a constant $a > 0$ such that $\tilde H_R(t) := H_R(t) + at^2$ is
  convex on $[0,c]$.  Let $\tilde Y_R(t) := Y_R(t) + at^2.$
  Now, $Y_R + A = Y $ for an affine function  $A$  and where $Y(t) = \int_0^t
  \int_0^u dX(z) du$.  Let $\tilde Y(t) = Y(t) + at^2 = \tilde Y_R(t) +  A(t)$
  so that
  \begin{equation}
    \label{eq:points-of-touch-set-equality}
    \begin{split}
      \{ x \in [0,c] : H_R(x)= Y_R(x) \}
      &= \{ x \in [0,c] : \tilde H_R(x) +
      A(x) = \tilde Y_R(x) + A(x) \} \\
      & = \{ x \in [0,c] : \tilde H_R(x) + A(x) = \tilde Y(x) \}
    \end{split}
  \end{equation}
  We also have $\tilde H_R + A \le \tilde Y_R + A = \tilde Y$ by
  \eqref{eq:CharznInequalityR2_full}, so that, letting $M_{\tilde Y}$ be the
  greatest convex minorant of $\tilde Y$ on $[0,c]$, we have $\tilde H_R + A
  \le M_{\tilde Y} \le \tilde Y$, since $\tilde H_R + A$ is convex and below
  $\tilde Y$.

  Let $T = \lb x \in [0,c] : M_{\tilde Y}(x) = \tilde Y(x) \rb$.
  By the proof of Corollary~2.1 of \cite{MR1891741}
  (see also
  Definition 1 and Theorem 1 of \cite{Sinaui:1992wl}),
  $T$ is a (Cantor-type) set which has Lebesgue
  measure $0$;  \eqref{eq:points-of-touch-set-equality} is contained
  in $T$, and thus
  (by \eqref{eq:equality-S} and)  by letting $c \to \infty$, we see that $S_R$
  is contained in a
  set which has Lebesgue measure $0$.
  Finally, $S_R$ is closed because $H_R-Y_R$ and $(H_R-Y_R)'$ are both
  continuous functions.  By an analogous argument, we can conclude that $S_L$
  is closed and has Lebesgue measure $0$ and thus also $S^0$ is closed and
  has Lebesgue measure $0$.  By \eqref{eq:CharznEquality2_full},
  \begin{equation}
    \label{eq:7}
    \int_{\{t \in (-\infty, \tau_-^0] : \HHLt(t)\ne\YYLt(t)\}} d(\vva)'(t)
    = 0
    = \int_{\{t \in [\tau_+^0, \infty) : \HHRt(t)\ne\YYRt(t)\}} d(\vva)'(t),
  \end{equation}
  (regardless of whether one of $\tau_+^0$ or $\tau_-^0$ is $0$ or not).

  Thus we now conclude that $\Sa \subset S^0$ as follows.  If $\tau$ is an
  element of the set $\Sa$ then for any $\epsilon >0 $
  $(\vva)'(\tau-\epsilon, \tau+\epsilon ) < 0$ (here $(\vva)'$ refers to the
  signed measure corresponding to $(\vva)'$).  This is by the definition of
  derivative; since $(\vva)'$ is nonincreasing, $\delta \mapsto (\vva)'(\tau
  + \delta) - (\vva)'(\tau-\delta) \le 0$ is nonincreasing.  Thus if
  \begin{equation}
    \label{eq:5}
    \frac{ (\vva)'(\tau + \delta) - (\vva)'(\tau-\delta)}{ 2 \delta}
  \end{equation}
  does not converge to $0$ as $\delta \searrow 0$, then for all $\epsilon
  >0 $ there is $0 < \delta < \epsilon$ such that \eqref{eq:5} is $<$
  const., i.e.,
  \begin{equation}
    \label{eq:6}
    (\vva)'(\tau + \epsilon) - (\vva)'(\tau-\epsilon)
    \le (\vva)'(\tau+\delta) - (\vva)'(\tau-\delta) < 0.
  \end{equation}
  Since $\HHRt-\YYRt$ is continuous on all of $\RR$, if
  $\HHRt(\tau)-\YYRt(\tau) < 0$ for $\tau \ge \tau^0_+$, then on a
  neighborhood $(\tau-\epsilon, \tau+\epsilon)$ for some $\epsilon > 0$ we
  have $\HHRt-\YYRt < 0$ and so the integral on the right-hand side of
  \eqref{eq:7} is strictly less than $0$.  Thus if $\tau \ge \tau_+^0$, then
  $(\HHRt-\YYRt)(\tau) = 0$.  For $\tau > 0$, this implies that $\tau \in
  S_R$ by \eqref{eq:equality-S}.  Similarly, if $ \tau \le \tau_-^0$ and
  $\tau < 0$, then $(\HHLt-\YYLt)(\tau)=0$ and $\tau \in S_L$.  We have
  $(\tau_-^0,\tau_+^0) \cap \Sa = \emptyset$ by
  Lemma~\ref{lem:process-uniqueness-theorem-constant-interval}, so we have
  shown that if $\tau \in \Sa$, then $\tau $ is either $0$ (if one of
  $\tau^0_-$ or $\tau^0_+$ is $0$) or is in $S_L$ or $S_R$, i.e.\ $\tau \in
  S^0$.

  Now, by the proof of Theorem 1 of \cite{Sinaui:1992wl}, any fixed
  point $t \ge 0$ belongs to $T \supset S_R$
  with probability zero.  An analogous statement holds
  for $t \le 0$ and $S_L$.  Thus, if $t \ne 0$, $t \not\in S^0$ and so
  $(H_L)^{(2)}$ is concave and affine in a neighborhood of $t$, so
  $(H_L)^{(3)}(t)$ is well-defined.
\end{proof}

By Lemma~\ref{cor:S0-knot-properties},
\begin{equation}
  \label{eq:10}
  \HHLt(\TauL) = \YYLt(\TauL) = 0, \mbox{ and }
  \HHLt'(\TauL) = \YYLt'(\TauL) = 0.
\end{equation}
This is because, by its definition, either $\TauL < 0$, in which case $\TauL
\in S_L$, or there is a sequence $\lb \tau^0_{L,n} \rb \subset S_L$ with
$\tau^0_{L,n} < 0$ for all $n$.  In this latter case, since $H_L-Y_L$ and
$H_L'-Y_L'$ are both continuous, we still conclude that $\HHLt(\TauL) =
\YYLt(\TauL) = 0$, and $\HHLt'(\TauL) = \YYLt'(\TauL) = 0$.  Analogously,
$H_R(\TauR) = Y_R(\TauR)$ and $H_R'(\TauR) = Y_R'(\TauR)$.  This suggests the
following definitions:
\begin{equation}
  \label{eq:defn_Ftildes}
  \FFLct(u) = \int^{\TauL}_u \vva(v)dv
  \quad   \mbox{ and } \quad
  \FFRct(u) = \int_{\TauR}^u \vva(v)dv,
\end{equation}
\begin{equation}
  \label{eq:defn_Xtildes}
  \XXLct(u) = \int^{\TauL}_u dX(v)
  \quad  \mbox{ and } \quad
  \XXRct(u) = \int_{\TauR}^u dX(v).
\end{equation}
It is then true by definition that
\begin{equation}
  \label{eq:9}
  \YYLt(t)
  = \int_t^{\TauL} \XXLct(u) du,
  \quad \mbox{ and }
  \quad
  \YYRt(t)
  = \int_{\TauR}^t \XXRct(u) du.
\end{equation}
It is furthermore true that
\begin{align}
  \HHLt(t) & = \int_t^{\TauL} \int_u^{\TauL} \vva(v)dvdu
  = \int_t^{\TauL} \FFLct(u) du,
  \label{eq:defHLtilde} \\
  \HHRt(t) & =  \int_{\TauR}^t \int_{\TauR}^u \vva(v)dv du
  = \int_{\TauR}^t \FFRct(u) du,
  \label{eq:defHRtilde}
\end{align}
since $\HHLt(t) = \int^{\TauL}_t \int^{\TauL}_u \vva(v) dv du + A_L(t)$ where
$A_L(t)$ is an affine function, and we just verified
in \eqref{eq:10} that $A_L(t) \equiv 0$. An
analogous statement holds for $H_R$.

\begin{remark}
  Note that since $S^0$ is closed, $\tau_+, \tau_-$ are both elements of $S^0$.
  Also, it is possible for $H_L^{(3)}$ to have a point of increase at $0$
  without $H_L'(0)=Y_L'(0)$, since the inequality
  \eqref{eq:CharznInequalityL2_full} only holds on $(-\infty, 0]$, so not in an
  open neighborhood of $0$.  Similarly for $H_R^{(3)}$.  This is why we add the
  point $0$ to $S^0$.
\end{remark}
Lemma~\ref{cor:S0-knot-properties} suggests that we can think of $\vva$
as being piecewise affine (with a potentially uncountable number of knot
points), because with probability $1$, the union of the open intervals on
which $\vva$ is affine
has full Lebesgue measure on the real line (meaning its complement has
Lebesgue measure $0$).
For $t \in \RR$, we let $\tauplusa(t)$ be the first knot larger than $t$, and
analogously for $\tauminusa(t)$,
\begin{align}
  \tauplusa(t) = \inf \lp \Sa \cap [t, \infty)\rp
  \quad \mbox{ and } \quad
  \tauminusa(t) = \sup  \lp \Sa \cap (-\infty, t] \rp .
  \label{eq:def-tauplusa-tauminusa}
\end{align}
\begin{lemma}
  \label{lem:vva_infiniteknots}
  We again assume the full setup of Theorem
  \ref{thm:MC-process-uniqueness-theorem}.
  Then, for any (fixed or random)
  $T \ge 0$, with probability $1$ there are `knot points'
  $\tauplusa(T)$ and $\tauminusa(T)$ %
  in $\Sa$.  %
\end{lemma}
\begin{proof}
  We fix $T \ge 0$, and we will show that there exists $\tau_+(T) \in \Sa$
  with $\tau_+(T) > T \ge 0$.  We assume for
  contradiction that $\vva$ has no knots on $(T,\infty)$, and thus is linear.
  Thus $\HHRt$ is cubic on $[T,\infty)$, so can be written as $\HHRt(t) =
  \sum_{i=0}^3 A_i (t-T)^i$ for some random $A_i$.  By definition, we have
  \begin{equation*}
    \begin{split}
      \YYRt(t) = \int_{\TauR}^t \XXRt(u)du
      & = \int_{\TauR}^t (X(u)-X(\TauR))du \\
      & = \int_0^t X(u)du -\int_0^{\TauR} X(u)du - (t-\TauR)X(\TauR).
    \end{split}
  \end{equation*}
  In other words, $\YYRt(t)$ is $\int_0^t X(u)du = V(t) + t^4$ plus a
  random affine function, where $V(t)=\int_0^t W(u)du$.  Thus we can
  write
  \begin{equation*}
    \YYRt(t)-\HHRt(t) = V(t) + t^4 - \sum_{i=0}^3 B_i(t-T)^i,
  \end{equation*}
  for some new random coefficients, $B_i$ (where only for $i=0,1$ are
  $B_i$ not equal to $A_i$).  Now, let $\varphi(t) =
  \sqrt{\frac{2}{3} t^3 \log \log t}$. Then by page 1714 of
  \cite{Lachal1997LIL} (or from page 238 of \cite{Watanabe1970LIL}),
  we know that almost surely
  \begin{equation*}
    \limsup_{t \to \infty} \frac{\int_0^t W(u)du}{\varphi(t)} = 1.
  \end{equation*}
  Thus,
  \begin{equation*}
    \frac{\YYRt(t)-\HHRt(t)}{\varphi(t)}
    = \frac{V(t)}{\varphi(t)} + \frac{t^4 - \sum_{i=0}^3 B_i(t-T)^i}{\varphi(t)},
  \end{equation*}
  which gets larger than $0$ for $t$ large enough, as it is almost surely
  bounded below by a quadratic polynomial (with positive first coefficient)
  minus $1$.  This contradicts the fact that $\YYRt(t)-\HHRt(t) \le 0$ for
  all $t$, so we are done. Our argument applies with probability $1$ to any
  $T \ge 0$, and thus to the entire sample space of any random $T \ge 0$.  An
  identical argument works for showing there exists a knot less than $-T$.
\end{proof}
We will not speak of $\vva$ as a minimizer of an objective
function, but we will instead show that for acceptable $\Delta$ perturbations
that $\int \Delta(t) (\vva(t)dt-dX(t)) \ge 0$, i.e. $\vva$ behaves as we
would expect a minimizer to behave.
\begin{proposition}
  \label{pro:charzn0_full}
  We assume the full setup of Theorem~\ref{thm:MC-process-uniqueness-theorem}.
  Let $\Delta: \RR \to \RR$ be %
  concave with %
  maximum at $0$. If $a,b \in \Sa$ with $-\infty < a < 0 < b < \infty$ then
  \begin{equation}
    \label{eq:charzn0_full}
    \int_{a}^{b} \Delta(t) (\vva(t)dt - dX(t)) \ge 0,
  \end{equation}
  and thus, by Lemma \ref{lem:vva_infiniteknots},
  $    \limsup_{a \to \infty} \int_{-a}^{a} \Delta(t) (\vva(t)dt - dX(t)) \ge     0.$
\end{proposition}
\begin{proof}
  We use the notation $g(a,b] = g(b)-g(a)$ here. We have
  \begin{align*}
    \MoveEqLeft   \int_{a}^{b} \Delta(t) (\vva(t)dt - dX(t))  \\
    & =  -\int_{a}^0 \Delta(t)(d\FFLt (t)-d\XXLt(t)) + \int_0^b \Delta(t)(d\FFRt(t)-d\XXRt(t))  \\
    & =  - \left[ (\Delta(\FFLt-\XXLt))(a,0] - \int_{a}^0 ((\FFLt-\XXLt)\Delta')(t)dt \right]  \\
    & \quad + (\Delta(\FFRt-\XXRt))(0,b] - \int_0^b ((\FFRt-\XXRt)\Delta')(t)dt.
  \end{align*}
  By Lemma~\ref{cor:S0-knot-properties}, $a \in S_L$, $b \in S_R$, and since
  neither $a$ nor $b$ is $0$, we have $(\FFRt-\XXRt)(b)=0=(\FFLt-\XXLt)(a)$ and
  we recall \eqref{eq:CharznEquality2_full}. Also recalling that
  $(\HHLt-\YYLt)'= -(\FFLt-\XXLt)$, we see that the above display equals
  \begin{align*}
    \MoveEqLeft[4]{ -\Delta(0)( (\FFLt-\XXLt)(0)+(\FFRt-\XXRt)(0) ) } \\
    \MoveEqLeft{ - \left[ ((\HHRt - \YYRt)\Delta'(\cdot +))(0,b] - \int_0^b (\HHRt-\YYRt)(t)d\Delta'(t)  \right] } \\
    \MoveEqLeft{ - \left[ ((\HHLt-\YYLt)\Delta'(\cdot -))(a,0] - \int_{a}^0 (\HHLt-\YYLt)(t)d\Delta'(t) \right] } \\
    & =  -\Delta(0) \left( \int_0^{\TauL}(\vva(t)dt-dX(t)) + \int_{\TauR}^0 (\vva(t)dt-dX(t))  \right) \\
    & \quad + (\HHRt-\YYRt)(0) \Delta'(0+) - (\HHLt-\YYLt)(0)\Delta'(0-) \\
    & \quad + \int_{0}^b (\HHRt-\YYRt)(t)d\Delta'(t) + \int_a^0 (\HHLt-\YYLt)(t)d\Delta'(t) \\
    & \ge 0,  %
  \end{align*}
  where the inequality follows because each of the three lines in the final
  expression is $\le 0$, as follows.  The first line is equal to $0$ by
  \eqref{eq:CharznFequalsX_M_full}; the third line is $\ge 0$ by
  \eqref{eq:CharznInequalityL2_full} and \eqref{eq:CharznInequalityR2_full},
  and the fact that $\Delta$ is concave so $\Delta'$ is monotonically
  nonincreasing so $d\Delta'$ is a nonpositive measure;
  similarly, the second line is $\ge 0$ because $\Delta$ has maximum at $0$,
  so that  $(\HHRt-\YYRt)(0)$, $\Delta'(0+)$,
  $(\HHLt-\YYLt)(0),$ and $-\Delta'(0-)$
  are nonpositive.
\end{proof}
The above proof can be extended  to $\Delta$ such that $\vva(t) + \epsilon
\Delta(t) \in {\cal G}^0$, where ${\cal G}^0$ is the set of concave functions
with maximum at $0$, but we will not need this per se.  %
Rather, in the next result we will express the same idea by showing for knots
$a < 0 < b$ that $\int_{a}^{b} \vva(t) \left(\vva(t)dt-dX(t) \right) = 0$,
and re-express this via integration by parts formulae.
\begin{proposition}
  \label{pro:charznEqualityRewrite_full}
  We again assume the full setup of Theorem
  \ref{thm:MC-process-uniqueness-theorem} and assume that $a,b \in \Sa$, and
  $a < 0 < b$.  Then
  \begin{align*}
    \MoveEqLeft \int_{a}^{b} \vva(t) (\vva(t)dt-dX(t)) \\
    & = \int_a^0 ((\FFLt-\XXLt)(\vva)')(t)dt - \int_0^b ((\FFRt-\XXRt)(\vva)')(t)dt  \\
    & =  \int_a^{\TauL} (\HHLt-\YYLt)(t)d(\vva)'(t)  +  \int_{\TauR}^b (\HHRt-\YYRt)(t)d(\vva)'(t)  \\
    & =  0.
  \end{align*}
\end{proposition}
\begin{proof}
  Since $a < 0 < b$, we again have $\FFLt(a)-\XXLt(a)=0=\FFRt(b)-\XXRt(b)$ by
  Lemma~\ref{cor:S0-knot-properties}, so
  \begin{align*}
    \MoveEqLeft \int_{a}^{b} \vva(t) (\vva(t)dt-dX(t)) \\
    & = -\int_{a}^0 \vva(t)d(\FFLt(t)-\XXLt(t))
    + \int_0^b \vva(t)d(\FFRt(t)-\XXRt(t))
  \end{align*}
  which equals
  \begin{align*}
    \MoveEqLeft
    -\left[ (\vva(\FFLt-\XXLt))(a,0] - \int_a^0 ((\FFLt-\XXLt)(\vva)')(t)dt  \right] \\
        &  + (\vva(\FFRt-\XXRt))(0,b] - \int_0^b ((\FFRt-\XXRt)(\vva)')(t)dt
  \end{align*}
  which equals
  \begin{align*}
    \MoveEqLeft
    -\left[ (\vva(\FFLt-\XXLt))(0) - \int_a^0 ((\FFLt-\XXLt)(\vva)')(t)dt  \right] \\
    & - (\vva(\FFRt-\XXRt))(0) - \int_0^b ((\FFRt-\XXRt)(\vva)')(t)dt \\
    & \quad =  \int_a^0 ((\FFLt-\XXLt)(\vva)')(t)dt - \int_0^b ((\FFRt-\XXRt)(\vva)')(t)dt,
  \end{align*}
  where the last equality is by \eqref{eq:CharznFequalsX_M_full}.  Since
  $\vva$ is constant on $(\tau_{-}^0,\tau_{+}^0)$, and using the notation $g[c,d]
  = g(d) - g(c-)$, we can write the last expression above as
  \begin{align*}
    \MoveEqLeft{\int_a^{\tau_{-}^0} ((\FFLt-\XXLt)(\vva)')(t)dt - \int_{\tau_{+}^0}^b ((\FFRt-\XXRt)(\vva)')(t)dt } \\
    & =   ((\vva)'(\HHLt-\YYLt))[a,\tau_{-}^0] - \int_{[a,\tau_{-}^0]} (\HHLt-\YYLt)(t)d(\vva)'(t) \\
    & \quad -\left[ ((\vva)'(\HHRt-\YYRt))[\tau_{+}^0,b] - \int_{[\tau_{+}^0,b]} (\HHRt-\YYRt)(t)d(\vva)'(t) \right]
  \end{align*}
which equals
\begin{align*}
   -\int_{[a,\tau_-^0]} (\HHLt-\YYLt)(t)d(\vva)'(t)  +  \int_{[\tau_+^0,b]} (\HHRt-\YYRt)(t)d(\vva)'(t)
   = 0,
\end{align*}
  using integration by parts formula (see
  Lemma~\ref{lem:integration-by-parts}) for the first equality since
  $\HHLt-\YYLt$ and $\HHRt-\YYRt$ are continuous, and using
  \eqref{eq:CharznEquality2_full} for the third equality.  The second
  equality follows from Lemma~\ref{cor:S0-knot-properties}, since a knot $a$
  and limit of knots $\tau_-$ are elements of $S_L$, and similarly $b,
  \tau_+$ are elements of $S_R$.
\end{proof}
Next we prove a representation lemma, analogous to the midpoint result for the
unconstrained (and compact support) case in Lemma 2.3 on page 1631 of
\cite{MR1891741}.  %
\begin{lemma}
  \label{lem:unq:Hrepresentation}
  We again assume the full setup of Theorem
  \ref{thm:MC-process-uniqueness-theorem}.  Let $\tau_1,\tau_2 \in \Sa$%
  be such that $\vva$ is affine on $[\tau_1,\tau_2]$,
  and let $t \in [\tau_1,\tau_2]$.  For any function $g$, we define
  $\Delta g = g(\tau_2)-g(\tau_1)$ and $\bar{g} =
  \frac{g(\tau_1)+g(\tau_2)}{2}$, including in particular,
  $\Delta \tau = \tau_2-\tau_1$ and
  $\bar{\tau} = (\tau_1+\tau_2)/2$.
  Then if $0 < \tau_1 < \tau_2$, we can conclude
  \begin{equation}
    \begin{split}
      \label{eq:unq:HrepresentationR}
      \HHRt(t) & = \frac{(\YYRt(\tau_2)(t-\tau_1) + \YYRt(\tau_1)(\tau_2-t))}{\Delta \tau} \\
      & \quad - \inv{2} \left( \frac{\Delta \XXRt}{\Delta \tau}
        + \frac{4}{(\Delta \tau)^3}(\bar{X}_R\Delta\tau - \Delta \YYRt)(t-\bar{\tau}) \right)
      (t-\tau_1)(\tau_2-t),
    \end{split}
  \end{equation}
  and thus
  \begin{equation}
    \label{eq:unq:HrepresentationR_taubar}
    \HHRt(\bar{\tau}) = \bar{Y}_R - \inv{8}\Delta \XXRt \Delta \tau.
  \end{equation}
  If $\tau_1<\tau_2<0$, we can conclude
  \begin{equation}
    \label{eq:unq:HrepresentationL}
    \begin{split}
      \HHLt(t) & = \frac{\YYLt(\tau_2)(t-\tau_1)+\YYLt(\tau_1)(\tau_2-t)}{\Delta \tau} \\
      & \quad -\inv{2}\left( \frac{-\bar{X}_L}{\Delta \tau} +
        \frac{4}{(\Delta \tau)^3}(-\bar{X}_L \Delta \tau - \Delta \YYLt)(t-\bar{\tau}) \right)
      (t-\tau_1)(\tau_2-t),
    \end{split}
  \end{equation}
  and thus
  \begin{equation}
    \label{eq:unq:HrepresentationL_taubar}
    \HHLt(\bar{\tau}) = \bar{Y}_L + \inv{8}\Delta\XXLt \Delta \tau.
  \end{equation}
\end{lemma}
\begin{proof}
  We assume that on $[\tau_1, \tau_2]$, that $\vva$ is linear and thus
  $\HHLt$ and $\HHRt$ are cubic polynomials.  Thus, taking $\tau_1 <\tau_2 <
  0$, $\HHLt$ is defined by its values and its derivative's values at
  $\tau_{i}$, for $i=1,2$.  Thus, if we name the polynomial on the right hand
  side of \eqref{eq:unq:HrepresentationL} $P_L$, it suffices to check that
  $P_L(\tau_{i})$ and $P_L'(\tau_{i})$ equal $H_L(\tau_i)$ and
  $H_L'(\tau_i)$, respectively, for $i=1,2$, to conclude that $\HHLt(t) =
  P_L(t)$ for $t \in [\tau_1,\tau_2]$.  We know that $\HHLt(\tau_{i}) =
  \YYLt(\tau_{i})$ by \eqref{eq:CharznEquality2_full} and it is immediate
  that $P_L(\tau_{i}) = \YYLt(\tau_{i})$, so we only need to check the
  derivative values.  To differentiate, we denote
  $$A(t) = \inv{2}\left(
    \frac{-\bar{X}_L}{\Delta \tau} + \frac{4}{(\Delta \tau)^3}(-\bar{X}_L
    \Delta \tau - \Delta \YYLt)(t-\bar{\tau}) \right),$$ so that
  \begin{equation*}
    P_L(t) = \frac{\YYLt(\tau_2)(t-\tau_1)+\YYLt(\tau_1)(\tau_2-t)}{\Delta \tau}  - A(t)(t-\tau_1)(\tau_2-t),
  \end{equation*}
  and
  \begin{equation*}
    P_L'(t) = \frac{\YYLt(\tau_2)
      - \YYLt(\tau_1)}{\Delta \tau}  - A'(t)(t-\tau_1)(\tau_2-t)
    - A(t)((\tau_2-t) - (t-\tau_1)),
  \end{equation*}
  so that
  \begin{align*}
    P_L'(\tau_1) & = \frac{\Delta \YYLt}{\Delta \tau} - A(\tau_1)\Delta \tau \\
    & = \frac{\Delta \YYLt}{\Delta \tau}
    - \inv{2} \left\{ \frac{-\Delta \XXLt}{\Delta \tau}
      + \frac{4}{(\Delta \tau)^3} (-\bar{X}_L\Delta\tau
      - \Delta\YYLt) \left(\frac{-\Delta\tau}{2}\right) \right\}\Delta \tau
    \\
    &  = -\XXLt(\tau_1).
  \end{align*}
  This equals $\HHLt'(\tau_1)$, as desired, since $\HHLt'(\tau_1) =
  \YYLt'(\tau_1)$
  by Lemma~\ref{cor:S0-knot-properties} since $\tau_1 $ is strictly less than
  $0$, and $ \YYLt'(\tau_1) = -\XXLt(\tau_1)$.  Similarly, $P_L'(\tau_2) = -
  \XXLt(\tau_2)$ and, letting $P_R$ be the polynomial on the right hand side
  of \eqref{eq:unq:HrepresentationL}, $P_R'(\tau_{i}) = \XXRt(\tau_{i})$ and
  $P_R(\tau_i) = \YYRt(\tau_i)$.  Then \eqref{eq:unq:HrepresentationR_taubar}
  and \eqref{eq:unq:HrepresentationL_taubar} follow immediately.
\end{proof}

Next, we show a tightness-type of results for the bend points.  Recall the
definition \eqref{eq:def-tauplusa-tauminusa} of $\tau_-^0(t)$ and
$\tau_+^0(t)$.
\begin{lemma}
  \label{lem:unq_knotsTight}
  Let the assumptions of Theorem~\ref{thm:MC-process-uniqueness-theorem} hold.
  Then, for all $\epsilon > 0 $ there exists $M_\epsilon$ such that for all
  $t >0  $,
  \begin{align}
    &P( \tauplusa(t) > t+M_\epsilon)  < \epsilon, \label{eq:unq:knotsTight_R_plus}\\
    &P( \tauminusa(-t) < -t-M_\epsilon)  < \epsilon,  \label{eq:unq:knotsTight_L_minus}\\
    &P( (t-M_\epsilon) \vee 0  \le  \tauminusa(t) \vee 0  )  > 1- \epsilon,  \label{eq:unq:knotsTight_R_minus}\\
    &P(  \tauplusa(-t) \wedge 0 \le (-t+M_\epsilon) \wedge 0 )  > 1-\epsilon, \label{eq:unq:knotsTight_L_plus}
  \end{align}
  where $M_\epsilon$ does not depend on  $t$.
\end{lemma}
\begin{proof}
  We will show for all $t,\epsilon >0 $ there exists $M=M_\epsilon$ such that
  $P( \tauplusa(t) > t+M) < \epsilon $. The statement for $\tauminusa(-t)$ is
  analogous. By Lemma \ref{lem:vva_infiniteknots}, for any $t$ we can find
  $\tau_2 \in \Sa$ where $\tau_2 < \infty$ is taken to be $\tauplusa(t)$; similarly, we can take $t \equiv
  t_\epsilon$ large enough such that with probability $1-\epsilon$ there
  exists a knot $0< \tau_1 < t$.  To match notation up with Lemma
  \ref{lem:unq:Hrepresentation}, we will
  define $\Delta g$ and $\bar{g}$, for any function $g$, as in the
  lemma. Since $\vva$ is affine on $[\tau_1,\tau_2]$, Lemma
  \ref{lem:unq:Hrepresentation} allows us to conclude that $\YYRt(\bar{\tau})
  \le \HHRt(\bar{\tau}) = \bar{Y}_R - \Delta \XXRt \Delta \tau /
  8$ %
  which is if and only if
  \begin{equation}
    \label{eq:unq:taubar_mainineq}
    Y(\bar{\tau}) \le \bar{Y} - \inv{8} \Delta X \Delta \tau,
  \end{equation}
  where $Y(t)= \int_0^t X(u)du = V(t) + t^4$ and $V(t)=\int_0^tW(u)du.$ The
  ``if and only if'' follows because $\YYRt(t) = Y(t) + A(t)$ where $A(t)$ is
  a random affine function.  Since for any affine function $ A(\bar{\tau}) =:
  \bar{A},$ we see that
  \begin{equation*}
    \YYRt(\bar{\tau})-\bar{Y}_R = Y(\bar{\tau}) + A(\bar{\tau}) - (\bar{Y}+\bar{A})
    = Y(\bar{\tau}) - \bar{Y}.
  \end{equation*}
  Since $\Delta X $ trivially equals $\Delta \XXRt$, we have shown
  \eqref{eq:unq:taubar_mainineq}. Let $M_\epsilon > 0$ and let $B_t$ be the
  event $\lb 0 < \tau_1 < t, \tau_2 > t +M_\epsilon \rb$. We then see
  \begin{align}
    P(\tau_2 > t_{\epsilon} + M_\epsilon) %
    & \le P( \tau_1 \le 0) +
    P(B_{t_\epsilon}) \notag \\
    & \le \epsilon
    + P\left(\YYRt(\bar{\tau}) \le \HHRt(\bar\tau)= \bar{Y}_R - \inv{8}\Delta\XXRt \Delta\tau ,
      B_{t_\epsilon} \right) \notag \\
    & \le 2\epsilon, \label{eq:unq:mainIneq_probability}
  \end{align}
  where we now show that the last inequality follows from page 1633 in the
  proof of Lemma~2.4 in \cite{MR1891741}. %
  We have already noted that $ \YYRt(\bar{\tau}) \le \bar{Y}_R -
  \inv{8}\Delta\XXRt \Delta\tau $ if and only if $ Y(\bar{\tau}) \le \bar{Y}
  - \inv{8}\Delta X \Delta\tau $.  Then
  \cite{MR1891741} %
  show algebraically that this inequality can be rewritten as
  \begin{align*}
    V(\bar{\tau}) - \bar{V} + \inv{8}\Delta W \Delta \tau \le - \left(\frac{\Delta \tau}{2}\right) ^4.
  \end{align*}
  Thus we have shown for any $t > 0$,
  \begin{align*}
    \MoveEqLeft{ P\left(\YYRt(\bar{\tau}) \le \HHRt(\bar\tau)
        = \bar{Y}_R - \inv{8}\Delta\XXRt \Delta\tau ,
        B_t \right) } \\
    & = P\left(  V(\bar{\tau}) - \bar{V} + \inv{8}\Delta W \Delta \tau \le
      - \left(\frac{\Delta \tau}{2}\right) ^4,
      B_t \right).
  \end{align*}
  \cite{MR1891741} %
  show that
  \begin{equation}
    \label{eq:unq:mainineqRewrite1}
    P\left(  V(\bar{\tau}) - \bar{V} + \inv{8}\Delta W \Delta \tau \le
      - \left(\frac{\Delta \tau}{2}\right) ^4,
      \tau_1 < M_\epsilon, \tau_2 > M_\epsilon \right)
    < \epsilon,
  \end{equation}
  and thus that
  \begin{equation}
    \label{eq:unq:mainineqRewrite2}
    P\left(  V(\bar{\tau}) - \bar{V} + \inv{8}\Delta W \Delta \tau \le
      - \left(\frac{\Delta \tau}{2}\right) ^4,
      \tau_1 < t-M_\epsilon, \tau_2 > t+M_\epsilon \right)
    < \epsilon.
  \end{equation}
  This independence from $t$ follows because
  \begin{equation*}
    \left\{(W(s)-W(t),
      V(s)-V(t)-(s-t)W(t))\right\}_{s \in \RR}
  \end{equation*}
  is equal in distribution to $
  \left\{W(u-t),\int_t^u W(u-t)\, du)\right\}_{u \in \RR}$, since $\int_t^s
  W(u)du = \int_t^s (W(u)-W(t))du + (s-t)W(t)$, and thus
  $V(\bar{s}) - \bar{V} + \inv{8}\Delta W \Delta s $ equals
  \begin{align*}
    \MoveEqLeft V(\bar{s})-V(t)-W(t)(\bar{s} -t) \\
    \MoveEqLeft \quad - \left(\inv{2}(V(s_1) - V(t)-W(t)(s_1-t))
      + \inv{2}(V(s_2) - V(t) -W(t)(s_2-t)) \right)\\
    \MoveEqLeft \quad +\inv{8}(W(s_2)-W(t)-(W(s_1)-W(t)))(s_2-t-(s_1-t)) \\
    & =_d V\left(\frac{r_1+r_2}{2}\right) - \frac{V(r_1)+V(r_2)}{2} + \inv{8}(W(r_2)-W(r_1))(r_2-r_1)
  \end{align*}
  where $\bar{s}=(s_1+s_2)/2$, $\bar{V}=(V(s_1)+V(s_2))/2$, $\Delta W=W(s_2)-W(s_1)$,
  and $\Delta s=s_2-s_2$ and $r_i=s_i-t$ for $i=1,2$.
  This shows that the left hand sides of both of \eqref{eq:unq:mainineqRewrite1} and
  \eqref{eq:unq:mainineqRewrite2} are, regardless of $t$, bounded by
  \begin{equation}
    \label{eq:unq:brownianBound}
    P\left(  V(\bar{s}) - \bar{V} + \inv{8}\Delta W \Delta s \le
      - \left(\frac{\Delta s}{2}\right) ^4,
      \mbox{ for some }\; s_1 < -M_\epsilon, s_2 > M_\epsilon \right).
  \end{equation}
  This probability is defined in (2.27) on page 1633 of
  \cite{MR1891741}, %
  and is shown to be less than $\epsilon$ at the top of page 1634, so, using
  this fact, we have now shown \eqref{eq:unq:mainineqRewrite1} and
  \eqref{eq:unq:mainineqRewrite2}.

  The probability we consider in \eqref{eq:unq:mainIneq_probability} is on
  the event $B_{t_\epsilon} = \{ 0<\tau_1<t_\epsilon , \tau_2 > t_\epsilon +
  M_\epsilon \} $ rather than $C_{t_\epsilon} \equiv
  \{ 0<\tau_1<t_\epsilon -
  M_\epsilon , \tau_2 > t_\epsilon + M_\epsilon \}$.  The only cost for this
  is we need to double our $M_\epsilon$ for this to correspond with the
  probability in \eqref{eq:unq:mainineqRewrite2}.  Thus
  \eqref{eq:unq:mainIneq_probability} holds, but we do not yet have
  independence from $t$ because of the $t_\epsilon$ in the expression. We
  easily circumvent this by replacing $M_\epsilon$ by $t_\epsilon +
  M_\epsilon$.  Now we have shown \eqref{eq:unq:knotsTight_R_plus} holds with
  $M_\epsilon$ independent of $t$.

  Now we show \eqref{eq:unq:knotsTight_L_minus}.
  Note that we can write an analogous version of
  \eqref{eq:unq:mainIneq_probability} for $t > M_\epsilon$ as
  \begin{align}
    P( 0 \le \tau_1 \le t - M_\epsilon)
    & \le P(\tau_2 > t+M_\epsilon) +
    P(  C_t ) \label{eq:tau1-prob-bound}  \\
    & \le \epsilon + P\left(\YYRt(\bar{\tau}) \le \HHRt(\bar\tau)= \bar{Y}_R
      - \inv{8}\Delta\XXRt \Delta\tau ,
      C_{t_\epsilon} \right) \notag \\
    & \le 2 \epsilon \notag
  \end{align}
  because, by the argument we just went through, the probability in the third
  line is again bounded by \eqref{eq:unq:brownianBound}.  Note we have $t$ in
  place of $t_\epsilon$ in \eqref{eq:tau1-prob-bound}, so the above statement
  is already independent of $t$ as long as $t > M_\epsilon$.  Thus we have
  shown $P( (t-M) \vee 0 \le \tauminusa(t) \vee 0 \le t ) > 1- \epsilon$,
  since if $t < M_\epsilon$ this probability is trivially $1$.  Showing the
  analogous statements for the left side, the existence of $M$, not depending on
  $t$, such that $P(\tauminusa(-t) < -t - M) < \epsilon$ %
  and $P( -t \le \tauplusa(-t) \wedge 0 \le (-t+M) \wedge 0 ) > 1-\epsilon$,
  can be done analogously.
\end{proof}

The next result will relate the unconstrained and constrained limit
estimators in the Gaussian setting.
\begin{corollary}
  \label{cor:unq:MCequalsUC}
  Let the definitions and assumptions of
  Theorem~\ref{thm:MC-process-uniqueness-theorem} hold, and let $Y$ and $H$
  be as in
  Theorem~\ref{thm:charzn_uniqueness_full_UC}. Let $\vv \equiv H''$.
  For any $t \in \RR$, define
  \begin{align}
    s^+(t)  &= \inf \{s \in [t,\infty) : \vva(s)= \vv(s)\} \\
    s^-(t)  &= \sup \{s \in (-\infty,t] : \vva(s)= \vv(s)\}.
  \end{align}
  Then we can say that for all $\epsilon > 0 $, there exists $M_\epsilon$,
  not depending on $t$,  such that
  \begin{align}
    P(t - s^-(t) > M_\epsilon) < \epsilon \label{eq:unq:MCequalsUC_minus}\\
    P(s^+(t)-t > M_\epsilon) < \epsilon.\label{eq:unq:MCequalsUC_plus}
  \end{align}
\end{corollary}
\begin{proof}
  Define a right-side sequence of knots to be a sequence of points
  \begin{equation*}
    0 < \nu_1 < \nu^0_{1} < \nu_2 < \nu^0_{2} < \nu_3,
  \end{equation*}
  where $\nu_{i}$ are knots for $\vv$ and $\nu^0_i$ are knots for $\vva$.
  Similarly, define  a left-side sequence of knots
  \begin{equation*}
    \nu_{-3} < \nu^0_{-2} < \nu_{-2} < \nu^0_{-1} < \nu_{-1} < 0.
  \end{equation*}
  Then we argue by the Intermediate Value Theorem and the Mean Value Theorem.
  First, we assume we are given such a sequence, without loss of generality
  take it to be a right-side sequence (on the probability $1$ event on which
  Theorem~\ref{thm:charzn_uniqueness_full_UC} holds).  Then we can say, by
  our hypotheses, that
  \begin{align}
    (\HHRt-\YYRt)(\nu^0_{i}) &= 0 \le (H-Y)(\nu^0_{i}) \quad \mbox{ for } i=1,2 \\
    (\HHRt-\YYRt)(\nu_i)    &\ge 0 = (H-Y)(\nu_i) \quad \mbox{ for } i=1,2,3.
  \end{align}
  By the Intermediate Value Theorem we can pick points $x_1 \in
  [\nu_1,\nu^0_{1}], x_2 \in [\nu^0_{1},\nu_2], x_3 \in
  [\nu_2,\nu^0_{2}]$ such that $ (\HHRt-\YYRt)(x_i) = (H-Y)(x_i)$ for
  $i=1,2,3$.  Since $\YYRt(t)-Y(t) = A(t)$ is a (random) affine
  function, we can conclude for $i=1,2,3$ that
  \begin{equation*}
    \HHRt(x_i)-H(x_i)-A(x_i)=0.
  \end{equation*}
  We apply the Mean Value Theorem and get $t_i \in (x_i,x_{i+1})$ for $i=1,2$
  such that
  \begin{equation*}
    \FFRt(t_i) - H'(t_i) - A'(t_i) = 0.
  \end{equation*}
  Again applying the Mean Value Theorem, we get $s \in (t_1,t_2) \subset
  (x_1,x_3) \subset [\nu_1, \nu^0_{2}] \subset [\nu_1,\nu_3]$.

  Now we will construct right-side sequences or left-side sequences of knots
  and be done.  Note that by Lemma \ref{lem:unq_knotsTight} and the analogous
  lemma for the unconstrained case, Lemma 2.7, page 1638, of
  \cite{MR1891741}, %
  there exists a large $M>0$ such that with probability $1-\epsilon$ there
  exists a right-side sequence of knots contained in any interval of length
  $\ge M$ that lies in $[0, \infty)$ and a left-side sequence of knots in any
  interval of length $\ge M$ that lies in $(-\infty,0]$.  For any $t>0$, note
  that the interval $[t-2M,t]$ contains an interval of length at least $M$
  which lies either entirely in $(-\infty,0]$ or entirely in
  $[0,\infty)$. Thus, $[t-2M,t]$ contains a one-sided sequence of knots, and
  thus an $s<t$ such that $\vva(s)=\vv(s)$, with probability $1-\epsilon$.
  Similarly, there exists a one-sided sequence of knots in $[t,t+M]$, and
  thus an $s>t$ such that $\vva(s)=\vv(s)$, with probability $1-\epsilon$.
  Thus, for $t>0$, we have shown \eqref{eq:unq:MCequalsUC_minus} and
  \eqref{eq:unq:MCequalsUC_plus}.  Similarly, for $t < 0$, we consider
  intervals $[t-M,t]$ and $[t,t+2M]$ in which there exist one-sided sequences
  of knots, which allows us to conclude that $\vva(s) = \vv(s)$ for an $s >
  t$ and an $s < t$.
\end{proof}

\begin{lemma}
  \label{lem:unq:MCdiffUCtight_t}
  Let the assumptions of Theorem~\ref{thm:MC-process-uniqueness-theorem}
  hold.
  For all $\epsilon > 0 $ there exists $M_\epsilon$, not depending on $t$, such that
  \begin{align}
    P \left( |\vva(t)- \vv(t) | > M_\epsilon \right) <
    \epsilon
    \mbox{ and }
    P \left( |(\vva)'(t)- \vv'(t) | > M_\epsilon \right) <
    \epsilon,
    \label{eq:unq:ftight_t}
  \end{align}
  where the derivatives can be taken
  to be right or left derivatives.
\end{lemma}
\begin{proof}
  This follows from Lemma \ref{lem:unq_knotsTight} and an argument similar to
  the finite sample tightness results.  %
  We can pick, by Corollary
  \ref{cor:unq:MCequalsUC}, $t-2M < s_{-2}<s_{-1} < t < s_1 < s_2 < t+ 2M$
  where $\vva(s_i) = \vv(s_i)$ for $i=-2,-1,1,2$, with probability
  $1-\epsilon$ for $M$ appropriately large. Then
  \begin{align*}
    (\vva)'(t) \le \frac{\vva(s_2)-\vva(s_1)}{s_2-s_1} = \frac{\vv(s_2)-\vv(s_1)}{s_2-s_1}
    \le \vv'(s_2),
  \end{align*}
  and, similarly,
  \begin{equation*}
    (\vva)'(t) \ge \vv'(s_{-2})
  \end{equation*}
  where $\vv'$ and $(\vva)'$ can be either the left or right
  derivatives. Thus
  \begin{equation*}
    (\vva)'(t)-\vv'(t) \le \vv'(s_{2})-\vv'(t) \le \vv'(s_2) - \vv'(s_{-2})
  \end{equation*}
  and
  \begin{equation*}
    (\vva)'(t)-\vv'(t) \ge \vv'(s_{-2})-\vv'(t) \ge \vv'(s_{-2}) - \vv'(s_2),
  \end{equation*}
  that is,
  \begin{equation}
    \label{eq:11}
    |(\vva)'(t)- \vv'(t)| \le \vv'(s_2) - \vv'(s_{-2}).
  \end{equation}
  Let $h_0(t) = -12t^2$.  With high probability, the right side of
  \eqref{eq:11} is bounded by
  \begin{equation}
    \label{eq:unq:stdMonotoneDecomp}
    \begin{split}
      \MoveEqLeft  \vv'(t+2M)-\vv'(t-2M) \\
      &  \le |\vv'(t+2M)- h_0'(t+2M)|
      + |h_0'(t+2M)- h_0'(t-2M)| \\
      & \quad + | h_0'(t-2M)- \vv'(t-2M)|,
    \end{split}
  \end{equation}
  which is less than $M+M+24\cdot 2M$ with probability $1-\epsilon$,
  independently of $t$, by (2.36) or (2.37) of Lemma 2.7 on page 1638 of
  \cite{MR1891741}. %
  Thus we have shown the second statement in \eqref{eq:unq:ftight_t}, which
  we will now use to show the first statement in \eqref{eq:unq:ftight_t}.

  We first apply Lemma~\ref{fact:unq:convexFunDiffs} to the difference
  $|\vva(t)-\vv(t)|$ by applying \eqref{eq:unq:convexFunDiffs} to both
  $\vva-\vv$ and to $\vv - \vva$, using the points $s_{-1}$ and $s_1$ as $a$
  and $b$, respectively. Then by \eqref{eq:unq:convexFunDiffs}, we
  can bound both of these differences if we can bound both
  \begin{equation}
    (\vva)'(s_1)- (\vva)'(s_{-1}) \le (\vva)'(t+M)- (\vva)'(t-M) \label{eq:unq:vvap_diff}
  \end{equation}
  and
  \begin{equation}
    \vv'(s_1)-\vv'(s_{-1})  \le \vv'(t+M)-\vv'(t-M) , \label{eq:unq:fp_diff}
  \end{equation}
  since all the other terms are $0$ by the definition of the $s_i$. Here we
  can take the derivatives to be either left or right derivatives.  As in
  \eqref{eq:unq:stdMonotoneDecomp}, we can bound $ (\vva)'(t+M) -
  (\vva)'(t-M)$ from above by
  \begin{equation*}
    \begin{split}
      \MoveEqLeft |(\vva)'(t+M)-\vv'(t+M)|
      +     |\vv'(t+M)-\vv'(t-M)| \\
      & +     |\vv'(t-M) - (\vva)'(t-M)|.
    \end{split}
  \end{equation*}
  The first and last terms are bounded by the second statement in
  \eqref{eq:unq:ftight_t}. The middle term is shown to be bounded by
  \eqref{eq:unq:stdMonotoneDecomp}. The middle term also bounds
  \eqref{eq:unq:fp_diff}. All of this is with probability $1-\epsilon$ and
  uniformly in $t$, so we are done.
\end{proof}

For the next lemma, let $h_0(t) = -12t^2$ be the ``true'' concave function.
\begin{lemma}
  \label{lem:unq:tight_t}
  Let the definitions and assumptions of
  Theorem~\ref{thm:MC-process-uniqueness-theorem} hold.  Then, for all
  $\epsilon > 0 $ there exists $M_\epsilon$, independent of $t$, such that
  \begin{align}
    P \left( |\vva(t)- h_0(t) | > M_\epsilon \right) < \epsilon \label{eq:unq:ftight_t_f0} \\
    P \left( |(\vva)'(t)- h_0'(t) | > M_\epsilon \right) <
    \epsilon. \label{eq:unq:fpR_tight_t_f0}
  \end{align}
  where the derivatives can be right or left derivatives.
\end{lemma}
\begin{proof}
  This is immediate from Lemma 2.7, page 1638, of
  \cite{MR1891741} %
  and Lemma \ref{lem:unq:MCdiffUCtight_t}.
\end{proof}

We are now in a position to prove
Theorem~\ref{thm:MC-process-uniqueness-theorem}.
\begin{proof}[Proof of Theorem~\ref{thm:MC-process-uniqueness-theorem}]
  We define objective functions with variable bounds of integration,
  \begin{equation}
    \label{eq:defn_phiab}
    \phi_{a,b}(g)= \inv{2} \int_{a}^{b} g^2(t)dt - \int_{a}^{b} g(t)dX(t),
  \end{equation}
  where we will always take $a < 0 < b$.  For $i=1,2$, we will take
  $\HHLti{i}$ and $\HHRti{i}$ to satisfy the hypotheses stated in the
  theorem, and we need to show $\HHLti{1} \equiv \HHLti{2}$ and
  $\HHRti{1} \equiv \HHRti{2}$ almost surely.  We will denote
  $\FFLti{i}=-\HHLti{i}'$ and $\FFRti{i}=\HHRti{i}'$ and
  \begin{equation}
    \label{eq:unq:defn:fi}
    \ffi{i}=\HHLti{i}''=\HHRti{i}''.
  \end{equation}
  We also will use the notation $d\FFi{i}(t) = \ffi{i}(t)dt$.  Now, using
  that $\ffi{1}^2-\ffi{2}^2 = (\ffi{1}-\ffi{2})^2 +
  2(\ffi{1}-\ffi{2})\ffi{2}$, we see that
  \begin{align*}
    \phi_{a,b}(\ffi{1})-\phi_{a,b}(\ffi{2})
    = &  \inv{2}\int_{a}^{b} \left( \ffi{1}-\ffi{2} \right)^2 d\lambda
    + \int_{a}^{b} \left( \ffi{1}(t)-\ffi{2}(t)\right) d(\FFi{2}(t)-X(t))
  \end{align*}
  where $\lambda$ is Lebesgue measure.
  Now, we specify that $\ani{i}$ and $\bni{i}$ \label{unq:ani_bni} are knots
  for $\ffi{i}$, and, using Lemma~\ref{lem:vva_infiniteknots}, we take
  $\ani{2}<\ani{1}<-n<0<n<\bni{1}<\bni{2}$. Then
  \begin{align*}
    \phi_{\ani{2},\bni{2}} (\varphi_1)-  \phi_{\ani{2},\bni{2}} (\varphi_2)
    \ge \inv{2}\int_{\ani{2}}^{\bni{2}} \left( \ffi{1}-\ffi{2} \right)^2 d\lambda
    \ge \inv{2}\int_{-n}^{n} \left( \ffi{1} -\ffi{2} \right)^2 d\lambda
  \end{align*}
  by Propositions \ref{pro:charzn0_full} and
  \ref{pro:charznEqualityRewrite_full}, and, similarly,
  \begin{equation*}
    \phi_{\ani{1},\bni{1}} (\varphi_2)- \phi_{\ani{1},\bni{1}} (\varphi_1) \ge
    \inv{2}\int_{-n}^{n} \left( \ffi{2}(t)-\ffi{1}(t) \right)^2 dt.
  \end{equation*}
  Now, we see directly from \eqref{eq:defn_phiab} that
  $  \phi_{\ani{2},\bni{2}}(\ffi{1}) - \phi_{\ani{1},\bni{1}} (\varphi_1)$ equals
  \begin{align*}
    \MoveEqLeft \inv{2}\int_{\ani{2}}^{\bni{2}} \ffi{1}^2(t)dt -
    \inv{2}\int_{\ani{1}}^{\bni{1}} \ffi{1}^2(t)dt -
    \left( \int_{\ani{2}}^{\bni{2}} \ffi{1}(t)dX(t)-\int_{\ani{1}}^{\bni{1}} \ffi{1}(t)dX(t)  \right) \\
    = & \inv{2}\int_{A_n} \ffi{1}^2(t)dt - \int_{A_n} \ffi{1}(t)dX(t),
  \end{align*}
  where $A_n = [\ani{2},\ani{1}] \cup [\bni{1},\bni{2}]$.
  Thus we have
  \begin{align*}
    \int_{-n}^{n} \left(\ffi{1}-\ffi{2}\right)^2 d\lambda
    \le & \; \phi_{\ani{2},\bni{2}}(\ffi{1}) - \phi_{\ani{1},\bni{1}} (\ffi{1}) - (\phi_{\ani{2},\bni{2}}(\ffi{2})
    - \phi_{\ani{1},\bni{1}}(\ffi{2})) \\
    = & \inv{2}\int_{A_n} \left( \ffi{1}^2-\ffi{2}^2 \right)d\lambda
    - \int_{A_n} \left( \ffi{1}(t)-\ffi{2}(t) \right) dX(t).
  \end{align*}
  Recalling $ \ffol(t) = 12t^2$,
  we see the previous display equals
  \begin{align*}
    \int_{A_n} \inv{2}\left( \left(\ffi{1}-\ffol\right)^2 - \left(\ffi{2}-\ffol\right)^2 \right) d\lambda
    - \left( \ffi{1}-\ffi{2} \right)dW.
  \end{align*}
  Thus we can conclude that
  \begin{align}
    0 & \le \lim_n \int_{-n}^{n} \left(\ffi{1}-\ffi{2}\right)^2 d\lambda
    \nonumber \\
    & \le \liminf_n  \left( \inv{2}\int_{A_n} \left( \left(\ffi{1}-\ffol\right)^2
        - \left(\ffi{2}-\ffol\right)^2 \right) d\lambda
      - \int_{A_n} \left( \ffi{1}-\ffi{2} \right)dW
    \right).  \label{eq:unq:contradictionIneq} %
  \end{align}
  The proof will now proceed as follows.  We first will show that the right
  hand side of \eqref{eq:unq:contradictionIneq} is finite.
  For a function $g \colon \RR \to \RR$, we let
  $ \Vert g\Vert_a^b = \sup_{t \in [a,b]} |g(t)|$ and $ \Vert g\Vert_a^\infty =
  \sup_{t \in [a,\infty)} |g(t)|.$ Using that $\int_{-\infty}^{\infty}
  \left(\ffi{1}-\ffi{2}\right)^2 d\lambda$ is finite we will then conclude that
  $||\ffi{1}(t)-\ffi{2}(t)||_n^\infty \to 0$ as $n \to \infty$.  We then will
  revisit our earlier argument which showed the right hand side of
  \eqref{eq:unq:contradictionIneq} was finite and use this new fact to show
  that \eqref{eq:unq:contradictionIneq} is, in fact, $0$. This will finish the
  proof.

  Thus, our next step is to show that $\int_{-\infty}^\infty \left(
    \ffi{1}(t)-\ffi{2}(t) \right)^2 dt < \infty$.  Note that we only need to
  control the $\liminf_n$ of the right hand side of
  \eqref{eq:unq:contradictionIneq} since $\int_{-n}^n (\ffi{1}-\ffi{2})^2 \,
  d\lambda $ is non-negative and non-decreasing in $n$.  We will first show
  that $\int_{\bni{1}}^{\bni{2}} (\ffi{1}-\ffi{2})\, dW < \infty$.  An
  identical argument shows $\int_{\ani{1}}^{\ani{2}} (\ffi{1}-\ffi{2})dW <
  \infty$.  By integration by parts,
  \begin{equation}
    \label{eq:unq:term2}
    \begin{split}
      \int_{\bni{1}}^{\bni{2}} \left(\ffi{1}-\ffi{2}\right)(u)dW(u)
      & = \int_{\bni{1}}^{\bni{2}} \left(\ffi{1}-\ffi{2}\right)(u) d(W(u)-W(\bni{1})) \\
      & = \left(W(\bni{2})-W(\bni{1}\right)
      \left(\ffi{1}(\bni{2})-\ffi{2}(\bni{2}) \right) \\
      & \quad - \int_{\bni{1}}^{\bni{2}}
      \left( W(u)-W(\bni{1}) \right)
      \left( \ffi{1}'(u)-\ffi{2}'(u) \right) du \\
    \end{split}
  \end{equation}
  where we can take $\ffi{i}'$ to be the right-derivative, but this choice is
  inconsequential because of the almost sure continuity of $W$.  Thus, by
  \eqref{eq:unq:ffi_ff0_sup} and \eqref{eq:unq:ffip_ff0p_sup}, for all $n$,
  with probability $1-\epsilon$, we can conclude that $
  \big| \int_{\bni{1}}^{\bni{2}} \left(\ffi{1}-\ffi{2}\right)(u)dW(u) \big|$
  is bounded above by
  \begin{equation*}
     K_\epsilon \left( \left|W(\bni{2})-W(\bni{1})\right|
      + \left| \int_{\bni{1}}^{\bni{2}} (W(u)-W(\bni{1}))du  \right| \right).
  \end{equation*}
  Lemma~\ref{lem:unq:tight_integrals} shows for $i=1,2$ that
  $\int_{\bni{1}}^{\bni{2}} (\ffi{i}-\ffol)^2 d\lambda < K_{\epsilon,2}$ with
  probability $1-\epsilon$.  Thus since, by \eqref{eq:unq:bdiff_Op1}, $\left(
    \left|W(\bni{2})-W(\bni{1}) \right| + \left| \int_{\bni{1}}^{\bni{2}}
      (W(u)-W(\bni{1}))du \right| \right)$ is $O_p(1)$, and since this argument
  is perfectly symmetrical and applies to the interval $[\ani{1},\ani{2}]$, we
  have now shown that the right hand side of \eqref{eq:unq:contradictionIneq}
  is $O_p(1)$ and thus finite almost surely, as desired.

  Now that we have shown that $\int_{-\infty}^{\infty} (\ffi{1}-\ffi{2})^2
  d\lambda < \infty$ almost surely, we can conclude that
  \begin{equation}
    \label{eq:unq:ffi_ff0_supToZero}
    \Vert \ffi{1}-\ffi{2} \Vert_{n}^{\infty}  \to 0
  \end{equation}
  almost surely as $n \to \infty$, and now using
  \eqref{eq:unq:ffi_ff0_supToZero} with arguments similar to those used above,
  we will show $\int_{-\infty}^{\infty} \left( \ffi{1}-\ffi{2} \right)^2
  d\lambda = 0 $ almost surely.  By \eqref{eq:unq:ffi_ff0_supToZero}, Lemma
  \ref{lem:unq:tight_integrals} below allows us to conclude that almost surely
  $\int_{\bni{1}}^{\bni{2}} | \ffi{1}'-\ffol' | d\lambda \to 0$.  Thus we can
  reexamine \eqref{eq:unq:term2} and
  see that the right side is bounded above by
  \begin{align*}
    & \left|W(\bni{2})-W(\bni{1})\right| \left|\ffi{1}(\bni{2})-\ffi{2}(\bni{2}) \right|
    + ||W(\cdot)-W(\bni{1})||_{\bni{1}}^{\bni{2}} \int_{\bni{1}}^{\bni{2}} |\ffi{1}'(u)-\ffi{2}'(u)|du \\
    & \le \epsilon \left( \left|W(\bni{2})-W(\bni{1})\right|
      + ||W(\cdot)-W(\bni{1})||_{\bni{1}}^{\bni{2}} \right),
  \end{align*}
  where we may choose $n$ large enough to make the inequality occur with
  probability $1-\epsilon$ for any positive $\epsilon$.  Thus, since $\left(
    \left|W(\bni{2})-W(\bni{1})\right| +
    ||W(\cdot)-W(\bni{1})||_{\bni{1}}^{\bni{2}} \right) = O_p(1)$, we have
  shown that we may choose $n$ large enough that with probability $1-\epsilon$
  \begin{equation}
    \label{eq:unq:term2_small}
    \left| \int_{\bni{1}}^{\bni{2}} \left(\ffi{1}(u)-\ffi{2}(u)\right) \,
      dW(u) \right| \le \epsilon.
  \end{equation}
  Next we show that the other term in \eqref{eq:unq:contradictionIneq},
  $\int_{A_n} \left( \left(\ffi{1}-\ffol\right)^2 -
    \left(\ffi{2}-\ffol\right)^2 \right) \, d\lambda / 2$, is small.
  By Lemma~\ref{lem:unq:tight_integrals},
  for any $\epsilon > 0 $ we may pick an $M_\epsilon$ such
  that both $|\int_{\bni{1}}^{\bni{2}} (\ffi{1}-\ffol) \, d\lambda | $ and
  $\bni{2}-\bni{1}$ are bounded by $M_\epsilon$ with probability
  $1-\epsilon$. Thus, defining $\epsilon_2 = \epsilon / M_\epsilon$ we take $n$
  large enough such that with probability $1-\epsilon$ we have
  $||\ffi{1}-\ffi{2}||_{n}^\infty < \epsilon_2$. Then let
  $\delta(t)=\ffi{1}(t)-\ffi{2}(t)$ and conclude that
  \begin{align*}
    \int_{\bni{1}}^{\bni{2}} (\ffi{1}-\ffol)^2 d\lambda
    & = \int_{\bni{1}}^{\bni{2}} (\ffi{2}-\ffol + \delta)^2 d\lambda \\
    & \le \int_{\bni{1}}^{\bni{2}} (\ffi{2}-\ffol)^2 d\lambda
    + 2 \epsilon_2 \int_{\bni{1}}^{\bni{2}} |\ffi{2}-\ffol| d\lambda
    + \epsilon_2^2 (\bni{2}-\bni{1}),
  \end{align*}
  and that the above display is bounded above by
  \begin{equation*}
    \int_{\bni{1}}^{\bni{2}} (\ffi{2}-\ffol)^2 \, d\lambda
    + \epsilon +
    \left(\frac{\epsilon}{M_\epsilon}\right)^2     M_\epsilon
    \le \int_{\bni{1}}^{\bni{2}} (\ffi{2}-\ffol)^2 \, d\lambda
    + 2 \epsilon ,
  \end{equation*}
  with probability $1-2\epsilon$ and $n$ large enough. Similarly, $
  \int_{\bni{1}}^{\bni{2}} (\ffi{2}-\ffol)^2 d\lambda \le
  \int_{\bni{1}}^{\bni{2}} (\ffi{1}-\ffol)^2 d\lambda + 2\epsilon$,
  and thus
  \begin{equation}
    \label{eq:unq:term1_small}
    \left| \inv{2}\int_{A_n} \left( \left(\ffi{1}-\ffol\right)^2 -
        \left(\ffi{2}-\ffol \right)^2 \right) d\lambda \right | \le \epsilon
  \end{equation}
  with probability $1-2\epsilon$.  Thus we have shown that with probability
  approaching $1$ both terms in \eqref{eq:unq:contradictionIneq} are bounded
  by $\epsilon$ as $n $ goes to infinity. Thus since $\int_{-n}^{n}
  (\ffi{1}-\ffi{2})^2 d\lambda$ is non decreasing in $n$, $\int_{-n}^{n}
  (\ffi{1}-\ffi{2})^2 d\lambda < \epsilon$ with probability $1-\epsilon$ and
  thus it must be $0$ almost surely.
\end{proof}

The following lemma translates Lemma~\ref{lem:unq:tight_t} into a more direct
tightness result.
\begin{lemma}
  \label{lem:unq:tight_integrals}
  Let the definitions and assumptions of
  Theorem~\ref{thm:MC-process-uniqueness-theorem} hold and let
  $h_0(t)=-12t^2$.  Let $\ffi{i}$, $i=1,2$, be as in \eqref{eq:unq:defn:fi}
  and $\ani{i}$ and $\bni{i}$ as defined on page~\pageref{unq:ani_bni}.  We
  then have
  \begin{equation}
    \bni{2}-\bni{1} = O_p(1). \label{eq:unq:bdiff_Op1}
  \end{equation}
  Furthermore, for $i=1,2$ and any $\epsilon > 0$ and $k>0$, there exist
  $K_\epsilon,K_{\epsilon,k} >0$ such that
  with probability greater than $1-\epsilon$ we have
  \begin{align}
    ||\ffi{i}-\ffol||_{\bni{1}}^{\bni{2}} < K_\epsilon  \label{eq:unq:ffi_ff0_sup} \\
    ||\ffi{i}'-\ffol'||_{\bni{1}}^{\bni{2}} < K_\epsilon,  \label{eq:unq:ffip_ff0p_sup}
  \end{align}
  (in which we take $\ffi{i}'$ to be either the right of the left derivative)
  and thus that
  \begin{equation}
    \label{eq:unq:ffi_ff0_int}
    \int_{\bni{1}}^{\bni{2}} |\ffi{i}-\ffol|^k d\lambda < K_{\epsilon,k},
  \end{equation}
  where $K_\epsilon $ and $K_{\epsilon,k} $ do not depend on $n$.
  Further, if almost surely
  $\Vert \ffi{1} - \ffi{2} \Vert_n^\infty \to 0$ as $n \to \infty$ then we can
  conclude that almost surely
  \begin{equation}
    \label{eq:unq:ffip_ff0p_int}
    \int_{\bni{1}}^{\bni{2}} | (\ffi{i})'- \ffol' | \, d\lambda \to 0
  \end{equation}
  as $n \to \infty$, for $i=1,2$.
  The statements also hold if we replace
  $\bni{1}$ by $\ani{2}$ and $\bni{2}$ by $\ani{1}$.
\end{lemma}
\begin{proof}
  \eqref{eq:unq:bdiff_Op1} follows
  immediately from Lemma \ref{lem:unq_knotsTight}.

  Next we will show \eqref{eq:unq:ffi_ff0_sup} and \eqref{eq:unq:ffip_ff0p_sup}.  Let
  $g_1$ and $g_0$ be monotone functions.
  Then for any $t \in [a,b]$, we have that
  \begin{equation*}
    g_1(t)-g_0(t)  \le g_1(b)-g_0(a)  = g_1(b)-g_0(b)+g_0(b)-g_0(a)
  \end{equation*}
  and similarly $g_0(t)-g_1(t) \le g_0(b)-g_0(a) + g_1(a)-g_0(a)$. Thus
  \begin{align*}
    |(g_1-g_0)(t)| & \le |(g_1-g_0)(b)| + |(g_1-g_0)(a)| + g_0(b)-g_0(a).
  \end{align*}
  By monotonicity and Lemma \ref{lem:unq:tight_t}, we can say
  \begin{equation*}
    || \ffi{i}'-\ffol' ||_{\bni{1}}^{\bni{2}} < 2M_\epsilon + \ffol'(n+M_\epsilon)-\ffol'(n)
    = 2M_\epsilon + 24M_\epsilon,
  \end{equation*}
  where $\ffi{i}'$ refers to either the left or the right derivative.
  This is independent of $n$ thanks to the linearity of $\ffol'$.
  Thus we have shown \eqref{eq:unq:ffip_ff0p_sup}.

  Now we establish \eqref{eq:unq:ffi_ff0_sup}.  Fix $i \in \lb 1,2 \rb$.  We
  will apply Lemma~\ref{fact:unq:convexFunDiffs} twice, with $\ffi{i}$ as
  $g_1$ and $\ffol$ as $g_0$, and then with the reverse assignments.  We let
  $[a,b] = [n, n+M]$.  Regardless of the choice of which is $g_1$ and which
  is $g_0$, we can bound the first two terms in
  \eqref{eq:unq:convexFunDiffs}, the weighted differences $\lambda
  (g_1(n+M)-g_0(n+M)) + (1-\lambda) (g_1(n)-g_0(n)) $, by $2M_\epsilon$ with
  probability $1-\epsilon$, independently of $n$, by
  Lemma~\ref{lem:unq:tight_t}.  If $\ffol$ is $g_0$ then for the third term
  of \eqref{eq:unq:convexFunDiffs} we have to bound $(h_0'(n+) - h_0'(n+M-)) =
  24M$ which is independent of $n$.  If $\ffi{i}$ is $g_0$, then for the
  third term of \eqref{eq:unq:convexFunDiffs} we have that $ |
  \ffi{i}'(n+M-)-\ffi{i}'(n+)) |$ is bounded above by
  \begin{align*}
    | \ffi{i}'(n+M-) - \ffol'(n+M)|
    + |\ffol'(n)-\ffi{i}'(n+)| + \ffol'(n+M)-\ffol'(n)
  \end{align*}
  which we can again bound independently of $n$ by the linearity of
  $\ffol'$ and Lemma \ref{lem:unq:tight_t} with probability
  $1-\epsilon$.  Since $[\bni{1},\bni{2}] \subset [n,n+M]$ with
  probability $1-\epsilon$ for appropriately large $M$, and since
  $(n+M-t)(t-n)/M \le (M/2)^2 / M$ which is independent of $n$ and
  $t$, the bound is independent of $n$ or $t$.
  Thus we have shown \eqref{eq:unq:ffi_ff0_sup}.  Then
  \eqref{eq:unq:ffi_ff0_int} follows immediately from
  \eqref{eq:unq:ffi_ff0_sup} and \eqref{eq:unq:bdiff_Op1}, since we can bound
  $ \int_{\bni{1}}^{\bni{2}} |\ffi{i} - \ffol|^k d \lambda \le
  \int_{\bni{1}}^{\bni{2}} K_\epsilon^k d\lambda \le K_\epsilon^k \cdot
  K_\epsilon,$ with probability $1-\epsilon$.

  Finally, we show that if for a random outcome $\omega$,
  $||\ffi{1}^\omega-\ffi{2}^\omega ||_n^\infty \to 0$ as $n \to \infty$ then
  \eqref{eq:unq:ffip_ff0p_int} follows.  First, note for any $a,b$ that if
  $\epsilon < \int_{a}^{b} ((\ffi{i}^\omega)'-\ffol')d\lambda =
  (\ffi{i}^\omega-\ffol)(b)-(\ffi{i}^\omega-\ffol)(a)$, and if
  $(\ffi{i}^\omega-\ffol)(b) > \epsilon/2$ then $ (\ffi{i}^\omega-\ffol)(a) <
  - \epsilon/2$.  Similarly, if $-\epsilon > \int_{a}^{b} (
  (\ffi{i}^\omega)'-\ffol')d\lambda $ we can conclude that
  $(\ffi{i}^\omega-\ffol)$ at $a$ or at $b$ is larger than $\epsilon/2$ in
  absolute value.  Since we can take $n$ large enough that $| \ffi{i}^\omega
  -\ffol|$ is less than $\epsilon/2$ at any $a,b > n$, by contradiction we
  have that $| \int_{a}^{b} ( (\ffi{i}^\omega)'-\ffol')d\lambda | < \epsilon$
  for such $a$ and $b$.  Now, since $\left\{ t \in [\bni{1},\bni{2}) :
    (\ffi{1}^\omega)'(t)> (\ffi{2}^\omega)'(t) \right\}$ and $\left\{ t \in
    [\bni{1},\bni{2}) : (\ffi{1}^\omega)'(t) \le (\ffi{2}^\omega)'(t)
  \right\}$ are both intervals by monotonicity of $( \ffi{1}^\omega) '$ and
  linearity of $(\ffi{2}^\omega)'$ on $[\bni{1},\bni{2}]$, we can conclude
  that $\int_{a}^{b} |(\ffi{i}^\omega)'-\ffol'|d\lambda < \epsilon$ as desired.
\end{proof}

\subsubsection{Proof for symmetric limit process}
\label{ssec:proofs_symm-lim-process}

\begin{proof}[Proof of Theorem~\ref{thm:symm-process-uniqueness-theorem}]
  The proof follows as in the proof of Theorem~\ref{thm:MC-process-uniqueness-theorem} in the previous section, where we replace $H_R$ by $H^+$, $\tau_+^0$ by $\tau_+^+$, $\tau_R^0$ by $\tau_R^+$, and we take $H_L$ to be $0$, and also replace $\tau_L$ and $\tau_+^0$ by $0$.

In particular, we can see that analogs of the characterizing equalities and inequalities hold, and that
an analog of Lemma~\ref{lem:process-uniqueness-theorem-constant-interval} holds.
The analog of Lemma~\ref{lem:process-uniqueness-theorem-constant-interval} can  immediately be seen to hold  since by definition $(\psia)^{(2)}(t)=0$ for $t \in (0, \tau_+^+)$.
Proofs similar to the proofs of Proposition~\ref{pro:charznEqualityRewrite_full} and
Proposition~\ref{pro:charzn0_full} show the following characterizing proposition holds.
\begin{proposition}
  \label{pro:symmetric-proc:charzn}
  If $\Delta \colon [0,\infty) \to \RR$ is concave with maximum at $0$ and if $0 \le b \in \widehat{S}^+$, then
  \begin{equation*}
    \int_0^b \Delta(t) \lp \psia(t)dt - dX(t) \rp \ge 0.
  \end{equation*}
  And if $\Delta = \psia$ then the inequality is an equality.
\end{proposition}
In the proof of this proposition we define $F_R^+(u) = \int_{\tau_R^+}^u \psia(v)dv$ and $X_R^+(u) = \int_{\tau_R^+}^u dX(v)$ and then use the fact that by \eqref{eq:12}, $(F_R^+-X_R^+)(0) = 0$.
The proof of Theorem~\ref{thm:symm-process-uniqueness-theorem} then  proceeds analogously to the proof of Theorem~\ref{thm:MC-process-uniqueness-theorem}.  This completes the proof of Theorem~\ref{thm:symm-process-uniqueness-theorem}.
\end{proof}

\subsubsection{Proofs for pointwise limit theory}
\label{subsec:proofs-pointwise-limit-theory}

Theorem~\ref{thm:BRW2009-UMLE-limit} %
A %
about the unconstrained estimator is Theorem~2.1 of \cite{BRW2007LCasymp}.
Part B of the theorem
is then proved in an identical fashion as that theorem, because for $n$ large enough, in an $n^{-1/5}$ neighborhood of $x_0 \ne m$, the constrained and unconstrained estimators satisfy the same characterization.   %
Part C then follows from Part B. %
Thus we focus first on proving Theorem~\ref{thm:MC-MLE-limit} \ref{thm:MC-MLE-limit:item-A},
where $x_0 = m$, which follows from Theorem~\ref{thm:MC-MLE-limit-A-process-version}.

The statements in
\eqref{thm:mode-ff-asymptotics}
about the limit distribution of $\ffna$ and $(\ffna)'$ now follow from the delta method (Taylor expansion).  We devote the entire remainder of this subsection to proving Theorem~\ref{thm:MC-MLE-limit-A-process-version}.  All the lemmas contained in the proof use the same notation, and have the same hypotheses as the theorem.
An outline of the structure of the proof can be found in Section~\ref{sec:Proofs-outlines}.

\begin{proof}[Proof of Theorem~\ref{thm:MC-MLE-limit-A-process-version}]
  Let $b \in \RR$ denote our ``local'' parameter and let
  \begin{equation}
    \label{eq:13}
    \tnb \equiv m + b n^{-1/5}
  \end{equation}
  be the ``global'' parameter.   We also let $\snL$ be any  knot (sequence)  of $\vvna$ strictly less than $m$ satisfying $n^{1/5}(\snL -m) = O_p(1)$ and let $\snR$ be any knot (sequence) of $\vvna$ strictly larger than $m$ satisfying $n^{1/5}(\snR-m)=O_p(1)$.  Recall $\lambda$ is Lebesgue measure, and define
  \begin{align*}
    \YYnffo(b) & \equiv \xn^{4/5} \IntUCo
    \left( \IntUCi (d\Fn   -  \ffo(m)d\lambda) \right) dv, \\
    \HHnffo(b) & \equiv \xn^{4/5} \IntUCo \left( \IntUCi (\ffn - \ffo(m))d\lambda  \right) dv
    + \An b + \Bn ,
  \end{align*}
and
  \begin{align*}
    \YYnffaL(b) & \equiv\xn^{4/5} \IntLo \left(\IntLi (d\Fn  - \ffo(m)d\lambda) \right) dv, \\
    \YYnffaR(b) & \equiv \xn^{4/5} \IntRo \left(\IntRi (d\Fn -  \ffo(m) d\lambda) \right) dv, \\
    \HHnffaL (b) &\equiv \xn^{4/5} \IntLo \left( \IntLi (\ffna - \ffo(m))d\lambda \right) dv
    + \AnL n^{1/5} (\snL - \tnb) , \\
    \HHnffaR(b) &\equiv \xn^{4/5} \IntRo \left( \IntRi (\ffna -
      \ffo(m))d\lambda \right) dv
    + \AnR n^{1/5}(\tnb - \snR),
  \end{align*}
where
\begin{align*}
  A_n = \xn^{3/5} \left( \FFn(\mm) - \Fn(\mm) \right)
  \quad  \mbox{ and }  \quad B_n   =  \xn^{4/5} \left( \HHn(\mm) - \YYn(\mm)\right),
\end{align*}
and
\begin{align}
  \AnL & = \xn^{3/5}\left(\FnL(\snL)-\FFnaL(\snL)\right),  \label{eq:defn:AnL} \\
  \AnR & = \xn^{3/5}\left(\FnR(\snR)-\FFnaR(\snR)\right) \label{eq:defn:AnR}
\end{align}
(recalling the definitions of $\FnL$ and $\FnR$ from \eqref{eq:def:finite-sample-L-R-processes}). Additionally, let
\begin{align*}
  \XXnffo(b) &:= (\YYnffo)'(b) = \xn^{3/5} \int_{\mm}^{\tnb} \left( d\Fn- \ffo(m)d\lambda  \right), \\
  \XXnffaL(b) &:= -(\YYnffaL)'(b) = \xn^{3/5} \int_{\tnb}^{\snL} \left( d\Fn - \ffo(m)d\lambda \right), \\
  \XXnffaR(b) &:= (\YYnffaR)'(b) = \xn^{3/5} \int_{\snR}^{\tnb} \left( d\Fn- \ffo(m)d\lambda \right).
\end{align*}
The terms for the constrained processes which would correspond to $B_n$ turn out to be $0$.  Also,  $\AnL$ and $\AnR$ appear to be off by a sign change when compared with $A_{\xn}$: this is because of the definitions of our left- and right-processes, which entails, e.g., $(\HHnaR-\YYnR)'(t) = -(\FFnaR-\FnR)(t)$.
Note that in \cite{BRW2007LCasymp}, $\YYnffo$ is denoted by
$\mathbb{Y}_n^{loc}$ and
similarly for $\HHnffo$.

The proof proceeds as follows. We will  derive the limit distribution for the empirical process-type $\YY$ and $\XX$ terms.  We will show that the estimator-type $H$ terms (and appropriate derivatives) are tight, and also satisfy characterizations analogous to those given in
Theorem~\ref{thm:charzn_uniqueness_full_UC}
and
Theorem~\ref{thm:MC-process-uniqueness-theorem}.  We argue then (by a continuous mapping argument) that a characterization must hold in the limit (along subsequences, using tightness of the $H$ processes) and
then apply
Theorem~\ref{thm:charzn_uniqueness_full_UC}
and
Theorem~\ref{thm:MC-process-uniqueness-theorem} to conclude that the limit is as desired.

For $0 < c \le \infty $, define
\begin{align*}
  \mathcal{C}_c &= \{h | h:[-c,c] \to \RR, h \mbox{ is continuous} \} \\
  \mathcal{D}_c &= \{h | h: (-c,c) \to \RR,
  h \mbox{ is cadlag and bounded} \},
\end{align*}
where ``cadlag'' means right-continuous functions which have limits from the left.  If $c = \infty$ then we we interpret the definition of $\mathcal{C}_\infty$ to mean continuous functions $h$ defined on $(-\infty,\infty)$.  We let $\Vert f \Vert$ be the supremum of $f$ over its domain, and this is the distance we use in $ \mathcal{C}_c $ when $c < \infty$. When $c = \infty$ we use the topology of
convergence on all compacta (see \cite{Whitt:1970ut}).  For $\mathcal{D}_c$ the uniform norm is too strong, so generally one uses a Skorokhod norm (\cite{Skorokhod1956}, see also \cite{Billingsley1999CPM}).  We endow, for the moment, $\mc{D}_c$ with the $J_1$ Skorokhod norm (referred to as ``the'' Skorokhod topology in chapter 12 of \cite{Billingsley1999CPM}).  When we come to proving tightness of our $H$-type processes we will further discuss topological details.  Now we focus on the empirical processes.  We let ``$A_n \Rightarrow A$'' mean that $A_n$ converges weakly to $A$ in a space that will be specified in each context \citep{Billingsley1999CPM}.  The proof of the following lemma is standard but we will refer to it several times so we provide it here. Recall $t_{n,b} = m + b n^{-1/5}$, for $b \in \RR$, and recall that $\snL, \snR$ are chosen to be of order $n^{-1/5}$ from $m$ in probability.  Let    $\DD_n = \Fn - \FFo$, and define processes
 $\underline{\AAA}_n (b) \equiv (\AAA_{n,1} (b), \AAA_{n,2} (b), \AAA_{n,3} (b) )$ and
  $\underline{\BB}_n (b) \equiv (\BB_{n,1} (b), \BB_{n,2} (b), \BB_{n,3} (b) )$ for $b \in \RR$  by
 \begin{eqnarray*}
  \underline{\AAA}_n (b) \equiv
  \frac{n^{3/5}}{ \sqrt{f_0(m)}} \left (
  \int_m^{t_{n,b}} d\DD_n,
  \int_{t_{n,b}}^{s_{n,L} } d\DD_n,
  \int_{s_{n,R}}^{t_{n,b}} d\DD_n  \right ) ,
  \end{eqnarray*}
  and
  \begin{eqnarray*}
  \underline{\BB}_n (b) \equiv
  \frac{n^{4/5}}{ \sqrt{f_0 (m)}}  \left (
  \int_m^{t_{n,b}} \int_m^v d\DD_n dv,
  \int_{t_{n,b}}^{s_{n,L}} \int_{v}^{s_{n,L}} d\DD_n dv,
  \int_{s_{n,R}}^{t_{n,b}} \int_{s_{n,R}}^v d\DD_n dv  \right ).
  \end{eqnarray*}
  Let $\nu_{n,R} = n^{1/5}( s_{n,R}-m)$ and $ \nu_{n,L} = n^{1/5}(s_{n,L}-m)$.
  For a (sequence of) Brownian motion processes $W \equiv W_n$ on $\RR$ we also define corresponding approximating (nearly Gaussian) processes $\underline{G}_{n} \equiv (G_{n,1}, \ldots, G_{n,6})$ by
  \begin{align*}
    (G_{n,1}, G_{n,2}, G_{n,3})
    =  \left ( \int_0^b dW,
      \int_b^{\nu_{n,L}} dW,
      \int_{\nu_{n,R}}^b dW \right ),
  \end{align*}
  and
  \begin{align*}
    (G_{n,4}, G_{n,5}, G_{n,6})
    =     \left (  \int_0^b \int_0^v dW dv,
      \int_b^{\nu_{n,L}} \int_{v}^{\nu_{n,L}} dW  dv,
      \int_{\nu_{n,R}}^b \int_{\nu_{n,R}}^v dW dv
    \right )   .
  \end{align*}
  \begin{lemma}
  \label{lem:local:Yconvergence}
  The vector of processes $(\AAA_n(b), \BB_n(b))$ can be defined on a common probability space  with a sequence of Brownian motion processes $W \equiv W_n$ so that
  \begin{equation*}
    \sup_{b \in [-c,c]} | ( \AAA_n(b) , \BB_n(b) )  - \underline{G}_n(b) | \to_p 0.
  \end{equation*}
\end{lemma}
\begin{proof}
  We prove that the difference of $\AAA_{n,1}(b) = n^{3/5} \int_m^{\tnb} d\DD_n$ and $G_{n,1}$ converges to $0$ in probability uniformly in $|b| \le c$ and that the difference of $\BB_{n,1}(b) = n^{4/5} \int_m^{\tnb} \int_m^v d\DD_n dv$ and $G_{n,4}(b) $ converges to $0$ in probability uniformly in $|b| \le c$.   The proofs for the other components are analogous.  We can see $n^{3/5}  \int_m^{\tnb} d\DD_n$ is equal in distribution to
  \begin{equation}
    \label{eq:local:BB} %
    n^{1/10} \lp \UU_n( \FFo(\tnb)) - \UU_n (\FFo(m)) \rp
  \end{equation}
  where $\UU_n(t) \equiv \sqrt{n} (\FF_n^*(t)-t)$
  is the empirical process corresponding to $\Fn^{*}(t)$, the empirical d.f.\ for $\xn$ i.i.d.\ uniform random variables.   By a Skorokhod construction (see e.g.\ Theorem 12.3.4, page 502, of
  \cite{MR3396731}; %
  or see \cite{MR893903}) %
  there exist a sequence of Brownian bridge processes $B_{\xn}$ such that $\Vert \UUn - B_{\xn} \Vert = O(\log(\xn)\xn^{-1/2})$ almost surely.  Thus, \eqref{eq:local:BB} is equal to
  \begin{align}
    \label{eq:local:BBseq}
    \xn^{1/10} (B_n(\FFo(t_{n,b})) - B_n(\FFo(\mm)) ) +  n^{-4/10}
      \log(\xn)  M_n(b)  ,
  \end{align}
  where for all $|b| \le c$, $0 \le M_n(b) \le M = O(1) $ almost surely.
  Next we use that $B_n(t) = W_n(t) - tW_n(1)$ where $W_n(t)= B_n(t)+t N$ is a Brownian motion and $N$ is a
  standard Normal random variable.
  Thus \eqref{eq:local:BBseq} equals
  \begin{align*}
    \MoveEqLeft \xn^{1/10} \left( W_n(\FFo(t_{n,b})) - W_n(\FFo(\mm)) -
      (\FFo(t_{n,b})-\FFo(\mm))W_n(1) \right)
    + o_p(1),
  \end{align*}
  which is equal to
  \begin{align*}
    \MoveEqLeft W_n(b) \sqrt{\xn^{1/5}(\FFo(t_{n,b})-\FFo(\mm)) / b} - W_n(1)
    \ffo(\mm) b \xn^{-1/10}
    + o_p(1) \\
    & = W_n(b) \sqrt{\ffo(\mm)} +  o_p(1).
  \end{align*}
  This shows that $\sup_{b \in [-c,c]} | \AAA_{n,1}(b) - G_{n,1}(b)| = o_p(1)$.   Using this, we see that the process $n^{4/5} \int_m^{\tnb} \int_{m}^v d\DD_n dv $ defined on this probability space equals
  \begin{equation*}
    n^{1/5} \int_m^{\tnb} \lp \sqrt{\ffo(m)} W_n ( n^{1/5}(v - m)) + o_p(1) \rp dv
    =   \sqrt{\ffo(m)} \int_0^b W_n(v)dv + o_p(1)
  \end{equation*}
with the $o_p(1)$ error still uniform in $|b| \le c$.  Thus $\sup_{b \in [-c,c]} | \BB_{n,1}(b) - G_{n,4}(b)| = o_p(1)$.  This completes the proof for two of the terms and the other four are analogous.
\end{proof}

\begin{lemma}
  \label{lem:MC-dataprocs-ff-limit}
  Let $\underline{P}_{\xn} \equiv (P_{n,1}, \ldots, P_{n,6})$ be a vector of drift terms where we let
  \begin{align*}
    P_{n,1}(b) & = \inv{6} b^3,  &
    P_{n,2}(b) & =   \int^{n^{1/5}(\snL - m)}_b  \inv{2} u^2 du, \\
    P_{n,3}(b)&  =\int_{n^{1/5}(\snR - m)}^b  \inv{2} u^2 du, &
    P_{n,4}(b) & =  \inv{24} b^4,
  \end{align*}
  and
  \begin{align*}
    P_{n,5}(b) & = \int_b^{n^{1/5}(\snL - m)}  \int^{n^{1/5}(\snL - m)}_v  \inv{2} u^2 dudv,
    \\
    P_{n,6}(b) & = \int^b_{n^{1/5}(\snR - m)}  \int_{n^{1/5}(\snR - m)}^v  \inv{2} u^2 dudv.
  \end{align*}
  Then the vector of processes $( \underline{\XX}_n^f , \underline{\YY}_n^f) \equiv \left( \XXnffo, \XXnffaL, \XXnffaR, \YYnffo, \YYnffaL, \YYnffaR \right)$ can be defined on a common probability space  with a sequence of Brownian motion processes $W \equiv W_n$ such that for $ 0 < c < \infty$
  \begin{equation*}
    \begin{split}
      \sup_{b \in [-c,c]}
      | ( \underline{\XX}_n^f(b) , \underline{\YY}_n^f(b))  -
      \sqrt{\ffo(\mm)}
      \underline{G}_{\xn}(b)
      -\ffo''(\mm) \underline{P}_{\xn}(b) |
      \to_p 0
      \quad \mbox{ as } n \to \infty,
    \end{split}
  \end{equation*}
  where $\underline{G}_n$ is as in Lemma~\ref{lem:local:Yconvergence}.
\end{lemma}
\begin{proof}
  We will show that the statement holds for  $\XXnffaR$.  The proof for the other $\XX$ terms and for all the $\YY$ terms are similar.  Now, $n^{-3/5} \XXnffaR(b)$ equals
  $\int_{\snR}^{t_{n,b}} d\DD_n + \int_{\snR}^{t_{n,b}} (\ffo - \ffo(m))d\lambda$; by a second-order Taylor expansion of $\ffo$, since $\ffo'(m)=0$, the second term equals
  \begin{align*}
    \int_{\snR}^{t_{n,b}} (\ffo^{(2)}(m)+o(1)) \frac{ (w-m)^2}{2} dw
    = n^{-3/5} (\ffo^{(2)}(m)+o(1)) \int_{n^{1/5}(\snR-m)}^b \frac{u^2}{2}du
  \end{align*}
  where the $o(1)$ is uniform in $|b| \le c$.  By Lemma~\ref{lem:local:Yconvergence}, we have shown $\XXnffaR$ has the desired limit.
\end{proof}

The next lemma shows that the localized processes still satisfy the characterizing system of equalities/inequalities of the global estimators; once these characterizing equalities/inequalities are carried over to the limit they will allow us to identify the limit distribution, via Theorems~\ref{thm:charzn_uniqueness_full_UC} and \ref{thm:MC-process-uniqueness-theorem}.

\begin{lemma}
  \label{lem:local-f-characterization}
  For $b \in \RR$,
  \begin{align}
    \YYnffaL(b)-\HHnffaL(b) \ge 0, \label{eq:local:YLbiggerHL-f} \\
    \YYnffaR(b)-\HHnffaR(b) \ge 0, \label{eq:local:YRbiggerHR-f}
  \end{align}
  and
  \begin{align}
    \int^{\tau_{\xn,-}^0}_{\infty} \lp \YYnffaL(b)-\HHnffaL(b) \rp d(\HHnffaL)^{(3)}(b)
    & = 0, \label{eq:local:YL-HLf-equality} \\
    \int_{\tau_{\xn,+}^0}^\infty \left( \YYnffaR(b)-\HHnffaR(b) \right) d(\HHnffaR)^{(3)}(b)
    & = 0 \label{eq:local:YR-HRf-equality}
  \end{align}
  where $\tau_{n,-}^0$ is the largest left-knot of $\vvna$ (no larger than $m$) and $\tau_{n,+}^0$ is the smallest right-knot of $\vvna$ (no smaller than $m$).
\end{lemma}
\begin{proof}
  We consider the right-side process, and the left-side ones are analogous. The difference
  $    \YYnffaR(b)-\HHnffaR(b) $ equals
  \begin{equation*}
    \xn^{4/5}
    \IntRo \left( \IntRi d\Fn -     \ffna d\lambda \right) dv -
    (b - n^{1/5}(\snR-m))\AnR ,
  \end{equation*}
  which, by the definition of $\AnR$, equals
  \begin{equation*}
    \begin{split}
      \MoveEqLeft  \xn^{4/5} \IntRo \left( \int_{\snR}^{X_{(\xn)}} (d\Fn(u) -
        \ffna(u)du ) - \int_{v}^{X_{(\xn)}} (d\Fn(u)
        - \ffna(u)du)  \right) dv  \\
      & \quad - \xn^{4/5} \left( t_{n,b} - \snR \right)
      \int_{\snR}^{X_{(\xn)}} (d\Fn(u) - \ffna(u)du) ,
    \end{split}
  \end{equation*}
  which equals
  \begin{equation*}
    -\xn^{4/5} \IntRo \left( \int_{v}^{X_{(\xn)}} (d\Fn(u)
      - \ffna(u)du)  \right) dv,
  \end{equation*}
  which equals
  \begin{equation}
    \label{eq:local:globalToLocalProcs}
    \xn^{4/5} \int_{t_{n,b}}^{X_{(\xn)}} \left( \int_{v}^{X_{(\xn)}} (d\Fn(u)
      - \ffna(u)du)  \right) dv,
  \end{equation}
  since
  $\HHnaR(\snR) - \YYnR(\snR)  = 0$ by
  Theorem~\ref{thm:CharThmTwo}~\ref{thm:CharThmTwo:B-MC}
  Thus,
  \begin{equation*}
    \YYnffaR(b)-\HHnffaR(b) = \xn^{4/5} \left(\YYnR(t_{n,b})-\HHnaR(t_{n,b}) \right) \ge 0
  \end{equation*}
  for all $b \ge 0$, with equality if $t_{n,b}$ is a right-knot, by
  Theorem~\ref{thm:CharThmTwo}~\ref{thm:CharThmTwo:B-MC} %
  We have thus shown \eqref{eq:local:YRbiggerHR-f} and \eqref{eq:local:YR-HRf-equality}.
  Showing \eqref{eq:local:YLbiggerHL-f} and \eqref{eq:local:YL-HLf-equality} is analogous.
\end{proof}

In order to show tightness of our $H$-processes, we want to apply Proposition~\ref{cor:local:vp_vpprime_tight}, which is at the log-density level.  To do this, we need to translate from the $f$-processes to processes defined at the log level, which we will refer to as $\vp$-processes.
Let
\begin{equation}
  \label{eq:18}
  \begin{split}
    \YYnvvo(b) &= \frac{\YYnffo(b)}{\ffo(\mm)}
    - \xn^{4/5} \IntUCo \IntUCi R(u) dudv, \\
    \HHnvvo(b) & = \frac{\HHnffo(b)}{\ffo(\mm)}
    - \xn^{4/5} \IntUCo \IntUCi R_{\xn}(u)  du dv ,
    \\
    \YYnvvaL(b) & =  \frac{\YYnffaL(b)}{\ffo(\mm)}
    - \xn^{4/5}\IntLo \IntLi \RRnaL(u) dudv, \\
    \YYnvvaR(b) & = \frac{\YYnffaR(b)}{\ffo(\mm)}
    - \xn^{4/5}\IntRo \IntRi \RRnaR(u)dudv, \\
    \HHnvvaL(b) & = \frac{\HHnffaL(b)}{\ffo(\mm)}
    - \xn^{4/5} \IntLo \IntLi \RRnaL(u) dudv,
    \\
    \HHnvvaR(b) & = \frac{\HHnffaR(b)}{\ffo(\mm)}
    - \xn^{4/5} \IntRo \IntRi \RRnaR(u)dudv,
  \end{split}
\end{equation}
where
\begin{align*}
  R_{\xn}(u) = \sum_{j=2}^\infty \inv{j!} \left( \vvn(u)-\vvo(\mm)
  \right)^j,
  \quad \mbox{ and } \quad
  \RRna(u) = \sum_{j=2}^\infty \inv{j!} \left( \vvna(u)-\vvo(\mm)
  \right)^{j}.
\end{align*}
Also let
\begin{align*}
  \XXnvvo(v) &= (\YYnvvo)'(v) = \frac{\XXnffo(v)}{\ffo(\mm)}
  - \xn^{3/5} \int_{\mm}^{t_{n,v}} \RRno(u) du, \\
  \XXnvvaL(v) &= -(\YYnvvaL)'(v) = \frac{\XXnffaL(v)}{\ffo(\mm)}
  - \xn^{3/5} \int_{t_{n,v}}^{\snL} \RRnaL(u) du, \\
  \XXnvvaR(v) &= (\YYnvvaR)'(v) = \frac{\XXnffaR(v)}{\ffo(\mm)}
  - \xn^{3/5} \int_{\snR}^{t_{n,v}} \RRnaR(u) du.
\end{align*}
The above definitions are motivated by the following identities.

\begin{lemma}[$f$ to $\vp$ identities]
  \label{lem:identities}
  We have
  \begin{align}
    \ffo(\mm)^{-1}\left( \ffn(u) - \ffo(m) \right) = \vvn(u) - \vvo(m)
    + \RRno(u) \label{eq:local:ffnTOvp}
    \\
    \ffo(\mm)^{-1}\left( \ffna(u) - \ffo(m) \right) = \vvna(u) - \vvo(m)
    + \RRna(u). \label{eq:local:ffnaTOvp}
  \end{align}
The $\vp$-processes thus satisfy
\begin{align}
  \HHnvvo(b)  &= \xn^{4/5} \IntUCo \IntUCi \left( \vvn(u)-\vvo(m) \right)dudv  +
                \frac{\An b + \Bn}{\ffo(\mm)} \label{eq:local:HHnvvo2HHnffo} \\
  \HHnvvaL(b) &= \xn^{4/5} \IntLo \IntLi \left( \vvna(u)-\vvo(m) \right) du dv
                + \frac{\AnL n^{1/5} (\snL- t_{n,b}) }{\ffo(\mm)}
                \label{eq:local:HHnvvaL2HHnffaL} \\
  \HHnvvaR(b) &=  \xn^{4/5} \IntRo \IntRi \left( \vvna(u)-\vvo(m) \right)du dv
                + \frac{\AnR n^{1/5} (t_{n,b} - \snR) }{\ffo(\mm)}
                \label{eq:local:HHnvvaR2HHnffaR}
\end{align}
\end{lemma}
\begin{proof}
  The identities \eqref{eq:local:ffnTOvp} and \eqref{eq:local:ffnaTOvp} are
  just the exponential series expansion about the density at $\mm$,
  \begin{equation}
    \label{eq:exp-expansion}
    \hat{g}(u)-\ffo(\mm) = \ffo(\mm)(e^{\{\hat{\vp}(u)-\vvo(\mm)\}} - 1)
    =\ffo(\mm) \sum_{j=1}^\infty \inv{j!}(\hat{\vp}(u)-\vvo(\mm))^j,
  \end{equation}
  where $\hat{g}$ is either $\ffn$ or $\ffna$, and $\hat{\vp}$ is either $\vvn$ or $\vvna$, respectively.  Now \eqref{eq:local:HHnvvo2HHnffo}, \eqref{eq:local:HHnvvaL2HHnffaL}, and \eqref{eq:local:HHnvvaR2HHnffaR} follow directly from either \eqref{eq:local:ffnTOvp} or \eqref{eq:local:ffnaTOvp}
  and the definitions of the processes.
\end{proof}

\begin{lemma}
  \label{lem:local:analyzeRemainder}
  Taking $b \in [\mm - c\xn^{-1/5},
  \mm+ c\xn^{-1/5}]$ for any $c > 0$, we have
  \begin{equation}
    \RRno(b) =  o_p(\xn^{-2/5})
    \quad \mbox{ and } \quad
    \RRna(b) = o_p(\xn^{-2/5})  ,
    \label{eq:local:analyzeRemainder}
  \end{equation}
  uniformly for $b \in [\mm - c \xn^{-1/5}, \mm+c\xn^{-1/5}]$.
\end{lemma}
\begin{proof}
  Note $(\hat{\vp}(u) - \vvo(m))^j = O_p(n^{-2j/5})$, for $j \ge 2$, from by \eqref{eq:vp_tight} for $\hat{\vp} = \vvna$, and the analogous (4.17) on page 1319 of \cite{BRW2007LCasymp} for $ \hat{\vp} = \vvn$, since $\vvo'(\mm) =0$.  Thus, we have shown \eqref{eq:local:analyzeRemainder}.
\end{proof}

\begin{lemma}
  \label{lem:Y-bigger-H-vv}
  We have
  \begin{align}
    \YYnvvaL(b)-\HHnvvaL(b) \ge 0, \label{eq:local:YLbiggerHL-v}
    \quad \mbox{ for } b \le 0, \\
    \YYnvvaR(b)-\HHnvvaR(b) \ge 0, \label{eq:local:YRbiggerHR-v}
    \quad \mbox{ for } b \ge 0,
  \end{align}
  and
  \begin{align}
    \int^{\tau_{\xn,-}^0}_{\infty} \lp \YYnvvaL(b)-\HHnvvaL(b) \rp d(\HHnvvaL)^{(3)}(b)
    & = 0, \label{eq:local:YL-HLv-equality} \\
    \int_{\tau_{\xn,+}^0}^\infty \left( \YYnvvaR(b)-\HHnvvaR(b) \right) d(\HHnvvaR)^{(3)}(b)
    & = 0 \label{eq:local:YR-HRv-equality}
  \end{align}
  where $\tau_{n,-}^0$ is the largest left-knot of $\vvna$ (no larger than $m$) and $\tau_{n,+}^0$ is the smallest right-knot of $\vvna$ (no smaller than $m$).
\end{lemma}
\begin{proof}
  By the process definitions,
  \begin{equation*}
    \YYnvvaR(b)-\HHnvvaR(b) = \inv{\ffo(\mm)} \left( \YYnffaR(b)-\HHnffaR(b)\right),
  \end{equation*}
  so by Lemma~\ref{lem:local-f-characterization} we can conclude for $b \ge 0$ that
  \begin{equation*}
    \YYnvvaR(b)-\HHnvvaR(b) \ge 0,
  \end{equation*}
  with equality if $t_{n,b}$ is a right-knot, as desired. We have thus shown \eqref{eq:local:YRbiggerHR-v} and \eqref{eq:local:YR-HRv-equality}, and \eqref{eq:local:YLbiggerHL-v} and \eqref{eq:local:YL-HLv-equality} are similar.
\end{proof}

\begin{lemma}
  \label{lem:MC-dataprocs-vv-limit}
  The vector of processes  $ (\underline{\XX}_n^{\vp}, \underline{\YY}_n^{\vp}) \equiv \left( \XXnvvo, \XXnvvaL, \XXnvvaR , \YYnvvo, \YYnvvaL
    , \YYnvvaR \right)$ can be defined on a common probability space with a
  sequence of Brownian motion processes $W \equiv W_n$ such that for $ 0 < c
  < \infty$
  \begin{equation}
    \label{eq:local:vvNoiseProcesses}
    \begin{split}
      \sup_{b \in [-c,c]} \lv
      (\underline{\XX}_n^{\vp}(b), \underline{\YY}_n^{\vp}(b)) -
      \inv{\sqrt{\ffo(\mm)}}
      \underline{G}_{\xn}(b)
      - \vvo''(\mm) \underline{P}_{\xn}(b) \rv
      \to_p 0
      \quad \mbox{ as } n \to \infty,
    \end{split}
  \end{equation}
  where $\underline{G}_{\xn}$ and $\underline{P}_n$ are as in
Lemma~\ref{lem:local:Yconvergence} and
  Lemma~\ref{lem:MC-dataprocs-ff-limit}.
\end{lemma}
\begin{proof}
  By \eqref{eq:local:analyzeRemainder}, since $\vvo''(\mm) = \ffo''(\mm) / \ffo(\mm)$, we can conclude that \eqref{eq:local:vvNoiseProcesses} holds.
\end{proof}

\medskip

We have established the appropriate characterizing properties of the $\vp$-processes, and the limit distribution of the $\YY^\vp$ processes.  It remains to prove tightness of the $H^\vp$-processes.  To begin, we discuss the spaces in which our convergences will occur.  For $0 < c \le \infty $, define
\begin{align*}
  \mathcal{F}_{c,M} & = \left\{ f \in {\cal D}_c |
    f \mbox{ is non-increasing and } \Vert f \Vert \le M \right\},
\end{align*}
where ``cadlag'' means right-continuous functions which have limits from the left.  Our $H$, $H^{(1)}$, and $H^{(2)}$ functions (for the constrained and unconstrained estimators) are continuous and the uniform norm is appropriate for them.  The $H^{(3)}$ type functions lie in $\mc{F}_{c,M}$, and the uniform norm is too strong.  For the convergence of the $\YY$-processes we used the $J_1$ Skorokhod metric.  Unfortunately this also is too strong as it does not allow multiple jumps to approximate a single jump in $\mathcal{F}_{c,M}$ (see Remark B.0.10 in Appendix B of \cite{Doss:2013}).  Thus, we will use the so-called $M_1$\label{page:def:M1} Skorokhod metric on $\mc{F}_{c,M}$.
This is defined in Section~12.3 of \cite{Whitt2002Stochastic}, and discussed in the following sections.  We give a brief introduction here.
The $M_1$ metric is defined as follows.  For a set $A \subseteq \RR$, let $\Vert x - A \Vert := \inf_{y \in A} \vert x-y \vert.$ (Note that we have also taken $\Vert f \Vert$ to be the supremum of a function $f$ over its domain.  It will be clear from context which usage is intended.)  For a function $x$ and $\delta > 0$, let $ w_s(x,\delta) = \sup \Vert x(t_2) - [x(t_1),x(t_3)] \Vert,$ where the $\sup$ is taken over $t_1, t_2,$ and $t_3$ such that $-c \vee (t_2-\delta) \le t_1 < t_2 < t_3 \le c \wedge (t_2 + \delta)$.\footnote{Note that $w_s$ coincides with the definition of $\Delta_{M_1}$ in \cite{Skorokhod1956}.}  Note that since sequences that converge in the $J_1$ topology also converge in the $M_1$ topology (\cite{Whitt2002Stochastic}), the weak convergences proved for the empirical processes in Lemma~\ref{lem:MC-dataprocs-vv-limit} still hold when we use the $M_1$ topology.  By \cite{Whitt:1980er} (see Theorem~B.0.2 of \cite{Doss:2013}), $\mc{F}_{c,M}$ is a complete, separable metric space. And furthermore we have the following.
\begin{proposition}[Lemma~B.0.9 of \cite{Doss:2013}]
  \label{prop:FccM-precompact}
  $\mc{F}_{c,M}$ is precompact, meaning that every sequence in $\mc{F}_{c,M}$ has a convergent subsequence (not necessarily lying in $\mc{F}_{c,M}$).
\end{proposition}
 This is the fundamental property we need for tightness arguments, to which we now proceed.

\begin{lemma}
  \label{lem:local:Htight}
  The processes $(\HHnvvaL)''', (\HHnvvaL)'', (\HHnvvaL)'$ and $\HHnvvaL$ are tight in $\mathcal{D}_c \times \mathcal{C}_c^3$ when $0 < c < \infty$.  The same tightness holds if we replace the $L$-processes by the $R-$processes.
\end{lemma}
\begin{proof}
  We will discuss the tightness for the left-side processes. The argument for
  the right-side processes is analogous.
  Proposition~\ref{cor:local:vp_vpprime_tight}
  shows that for any $\epsilon$, we can take $M>0$ large enough that $(\HHnvvaL)'''$ lies in $\mathcal{F}_{c, M}$ with probability $1-\epsilon$.  Since $\mathcal{F}_{c,M}$ is precompact in $\mathcal{D}_c$ by Proposition~\ref{prop:FccM-precompact}, $(\HHnvvaL)'''$ is tight.  Then $(\HHnvvaL) ''$ is uniformly bounded by Proposition~\ref{cor:local:vp_vpprime_tight}, and since its derivative is uniformly bounded, and since the set of functions with their values as well as the values of their derivatives uniformly bounded by $M$ is compact in $\mathcal{C}_c$ (via the Arzela-Ascoli theorem, see e.g.\ \cite{Royden}), we can conclude that $(\HHnvvaL)''$ is tight in $\mathcal{F}_{c,M}$. Similarly, since integrals on bounded intervals of uniformly bounded functions are also uniformly bounded, and by Lemma~\ref{lem:AnToZero} below, together with the fact that $n^{1/5}(\snL - b)$ is $O_p(1)$ by assumption
we see that $(\HHnvvaL)'$ and $\HHnvvaL$ are uniformly bounded, and their respective derivatives are uniformly bounded, so we can again conclude that they are tight.
  An identical argument works for the right-side
  processes.
\end{proof}

We will want to consider our processes in $\mc{C}_\infty$ and in $\mc{D}_\infty$.  For the continuous processes in $\mc{C}_\infty$, Corollary~5 of \cite{Whitt:1970ut} says that processes that are tight in $\mc{C}_c$ for all $0 < c < \infty$ are then tight in $\mc{C}_{\infty}$.  By Theorem~12.9.3 of \cite{Whitt2002Stochastic} (with Prohorov's theorem, e.g.\ \cite{VW1996WCEP} page 21),  processes that are tight in $\mc{D}_c$, $0 < c < \infty$, are tight in $\mc{D}_\infty$. For the next lemma, recall the definitions of $\AnL, \AnR$ in
\eqref{eq:defn:AnL} and \eqref{eq:defn:AnR}.

\begin{lemma}
  \label{lem:AnToZero}
  As $n \to \infty$,
  \begin{align}
    \label{eq:local:AnToZero}
    |\AnL| \to 0  \; \mbox{ and } \;    |\AnR| \to 0,
    \mbox{ almost surely.}
  \end{align}
\end{lemma}
\begin{proof}
  Because $\snL$ is strictly less than $m$, we can apply
  Corollary~\ref{cor:EstimatedFNearlyTouchesEDFatKnotsUnconstrained} B,
  so that $|\AnL| = \xn^{3/5} \left| \FnL(\snL)-\FFnaL(\snL) \right| \le \xn^{-2/5} \to 0$ almost surely.  Similarly, since $\FnL(X_{(\xn)})=1=\Fn(X_{(\xn)})$, by the same corollary, $|\AnR| \to 0$ almost surely.
\end{proof}

With Lemmas~\ref{lem:Y-bigger-H-vv}, \ref{lem:MC-dataprocs-vv-limit}, and \ref{lem:local:Htight} in hand, we can now finish the proof of the theorem.  Fix a subsequence $n'$. Let
\begin{align*}
  Z_{n,L} & = \lp (\HHnvvaL)^{(3)}, (\HHnvvaL)^{(2)}, (\HHnvvaL)^{(1)}, \HHnvvaL,
  \XXnvvaL,  %
  \YYnvvaL \rp, \\
  Z_{n,R} & = \lp (\HHnvvaR)^{(3)}, (\HHnvvaR)^{(2)}, (\HHnvvaR)^{(1)}, \HHnvvaR,
  \XXnvvaR, %
  \YYnvvaR  \rp,
\end{align*}
By Lemmas~\ref{lem:MC-dataprocs-vv-limit} and \ref{lem:local:Htight},    $Z_{n,R}$ and $Z_{n,L}$ are both tight in the space
$E_c \equiv \mc{D}_c \times \mc{C}_c^3 \times \mc{D}_c \times \mc{C}_c$ with $0 < c < \infty$.
This means they are also tight in $E_\infty$,
by the discussion after Lemma~\ref{lem:local:Htight}.
Thus there exists a subsubsequence $n''$ such that $Z_{n'',R}$ and $Z_{n'',L}$ converge weakly.  By the Skorokhod construction (see e.g., Chapter 14 of \cite{MR1762415}), %
we may assume that the convergence is almost sure (a.s.).  Let
 $(Z_{0,L}, Z_{0,R})$ be the limit and let
$Z_{0,L}=(H_L^{(3)} , H_L^{(2)}, H_L^{(1)}, H_L,  %
X_L, %
Y_L)$,
and
$Z_{0,R}=(H_R^{(3)} , H_R^{(2)}, H_R^{(1)}, H_R,  %
X_R, %
Y_R)$.
Note that $(\HHnvvaR)^{(2)} = (\HHnvvaL)^{(2)}$ and this function is of course concave with mode at $0$, so  $H_R^{(2)} = H_L^{(2)}$ must also be concave with mode at $0$.
Let
\begin{equation}
  \label{eq:def:tauL-tauR}
  \TauL = \sup \lp \Sa( H^{(2)}_R) \cap (-\infty, 0) \rp
  \, \mbox{ and } \,
  \TauR = \inf \lp  \Sa(H^{(2)}_R) \cap (0  , \infty) \rp
\end{equation}
with $\Sa(H^{(2)}_R)$ defined as in \eqref{eq:defnSa}.  There must be a sequence of knots $\tau_{n'',R} \in ( S_n(\vvna)  \cap (m, \infty))$ such that $n^{1/5}(\tau_{n'', R} - m) \to \TauR$ a.s.  To see that we can take $\tau_{n'',R}$ strictly greater than $m$, by \eqref{eq:def:tauL-tauR}
we see the only way $\TauR=0$ is if there exists a sequence of points of $\Sa( H^{(2)}_R)$ strictly greater than $0$ and converging to $0$.   Similarly there is a sequence $\tau_{n'', L} \in ( S_n(\vvna) \cap (-\infty, m))$ such that $n^{1/5}( \tau_{n'', L} \to \tau_L$ a.s.  In our definitions of the ($f$- and $\vp$-) processes, $s_{n,R}$ was any knot strictly greater than $m$ satisfying $n^{1/5}(s_{n,R} - m) = O_p(1)$, and analogously for $s_{n,L}$.  Take $s_{n'',R} = \tau_{n'',R}$ and $s_{n'',L} = \tau_{n'',L}$.
Then let
\begin{align*}
  s_{n''}
  & = (n'')^{1/5} \lp \tau_{n'',L} -m, \tau_{n'',R}-m, \tau^0_{n'',-} - m, \tau^0_{n'',+} - m \rp,
\end{align*}
let $s_{n''} \to (\tau_L, \tau_R, \tau_-, \tau_+)$
and let $Z_0 = (Z_{0,L}, Z_{0,R}, \tau_L, \tau_R, \tau_-, \tau_+)$, where $\tau_-$ and $\tau_+$ are the limits of the corresponding terms again by tightness from Proposition~\ref{prop:uniform-local-tightnessKnots}.
By Lemma~\ref{lem:MC-dataprocs-vv-limit},
if we let $Y \equiv Y_{a, \sigma}$
as defined in \eqref{eq:defn:Y-a-sigma}
with $a = |\vvo^{(2)}(m)|/ 4!$ and $\sigma = 1/ \sqrt{\ffo(m)}$,
then
\begin{align*}
  Y_R(b) & = \int_{\tau_R}^b \int_{\tau_R}^v dY'(v) dv,
  &
  X_R(b) & = Y_R'(b) = \int_{\tau_R}^b  dY', \\
  Y_L(b) & = \int_{b}^{\tau_L} \int_v^{\tau_L} dY'(v) dv,
  &
  X_L(b) & = -Y_L'(b)  = \int_b^{\tau_L} dY'.
\end{align*}
Let $\phi : E_c^2\times \RR^4 \to \RR$ be defined by
\begin{equation*}
  \phi(z_1, \ldots, z_{16}) = \int_{z_{13}}^{z_{14}} z_8 d\lambda - dz_{11}
\end{equation*}
giving
\begin{align*}
  \phi(Z_{n})
  & = \int_{n^{1/5}(\tau_{n,L}-m)}^{ n^{1/5}(\tau_{n,R}-m)}
  (\HHnvvaR)^{(2)} d\lambda - d (\YYnvvaR)^{(1)}\\
  &=
  (\HHnvvaR- \YYnvvaR)'( n^{1/5}(\tau_{n,R}-m))
  - (\HHnvvaR- \YYnvvaR)'(n^{1/5}(\tau_{n,L}-m)) .
\end{align*}
By Lemma~\ref{lem:FequalsX_M} below, $\phi(Z_{n''}) \to 0$ a.s., and $Z_{n''} \to Z_0$ a.s., and $\phi$ is continuous at $z$ such that $z_8$ and $z_{11}$ are continuous functions, so a.s.\ $\phi(Z_0) = 0$, i.e.\
\begin{equation*}
  \int_{\tau_{L}}^{\tau_{R}} H_{R}^{(2)}d\lambda - d(Y_{R}^{(1)}) = 0
  \quad   \mbox{ a.s.},
\end{equation*}
so condition \eqref{eq:CharznFequalsX_M_full}  of Theorem~\ref{thm:MC-process-uniqueness-theorem} holds.
Now, let $\phi_c : E_c^2 \times \RR^4 \to \RR$ be defined by $ \phi_c(z) = \inf_{b \in [0,c]} (z_4(b) - z_6(b)) \wedge 0$ giving $ \phi_c(Z_n) = \inf_{b \in [0,c]} (\HHnvvaL(b) - \YYnvvaL(b)) \wedge 0$.  This $\phi_c$ is continuous since $z_4$ and $z_6$ are continuous, and $\phi_c(Z_n) = 0$ a.s.\ by
Lemma~\ref{lem:Y-bigger-H-vv}, so $\phi_c(Z_0) = 0$ a.s., for all $c$.  An analogous argument holds for the right-side processes.  Thus,
\begin{align*}
  (H_L - Y_L)(b) & \le 0 \quad \mbox{ for } \ b \le 0, \\
  (H_R - Y_R)(b) & \le 0 \quad \mbox{ for } \ b \ge 0,
\end{align*}
so \eqref{eq:CharznInequalityL2_full} and \eqref{eq:CharznInequalityR2_full} hold.
Now, let $\phi_{R,c}(z) = \int \one_{[z_{16},c]} (z_{10} - z_{12}) dz_7$
(where $z_{16}$ corresponds to $\tau_+$, $z_{10}$ to $H_R$, $z_{12}$ to $Y_R$, and $z_7$ to $H_R^{(3)}$).
Let
\begin{equation*}
  \phi_R(z) = \int_{[\tau_+^0(z_2), \infty)} (z_{10}- z_{12}) dz_7,
\end{equation*}
where $\tau_+^0(z_2)$ is defined as in \eqref{eq:defn:tau0+}
(and thus $\Sa(z_2)$ is defined as in \eqref{eq:defnSa}).
We want to show $\phi_R(Z_0) = 0 $ a.s.\
Note that $\tau_+ \le \tau_+^0(Z_{0,2})$, although a priori we may not have equality. This is because a linear function may be well approximated by a nonlinear function, but the reverse is not true.  Thus $\vvna$ could potentially have knots strictly between the limit knot $\tau_+^0(Z_{0,2})$ and $m$ (on an $n^{-1/5}$ scale) so $\tau_+$ could be smaller than $\tau_+^0(Z_{0,2})$, but $\vvna$ must have knots approaching $\tau_+^0(Z_{0,2})$ (on an $n^{-1/5}$ scale), so $\tau_+$ cannot be larger than $\tau_+^0(Z_{0,2})$.
By Lemma~\ref{lem:Y-bigger-H-vv}, $\phi_{R,c}(Z_n)= 0$ a.s., and
 by Lemma~\ref{lem:charzn-equality-phi-cts} below we can conclude $\phi_{R,c}(Z_0)=0$ a.s.
Now, let $c \to \infty$ to see
\begin{equation*}
  \int \one_{[Z_{0,16}, \infty)} ( Z_{0, 10} - Z_{0, 12}) dZ_{0,7} = 0
\end{equation*}
and since the integrand is nonpositive and the integrating measure is nonpositive, this implies
\begin{equation*}
  \int \one_{[ \tau_+^0(Z_{0,2}), \infty)} (Z_{0, 10} - Z_{0,12} ) dZ_{0,7} = 0
\end{equation*}
as desired.
An analogous argument holds for the functional
$$\phi_L(z) = \int_{(-\infty, \tau_-^0(z_{2})]} (z_4-z_6) dz_1$$
with $\tau_-^0(z_{2})$ defined as
in \eqref{eq:defn:tau0-}.
Thus we can a.s. conclude $\phi_R(Z_0) = 0$ and $\phi_L(Z_0)=0$,
so we have shown
condition \eqref{eq:CharznEquality2_full} of
Theorem~\ref{thm:MC-process-uniqueness-theorem} holds.
We have shown that as $n'' \to \infty$, $Z_{n'', 0}$ converge a.s., so weakly, to $Z_0$ which satisfies the uniqueness criteria of
 Theorem~\ref{thm:MC-process-uniqueness-theorem}.  Thus we conclude that the limit does not depend on the choice of subsequence, and so can conclude
$Z_{n,L} \Rightarrow Z_{0,L}$ and $Z_{n,R} \Rightarrow Z_{0,R}$  both in $E_\infty$, as desired.
This ends the proof of Theorem~\ref{thm:MC-MLE-limit-A-process-version}.\label{end-proof-processes}
\end{proof}

Here are the two remaining lemmas we used in the proof of
Theorem~\ref{thm:MC-MLE-limit-A-process-version}.
\begin{lemma}
  \label{lem:FequalsX_M}
  Let $\nu_{n,R} = n^{1/5}(\snR-m)$ and $\nu_{n,L} = n^{1/5}(\snL-m)$. Then almost surely,
  recalling the notation $g(a,b] = g(b)-g(a)$, we have
  \begin{equation*}
    \lv \lp \HHnvvaR - \YYnvvaR \rp^{\prime}(\nu_{n,L}, \nu_{n,R}] \rv
    \le \frac{4}{\ffo(m)} n^{-2/5},
  \end{equation*}
\end{lemma}
\begin{proof}
    We have
  \begin{align*}
    (\HHnffaR -  \YYnffaR)'(b) = \xn^{3/5} \int_{\snR}^{\tnb} (\ffna(u)du - d\Fn(u)) + \AnR
  \end{align*}
  so that
  \begin{equation*}
    \begin{split}
      \MoveEqLeft
      (\HHnffaR - \YYnffaR)'(\nu_{n,R})
      - (\HHnffaR-\YYnffaR)'(\nu_{n,L})
      \\
      & = \xn^{3/5} \int_{\snL}^{\snR} ( \ffna(u)du - d\Fn(u) ).
    \end{split}
  \end{equation*}
  Thus,
  \begin{equation*}
    \left(\HHnvvaR -\YYnvvaR \right)' ( \nu_{n,L}, \nu_{n,R} ]
    = \frac{n^{3/5}}{\ffo(\mm)}  \int_{\snL}^{\snR} ( \ffna(u)du - d\Fn(u) )
  \end{equation*}
  which is bounded in absolute value by $2 n^{-2/5} / \ffo(m)$ almost surely, by applying
  Corollary~\ref{cor:EstimatedFNearlyTouchesEDFatKnotsUnconstrained}
  (recall, by definition, $\snR$ and $\snL$ are not equal to $\mm$).
\end{proof}

In the definition of $\phi_{R,c}$ in the lemma below, $z_{16}$ corresponds to $\tau_+$,
$z_{10}$ to $H_R$,
$z_{12}$ to $Y_R$,  and $z_7$ to $H_R^{(3)}$.  The variables defining $\phi_{L,c}$ are the analogous left-side terms.
\begin{lemma}
  \label{lem:charzn-equality-phi-cts}
    For $n \ge 0$,  assume $z_n \in E_c^2 \times \RR^4$ is such that $z_n $ converges to $ z_0$.  Assume $z_{0,1}$ and $z_{0,7}$ are nonincreasing.  Let $\phi_{R,c}(z) = \int \one_{[z_{16},c]} (z_{10} - z_{12}) dz_7$ and $\phi_{L,c}(z) = \int \one_{[-c, z_{15}]} (z_4 - z_6 ) dz_1$.  Assume further that $\phi_{R,c}(z_n) = 0$, $\phi_{L,c}(z_n)=0$,  $z_{n,10} - z_{n,12} \le 0$ and $z_{n,4} - z_{n,6} \le 0$.  Then for  $c > 0$, $\phi_{R,c}(z_0) = 0$ if
$c$ is not a discontinuity point of $z_{0,7}$
and $\phi_{L,c}(z_0)=0$
if $-c$ is not a discontinuity point of $z_{0,1}$.
\end{lemma}
\begin{proof}
  We consider $\phi_{R,c}$, the proof for $\phi_{L,c}$ is analogous.
For convenience, let $h_c(z) = \one_{[z_{16},c]} (z_{10}-z_{12})$, and then
$  \int h_c(z_n) dz_{n,7} - \int h_c(z_0) dz_{0,7}$ equals
\begin{align}
\label{eq:14}
   \int \lp h_c(z_n) - h_c(z_0) \rp dz_{n,7} - \int h_c(z_0) d(z_{0,7} - z_{n,7}).
\end{align}
The integrand of the first term in \eqref{eq:14} is uniformly converging to $0$.  If we let $x_1$ and $x_2$ be any fixed continuity points of $z_0$ satisfying  $x_1 \le z_{0,22} \le c \le x_2$ then, for large enough $n$ the  measure $z_{n,7}$ has total mass bounded by $z_{0,7}(x_1) - z_{0,7}(x_2) + 2 < \infty$.  Thus the first term in \eqref{eq:14} converges to $0$ as $n \to \infty$ since $dz_{n,7}$ is a nonpositive measure %
(\cite{Royden}, Chapter 11.5).
    Now since $z_{n,5}$ converges to $z_{0,5}$ in the $M_{1}$ topology, $z_{n,5}$ converges weakly to $z_5$, in the sense that for all $t \in (-c,c)$ that are continuity points of $z_{0,5}$, $z_{n,5}(t) \to z_{0,5}(t)$ as $n \to \infty$ (Lemma~12.5.1, %
\cite{Whitt2002Stochastic}).   The integrand of the second term in \eqref{eq:14}
is uniformly bounded and has discontinuity points at $z_{0,22}$ and $c$, so
$\int \one_{(z_{0,16},c)} h_c(z_0) d(z_{0,7} - z_{n,7})$ converges to $0$.
 By assumption $c$ is not a discontinuity point of $z_{0,7}$ so $\int \one_{\lb c\rb} h_c(z_0) d(z_{0,7}-z_{n,7})$ goes to $0$ as $n \to \infty$.  If $z_{0,16}$ is not a discontinuity point of $z_{0,7}$ then similarly $\int \one_{\lb z_{0,16} \rb} h_c(z_0) d(z_{0,7}-z_{n,7})$ goes to $0$.   Thus, assume $z_{0,16}$ is a discontinuity point of $z_{0,7}$.  Then by the $M_1$ convergence of $z_{n,7}$ to $z_{0,7}$, there exists a sequence $x_n$ of discontinuity points of $z_{n,7}$ such that $x_n \to z_{0,16}$.  By the assumption that $\phi_{R,c}(z_n) = 0$, we know $(z_{n,10}-z_{n,12}(x_n) = 0$ (since $z_{n,10} - z_{n,12} \le 0$), and by uniform convergence, we see that $(z_{0,10}-z_{0,12})(z_{0,16}) = 0$.  Thus
$\int \one_{\lb z_{0,16} \rb} h_c(z_0) d(z_{0,7}-z_{n,7})$ converges to $0$, and so
\eqref{eq:14} converges to $0$ as $n \to \infty$, so we are done.
\end{proof}

\bigskip

\par\noindent
Next we prove Theorem~\ref{thm:process-asymptotics-symmetry}.
\begin{proof}[Proof of Theorem~\ref{thm:process-asymptotics-symmetry}]
  We argue by considering $|X_1|, \ldots, |X_n| \stackrel{iid}{\sim}
  2 F_0 - 1$ on $[0,\infty)$ with density $g_0^+ = 2 f_0 \one_{[0,\infty)}$.  Note that $g_0^+$ is log-concave with mode known to be at $0$ and so we can use the results based on the mode-constrained MLE.  Let $\psi_0^+ = \log g^+_0 = \vp_0 + \log 2$ on $[0,\infty)$.  Recall that
  \begin{equation*}
    \FnsR(x) = n^{-1} \sum_{i=1}^n \one_{ \lb |X_i| \ge x \rb}
    \quad \text{ and } \quad
    \GGnsR(x) = \int_x^{|X|_{(n)}} \ggns(u)du
  \end{equation*}
  where $\ggns = 2 \widehat{g}_n^0 \one_{[0,\infty)}$.  Let $\psins = \log \ggns$.  Let $s_{n,R} > 0$ be a knot %
  of $\psins$ satisfying $n^{1/5} s_{n,R} = O_p(1)$.  The proof proceeds as in the proof of Theorem~\ref{thm:MC-MLE-limit-A-process-version}, except we replace $\FFnaR$ by $\GGnsR$, $\FnR$ by $\FnsR$, and $f_0$ by $g_0^+$, and we consider only ``right-side'' processes on $[0,\infty)$.  To make things line up fully, we can take $s_{n,L}$ to be $0$ and define left side analogs of $\GGnsR$ and $\FnsR$ to be identically $0$.
\begin{mylongform} %
  \begin{longform}
    In the mode-constrained proof, $s_{n,L}$ must be strictly less than $0$.  However in the present case because all left-side processes are identically 0, we may take $s_{n,L}$ to be $0$.  What is needed is that $H_{n,L} = Y_{n,L}$ and $|H'_{n,L} - Y'_{n,L}| \le 1/n$ at $s_{n,L}$, which is immediately true.
  \end{longform}
\end{mylongform}
\begin{mylongform}
  \begin{longform}
    Let $b \ge 0$ and $t = t_{n,b} \equiv bn^{-1/5}$.

    Thus (in parallel with \eqref{eq:local:HHnvvaR2HHnffaR}) we let
    \begin{equation*}
      \HHnpsisR(b) = n^{4/5} \int_{s_{n,R}}^{t_{n,b}} \int_{s_{n,R}}^v
      (\psins(u)- \psi_0^+(0))  du dv +
      \frac{A_{n,R}^+ n^{1/5}(t_{n,b}-s_{n,R})}{ g_0^+(0)}
    \end{equation*}
    where $|A_{n,R}^+| \to 0$ almost surely,
    by Lemma~\ref{lem:AnToZero}, and we let (in parallel with \eqref{eq:18})
    \begin{equation*}
      \YYnpsisR(b) =
      g_0^+(0)^{-1} n^{4/5} \int_{s_{n,R}}^{t_{n,b}} \int^{v}_{s_{n,R}}
      (d \GG_n - g_0^+(0) d\lambda)dv %
      - E^+(b)
    \end{equation*}
    where $E^+(b) = o_p(1)$ uniformly for $b \in [-cn^{-1/5}, cn^{1/5}]$, $c > 0$
    (by Lemma~\ref{lem:local:analyzeRemainder}).
    Then by \eqref{eq:local:YRbiggerHR-v} and
    \eqref{eq:local:YR-HRv-equality} we have that $\YYnpsisR(b) - \HHnpsisR(b) \ge 0$ for $b \ge 0$, and
    $\int_{\tau^+_{n,+}}^\infty \lp \YYnpsisR(b) - \HHnpsisR(b) \rp d(\HHnpsisR)^{(3)}(b) = 0$,
    where $\tau^+_{n,+}$ is the smallest right-knot of $\psins$.
    By arguing along subsequences,
    as in the proof of Theorem~\ref{thm:MC-MLE-limit-A-process-version},
    we can take $\HHnpsisR$ and its first three derivatives to converge in distribution on $[0,c]$ to a process $\widehat{H}^+_{a,\sigma,R}$ and its first three derivatives.  We can also then assume $n^{1/5} s_{n,R} \to_d \tau_R$ where $\tau_R$ is the infimum of the strictly positive knots of $\widehat{H}^+_{a,\sigma,R}$.
    Let $Y_{a,\sigma,R}(b) = \int_{\tau_R}^b \int_{\tau_R}^v dY_{a,\sigma}'(u)dv$.
    By Lemma~\ref{lem:MC-dataprocs-vv-limit} we see that $((\YYnpsisR)', \YYnpsisR)$ converges in distribution to $( (Y_{a,\sigma,R})', Y_{a,\sigma,R})$ on $[0,c]$.

    Now, we need to see that the characterization given by Theorem~\ref{thm:symm-process-uniqueness-theorem} holds for $\widehat{H}^+_{a,\sigma,R}$.
    This can be checked in fashion similar to what is done at the end of the proof of
    Theorem~\ref{thm:MC-MLE-limit-A-process-version}.  Thus the proof is complete.
  \end{longform}
\end{mylongform}
\end{proof}
See
the scaling relations in
\eqref{GammaDefnsLogConcaveAtMode-1}
and \eqref{GammaDefnsLogConcaveAtMode}, and the three following displays,
to see how
Theorem~\ref{thm:process-asymptotics-symmetry} proves
Theorem~\ref{thm:MC-MLE-limit}~\ref{thm:MC-MLE-limit:item-B}.
Note that
in Theorem~\ref{thm:process-asymptotics-symmetry}, $\sigma = 1/ \sqrt{2 f_0(0)}$,
whereas in
 \eqref{GammaRelationsPart2} and \eqref{ScalingRelation-MC-2},
$\sigma$ is $1 / \sqrt{f_0(0)}$.  This modification yields  the factors of $2^{-2/5}$ and $2^{-1/5}$ appearing (twice each) on the right side of
\eqref{thm:mode-ff-asymptotics}.

\subsubsection{Proof for the maximum functional limit theory}

\begin{proof}[Proof of Theorem~\ref{thm:max-functional}]
  If  $\lb F_n \rb$ is a concave function sequence on $\RR$ that  converges uniformly on compacta to a function $F$, where $F$ has a unique maximizer, then $N(F_n) \to N(F)$.
  Thus $N$ is continuous on the subset of convex functions in ${\cal C}_\infty$ with the topology of uniform convergence on compacta.
Let $\sigma = 1/\sqrt{f_0(m)}$,  $a = | \vvo^{(2)}(m)| / 4!$,
 $F_n(t) = n^{2/5}( \vvn(m + n^{-1/5}t) - \vvo(m))$, and $F(t) = \widehat{\vp}_{a,\sigma}(t)$. Then  by Theorem~\ref{thm:MC-MLE-limit-A-process-version}, $F_n$ converges weakly to $F$ so by the continuous mapping theorem $N(F_n) \to_d N(F)$.  Now, by the scaling relationship
\eqref{eq:H2-ScalingRelationUnConstrained-1}, %
  $N(F) =_d \sigma^{4/5} a^{1/5} N( \widehat{\vp}( (a/\sigma)^{2/5} \cdot) ) = \sigma^{4/5} a^{1/5} N( \widehat{\vp})$.
  We can check that
  $  \sigma^{4/5} a^{1/5} = C(m, \vvo)$.
  Thus we have shown
  \begin{equation*}
    n^{2/5}(\log N(\ffn) - \log N(f_0)) \to_d
    C(m, \vvo) N( \widehat{\vp})
  \end{equation*}
  since $N(F_n)$ is the left side of the above display.
  Applying the delta rule, we see also
  \begin{equation*}
    n^{2/5} ( N(\ffn) - N(f_0)) \to_d c(m, \vvo) N(\widehat{\vp}).
  \end{equation*}
\end{proof}

\section{Technical lemmas}
\label{sec:technical-lemmas}

Here is a statement of the general integration by parts formulas for
functions of bounded variation, used in our proof of
Proposition~\ref{pro:charznEqualityRewrite_full}.  See, e.g., page 102 of
\cite{Folland1999RealAnalysis} for the definition of bounded variation.
\begin{lemma}[\cite{Folland1999RealAnalysis}]
  \label{lem:integration-by-parts}
  Assume that $F$ and $G$ are of bounded variation on a  set $[a,b]$ where
  $-\infty < a < b < \infty$

  \begin{enumerate}[label=\Alph*.,ref=\Alph*]
  \item   \label{lem:integ-by-parts-open-interval} If at least one of $F$ and $G$ is continuous, then
    \begin{equation*}
      \int_{(a,b]} FdG + \int_{(a,b]} GdF
      = F(b)G(b) - F(a)G(a).
    \end{equation*}
  \item  \label{lem:integ-by-parts-closed-interval} If there are no points in $[a,b]$ where $F$ and $G$ are both discontinuous,
    then
    \begin{equation*}
      \int_{[a,b]} F dG + \int_{[a,b]} G dF
      = F(b)G(b) - F(a-)G(a-).
    \end{equation*}
  \end{enumerate}
\end{lemma}

The next lemma
is proved in \cite{Doss:2013}, page 143, for convex rather than concave
functions.
\begin{lemma}[\cite{Doss:2013}]
  \label{fact:unq:convexFunDiffs}
  Let $g_1$ and $g_0$ be concave functions on $[a,b]$, and let $t \in
  [a,b]$.  Then
  \begin{equation}
    \label{eq:unq:convexFunDiffs}
    \begin{split}
      \MoveEqLeft g_0(t) - g_1(t) \le \frac{ b-t}{b-a}(g_0(a) - g_1(a))
      + \frac{ t-a}{b-a}(g_0(b)-g_1(b)) \\
      & \qquad \qquad + \frac{(b-t)(t-a)}{b-a}(g_0'(a+) - g_0'(b-)).
    \end{split}
  \end{equation}
\end{lemma}

\section*{Acknowledgements}
Both authors owe thanks to Tilmann Gneiting for support during visits to the Applied Mathematics Institute at the
University of Heidelberg in 2011-2012.  We also owe thanks to Lutz D\"umbgen for several helpful conversations.
We  are grateful to an Associate Editor and two referees for a very careful reading and for helpful comments.


%
\bibliographystyle{imsart-nameyear}
%
\bibliography{p}

\end{document}